\documentclass[11pt,reqno]{amsart}
\usepackage{amsmath}
\usepackage{array}
\newtheorem{theorem}{Theorem}[section]
\newtheorem{proposition}[theorem]{Proposition}
\newtheorem{lemma}[theorem]{Lemma}
\newtheorem{corollary}[theorem]{Corollary}
\theoremstyle{definition}
\newtheorem{definition}[theorem]{Definition}
\newtheorem{example}[theorem]{Example}

\theoremstyle{remark}
\newtheorem{remark}[theorem]{Remark}
\input{amssym.def}
\input{amssym}
\def\L{{\mathcal L}}
\def\N{{\mathcal N}}
\def\B{{\mathcal B}}
\def\D{{\mathcal D}}
\def\H{{\mathcal H}}
\def\P{{\mathcal P}}
\def\M{{\mathcal M}}
\def\V{{\mathcal V}}
\def\W{{\mathcal W}}
\def\G{{\mathcal G}}  
\def\F{{\mathcal F}} 
\def\F{{\mathcal F}} 
 
\def\K{{\mathcal K}} 
 
\def\J{{\mathcal J}}

\def\R{\Bbb R} 
\def\Z{\Bbb Z} 
\def\Na{\Bbb N}
\def\T{\Bbb T}  
\def\C{\Bbb C}
\def\Q{\Bbb Q}
\def\He{\Bbb H}
\def\A{{\mathcal A}}

\def\Sc{Schr\"o\-din\-ger}
\def\la{\langle}
\def\be{\begin{equation}}
\def\ee{\end{equation}}
\def\bea{\begin{eqnarray}}
\def\eea{\end{eqnarray}}
\def\ra{\rangle}
\def\ds{\displaystyle}
\def\om{\omega}
\def\Om{\Omega}
\def\ep{\varepsilon}
\def\RR{\mathcal R}

\newcommand{\truc}{\widetilde}

\begin{document}
\baselineskip=18pt
\title{Convergence of a quantum normal form and an exact quantization for\-mu\-la}
\author{Sandro Graffi}
\address {Dipartimento di Matematica, Universit\`{a} di Bologna, 40127 Bologna, Italy}
\email{graffi@dm.unibo.it}
\author{Thierry Paul}
\address{Centre de Math\'ematiques Laurent Schwartz, 
\'Ecole Polytechnique, 91128 Palaiseau Cedex, Fran\-ce}
\email{paul@math.polytechnique.fr}

\date{}

\begin{abstract}
{Let the quantization of the linear flow of diophantine
frequencies $\om$ over the torus $\T^l$, $l>1$, namely the Schr\"odinger operator
 $-i\hbar\omega\cdot\nabla$ on $L^2(\T^l)$, be
perturbed by the quantization of a function $\V_\om: \R^l\times\T^l\to\R$ of the form
\vskip 5pt\noindent
$$
\V_\om(\xi,x)=\V(z\circ \L_\om(\xi),x),\quad \L_\om(\xi):= \om_1\xi_1+\ldots+\om_l\xi_l
$$
\vskip 4pt\noindent
where  $z\mapsto \V(z,x): \R\times\T^l \to\R$ is real-holomorphic.
 We prove that the corresponding quantum normal form converges uniformly with respect to $\hbar\in [0,1]$. 
Since the quantum normal form reduces to the classical one  
for $\hbar=0$, this result simultaneously yields an exact
 quantization formula for the quantum spectrum, 
as well as a convergence criterion for the Birkhoff normal form, valid  
for a class of perturbations holomorphic away from the origin.  The main technical aspect 
concerns  the quantum homological equation $\ds 
{[F(-i\hbar\om\cdot\nabla),W]}/{i\hbar}+V=N$,  
$F:\R\to\R$ being  a smooth function $\ep-$close to the identity. Its  solution is 
constructed, 
and estimated uniformly with respect to $\hbar\in [0,1]$,  by solving 
the  equation $\{F(\L_\om),\W\}_M+\V=\N$ for the corresponding symbols. 
Here $\{\cdot,\cdot\}_M$ stands for the Moyal bracket.  
As a consequence, the KAM iteration for the symbols of the quantum operators can be implemented, 
and its convergence proved, uniformly with respect to 
$(\xi,\hbar,\ep)\in \R^l\times [0,1]\times \{\ep\in\C\,|\;|\ep|<\ep^\ast\}$, where  $\ep^\ast>0$ is explicitly estimated in terms only of the diophantine constants. This in turn entails the uniform 
convergence of the quantum normal form. }
\end{abstract}
\maketitle 

%
\tableofcontents

\section{Introduction}
\renewcommand{\thetheorem}{\thesection.\arabic{theorem}}
\renewcommand{\theproposition}{\thesection.\arabic{proposition}}
\renewcommand{\thelemma}{\thesection.\arabic{lemma}}
\renewcommand{\thedefinition}{\thesection.\arabic{definition}}
\renewcommand{\thecorollary}{\thesection.\arabic{corollary}}
\renewcommand{\theequation}{\thesection.\arabic{equation}}
\renewcommand{\theremark}{\thesection.\arabic{remark}}
\setcounter{equation}{0}%
\setcounter{theorem}{0}%
\noindent
\subsection{Quantization formulae}
The establishment of a quantization formula (QF) for the eigenvalues of the \Sc\ operators is a  classical mathematical problem of quantum mechanics (see e.g.\cite{FM}). 
To review the notion of QF, consider first  a semiclassical pseudodifferential operator $H$ (for this notion, see e.g.\cite{Ro}) acting on $L^2(\R^l)$, $l\geq 1$, of order $m$,  self-adjoint with pure-point spectrum, with (Weyl) symbol $\sigma_H(\xi,x)\in C^\infty(\R^l\times\R^l;\R)$.  
 \begin{definition}\label{quant}
 {\it We say that $H$ admits an $M$-smooth {\rm exact} QF, $M\geq 2$,
 if  there exists a function $\mu:$ $(A,\hbar)\mapsto \mu(A,\hbar)\in C^M(\R^l\times [0,1]; \R)$  such that:
 \begin{enumerate}
 \item 
 $\mu(A,\hbar)$ admits an asymptotic expansion up to order $M$ in $\hbar$ uniformly on compacts with respect to $A\in\R^l$;
 \item
 $\forall\hbar\in]0,1]$, there is a sequence $n_k:=(n_{k_1},\ldots,n_{k_l})\subset \Z^l$ such that all eigenvalues  $\lambda_{k}(\hbar)$ of $H$ admit the representation:
 \be
 \label{FQ1}
\lambda_{k}(\hbar)=\mu(n_k\hbar,\hbar).
 \ee
 \end{enumerate}}
 \end{definition}
 \begin{remark} (Link with the Maslov index)
 \label{maslov}
Consider any function $f:\ \R^l\to\R^l$ with the property  
 $\la f(A),\nabla\mu(A,0)\ra $ $=\partial_\hbar\mu(A,0)$. Then we can rewrite the asymptotic expansion of $\mu$ at second order as :
 \be
 \mu(n_k\hbar,\hbar)=\mu(n_k\hbar+\hbar f(n_k\hbar))+O(\hbar^2).
 \ee
  When  $f(m\hbar)=\nu, \;\nu\in\Q^l$, the Maslov index \cite{Ma} is recovered.
Moreover, when 
 \be
 \label{QF2}
 |\lambda_{k}(\hbar)-\mu(n_k\hbar,\hbar)|=O(\hbar^M), \quad \hbar\to 0, \quad M\geq 2
 \ee
 then we speak of {\it approximate} QF of order $M$.   
\end{remark}
\begin{example}
 (Bohr-Som\-mer\-feld-Ein\-stein for\-mu\-la). Let  $\sigma_H$ fulfill the conditions of the Liouville-Arnold theorem (see e.g.\cite{Ar1}, \S 50). Denote  $A=(A_1,\ldots,A_l)\in \R^l$ the action variables, and $E(A_1,\ldots,A_l)$ the symbol $\sigma_H$ expressed as a function of the action variables.   Then   the Bohr-Som\-mer\-feld-Ein\-stein for\-mu\-la (BSE) QF is
 \be
 \label{QF3}
 \lambda_{n,\hbar}=E((n_1+\nu/4)\hbar,\ldots,(n_l+\nu/4)\hbar)+O(\hbar^2)
 \ee
 where $\nu=\nu(l)\in\Na\cup\{0\}$ is the Maslov index \cite{Ma}. When $H$ is the \Sc\ operator, and $\sigma_H$ the corresponding classical Hamiltonian, (\ref{QF3}) yields the approximate eigenvalues, i.e. the approximate quantum energy levels.  In the particular case of a quadratic, positive definite Hamiltonian, which can always be reduced to the harmonic oscillator with frequencies $\om_1>0,\ldots,\om_l>0$,  the BSE is an exact quantization formula in the sense of Definition 1.1 with $\nu=2$, namely:
 $$
 \mu(A,\hbar)=E(A_1+\hbar/2,\ldots,A_l+\hbar/2) =\sum_{k=1}^l\om_k(A_k+\hbar/2)
  $$
  \end{example}
  \vskip 10pt
To our knowledge, if $l>1$ the only known examples of  exact QF in the sense of Definition 1.1 correspond to classical systems integrable by separation of variables, such
 that each separated system admits in turn an exact QF, as in the case of the Coulomb potential (for exact QFs for general one-dimensional \Sc\ operators see \cite{Vo}).  For general integrable systems, only the approximate BSE formula is valid. Non-integrable systems admit a formal approximate QF,  the so-called Einstein-Brillouin-Keller (EBK), recalled below, provided they possess a normal form to all orders.  

In this paper we consider a  perturbation of a linear Hamiltonian on  $T^\ast\T^l=\R^l\times\T^l$,  and prove that the corresponding quantized operator can be unitarily
conjugated to a function of the differentiation operators  via the construction of a quantum normal form which converges uniformly with respect to $\hbar\in [0,1]$. This yields immediately an exact, $\infty$-smooth QF. The uniformity with respect to $\hbar$ yields also an explicit family of classical Hamiltonians admitting a convergent normal form, thus making the system integrable.  
\subsection{Statement of the results}
 Consider the Hamiltonian family $\H_\ep: \R^l\times \T^l\rightarrow \R, (\xi,x)\mapsto \H_\ep(\xi,x)$,   indexed by $\ep\in\R$, defined as follows:
 \be
\H_\ep(\xi,x):=\L_\om(\xi)+\ep \V(x,\xi);\quad \L_\om(\xi):=\la\om,\xi\ra, \quad\om\in\R^l,\quad  \V\in C^\infty(\R^l\times\T^l;\R).  
\ee
Here $\xi\in\R^l, x\in\T^l$ are canonical coordinates on the phase space $\R^l\times\T^l$, the $2l-$cylinder.  $\L_\om(\xi)$ 
generates the linear Hamiltonian flow $\xi_i\mapsto \xi_i, x_i\mapsto x_i+\om_it$ on $\R^l\times\T^l$.  For $l>1$ the dependence of 
$\V$ on $\xi$ makes non-trivial  the integrability of the flow of $\H_\ep$ when $\ep\neq 0$, provided the  
{\it frequencies} $\om:=(\om_1,\ldots, \om_l)$ are independent  over $\Q$ and fulfill a diophantine condition such as (\ref{DC}) below. Under this assumption it is well known that
 $\H_\ep$ admits a  {\it normal form} at any order (for this notion, see e.g. \cite{Ar2}, \cite{SM}). Namely, $\forall\,N\in\Na$ 
a canonical 
bijection ${\mathcal C}_{\ep,N}:\R^l\times\T^l\leftrightarrow \R^l\times\T^l$ close to the identity can be constructed in such a way that:   
\be
\label{CNF} 
(\H_\ep\circ {\mathcal C}_{\ep,N})(\xi,x)=\L_\om(\xi)+\sum_{k=1}^N 
\B_k(\xi;\om)\ep^k+\ep^{N+1}{\mathcal R}_{N+1,\ep}(\xi,x)
\ee
This  makes the flow of  $\H_\ep(\xi,x)$ integrable up to an error of order $\ep^{N+1}$. Here  ${\mathcal C}_{\ep,N}$ is the  flow at time $1$ generated by  the Hamiltonian
\be
\label{FGen}
 \W^N_\ep(\xi,x):=\la\xi,x\ra+\sum_{k=1}^N\W_k(\xi,x)\ep^k.
\ee
 The functions $\W_k(\xi,x): \R^l\times \T^l\to\R$ are recursively computed by canonical perturbation theory via the standard Lie transform method of Deprit\cite{De} and Hori\cite{Ho}  (see also e.g \cite{Ca}). 

To describe the quantum counterpart, let $H_\ep=L_\om+\ep V$ be the  
operator in 
$L^2(\T^l)$  of symbol $\H_\ep$, with domain $D(H_\ep)= H^1(\T^l)$ 
 and action specified as follows: 
\bea
\forall  u\in D(H_\ep), \quad H_\ep u= L_\om u+Vu,  \quad 
 L_\om u=\sum_{k=1}^l\om_kD_ku, \;\; D_k  u:=-i\hbar\partial_{x_k}u.  
\eea
$V$ is  the Weyl quantization of $\V$ (formula (\ref{1erweyl}) below).

 Since {\it uniform} quantum normal forms (see e.g. \cite{Sj},\cite{BGP},\cite{Po1}, \cite{Po2}) are not so well known as the classical ones, let us recall here their definition. The  construction is reviewed  in Appendix. 
 
 \begin{definition}
 [Quantum normal form (QNF)]\label{QuNF}
{\it  We say that a family of operators $H_\ep$ $\ep$-close (in the norm resolvent topology) to $H_0=L_\omega$ admits a uniform quantum normal form (QNF) at any order if 
\begin{itemize}
\item[(i)]
  There exists a sequence of continuous self-adjoint operators $W_k(\hbar)$ in $L^2(\T^l)$, $k=1,\ldots$ and a sequence of 
  functions
 $B_k(\xi_1,\ldots,\xi_l,\hbar)\in C^\infty(\R^l\times [0,1];\R)$, such that, defining $\forall\,N\in\Na$ the family of unitary operators:
 \bea
\label{QNF}
U_{N,\ep}(\hbar)=e^{iW_{N,\ep}(\hbar)/\hbar}, \quad W_{N,\ep}(\hbar)=\sum_{k=1}^N 
W_k(\hbar)\ep^k
\eea 
we have:
\bea
\label{AQNF}
&&
U_{N,\ep}(\hbar)H_\ep U_{N,\ep}^\ast(\hbar)=L_\om+\sum_{k=1}^N 
B_k(D_1,\ldots,D_l,\hbar)\ep^k+\ep^{N+1}R_{N+1,\ep}(\hbar).
\eea
\item [(ii)]
The continuous operators $W_k$,  $B_k(D,\hbar)$, $R_{N+1}$   admit smooth symbols $\W_k, \B_k,  \RR_{N+1}(\ep)$, which reduce  to the classical normal form construction (\ref{CNF}) and (\ref{FGen}) as $\hbar\to 0$:
\be
\label{princip}
\B_k(\xi;0)=\B_k(\xi);\quad \W_k(\xi,x,0)=\W_k(\xi,x),\quad \RR_{N+1,\ep}(x,\xi;0)=\RR_{N+1,\ep}(x,\xi)
\ee
 \end{itemize}}
\end{definition}
 (\ref{AQNF})  entails that  $H_\ep$ commutes with $H_0$ up to an error of order $\ep^{N+1}$; hence  the following approximate QF formula holds for the eigenvalues of $H_\ep$:
\be
\label{AQF}
\lambda_{n,\ep}(\hbar)=\hbar\la n,\om\ra+\sum_{k=1}^N \B_k(n_1\hbar,\ldots,n_l\hbar,\hbar)\ep^k+O(\ep^{N+1}).
\ee
 \vskip 6pt\noindent
 \begin{definition}
 \label{QNFConv}{(Smoothly and uniformly convergent quantum normal forms)}
{\it We say that the QNF is smoothlly (with respect to $(\xi,x)\in\R^l\times\T^l)$ and  uniformly (with respect to $\hbar$)} convergent,  {\it if  there is $\ep^\ast>0$ such that, for $|\ep|<\ep^\ast$  and any $\alpha,\beta,\gamma\in\mathbb N^l$, one has}
\vskip 5pt\noindent
\bea
&&
\label{convunifQ1}
\sum_{k=1}^\infty\,\sup_{\R^l\times\T^l\times [0,1]}|D^\alpha_\xi D^\beta_x\W_k(\xi,x;\hbar)\ep^k|<+\infty 
\\
&&
\label{convunifQ2}
\sum_{k=1}^\infty\,\sup_{\R^l\times [0,1]}|D^\gamma_\xi\B_k(\xi,\hbar)\ep^k|<+\infty.
\eea
\vskip 5pt\noindent
\end{definition}
\noindent
(\ref{convunifQ1},\ref{convunifQ2}) entail that,    if $|\ep|<\ep^\ast$, we can define the symbols  
\vskip 3pt\noindent
\bea
\label{somma}
&&
\W_{\infty}(\xi,x;\ep,\hbar):=\la \xi,x\ra+\sum_{k=1}^\infty\W_k(\xi,x;\hbar)\ep^k\in C^M(\R^l\times\T^l\times [0,\ep^\ast] \times[0,1];\C), 
\\
\label{somma1}
&&
 \B_{\infty}(\xi;\ep,\hbar):=\L_\om(\xi)+\sum_{k=1}^\infty\B_k(\xi;\hbar)\ep^k \in C^M(\R^l\times [0,\ep^\ast] \times[0,1];\C)
\eea
such that,  $\forall\,\alpha,\beta,\gamma\in\mathbb N^l$
\bea
\label{stimasomma}
&&
\sup_{\R^l\times\T^l\times [0,1]}|D^\alpha_\xi D^\beta_x\W_{\infty}(\xi,x;\ep,\hbar)-\la\xi,x\ra|<+\infty, 
\\
\label{stimasomma1}
&&
\sup_{\R^l\times [0,1]} |D^\gamma\B_{\infty}(\xi;\ep,\hbar)|<+infty
\eea
\vskip 3pt\noindent
The uniform convergence of the QNF  has the following  straightforward consequences: 
\begin{itemize}
\item[(A1)]
{\it By the Calderon-Vaillancourt theorem   (see \S 3 below) 
the Weyl quantizations   $W_{\infty}(\ep,\hbar)$, $B_{\infty}(\ep,\hbar)$ of $\W_{\infty}(\xi,x;\ep,\hbar)$, $\B_{\infty}(\ep,\hbar)$ are continuous operator in $L^2(\T^l)$. Then:} 
\begin{eqnarray*}
&&
e^{iW_{\infty}(\ep,\hbar)/\hbar}H_\ep e^{-iW_{\infty}(\ep,\hbar)/\hbar}=B_{\infty}(D_1,\ldots,D_l;\ep,\hbar).
\\
&&
B_{\infty}(D_1,\ldots,D_l;\ep,\hbar):=L_\om+\sum_{k=1}^\infty B_k(D_1,\ldots,D_l;\hbar)\ep^k.
\end{eqnarray*}
\item[(A2)]
{\it The eigenvalues  of $H_\ep$ are given by the {\rm exact} quantization formula:}
\be
\label{QF}
\lambda_{n}(\hbar,\ep)=\B_{\infty}(n\hbar,\hbar,\ep), \qquad n\in\Z^l, \quad  \ep\in {\frak D}^\ast:=\{\ep\in \R\,|\,|\ep|<\ep^\ast\}
\ee
\item [(A3)] {\it The  classical normal form is convergent, uniformly on compacts with respect to $\xi\in\R^l$, and therefore  if $\ep\in {\frak D}^\ast$ the Hamiltonian $\H_\ep(\xi,x)$ is integrable.}
\end{itemize}

Let us now state explicit conditions on $V$ ensuring the uniform convergence of the QNF.
\newline
Given $\F(t,x)\in C^\infty(\R\times\T^l;\R)$, consider  its Fourier expansion
\be
\label{FFE}
 \F(t,x)=\sum_{q\in\Z^l}\F_q(t)e^{i\la q,x\ra}.
 \ee
and define  $ \F_\om \in C^\infty(\R^l\times\T^l;\R)$ in the following way: 
\vskip 4pt\noindent
\bea
&&
\label{Fouom}
 \F_\om(\xi,x):=\F(\L_\om(\xi),x)=\sum_{q\in\Z^l}\F_{\om,q}(\xi)e^{i\la q,x\ra}, 
 \\
 &&
  \F_{\om,q}(\xi):=(\F_q\circ \L_\om)(\xi)=\frac1{(2\pi)^{l/2}}\int_\R\widehat{\F}_q(p)e^{-ip\L_\om(\xi)}\,dp=
  \\
  &&
 = \frac1{(2\pi)^{l/2}}\int_\R\widehat{\F}_q(p)e^{-i\la p\om,\xi\ra}\,dp, \quad p\om :=(p\om_1,\ldots,p\om_l ).
\eea
\vskip 4pt\noindent
Here, as above, $\L_\om(\xi)=\la\om,\xi\ra$.
\vskip 4pt\noindent
Given $\rho>0$, introduce the weighted norms:
\bea
&&
\|\F_{\om,q}(\xi)\|_\rho:=\int_\R|\widehat{\F}_q(p)|e^{\rho |p|}|\,dp
\\
&&
\|\F_\om(x,\xi)\|_{\rho}:=\sum_{q\in\Z^l}\,e^{\rho |q|}\|\F_{\om,q}\|_\rho
\eea
\vskip 4pt\noindent
 We can now formulate the main result of this paper. Assume:
 \vskip 4pt\noindent
\begin{itemize}
\item[(H1)] There exist $\gamma >0, \tau \geq l$ such that the frequencies $\om$ fulfill 
 the diophantine condition 
 \be
\label{DC}
|\la\om,q\ra|^{-1}\leq \gamma |q|^{\tau}, \quad q \in\Z^l, \; q\neq 0. 
\ee
\item[(H2)]  $V_\om$ is the Weyl quantization of $\V_\om(\xi,x)$ (see Sect.3 below), that is:
\vskip 8pt\noindent
\be\label{1erweyl}
V_\om f(x)=\int_{\R}\sum_{q\in\Z^l}\widehat{\V}_q(p)
e^{i\la q,x\ra+\hbar p\la \om,q\ra/2}f(x+\hbar p\om)\,dp, \quad f\in L^2(\T^l).
\ee
\vskip 5pt\noindent
Here $\V_\om(\xi,x)=\V(\la\omega,\xi\ra,x)$ for some smooth function
$\V(t;x): \R\times\T^l\to \R$.
\vskip 10pt\noindent
\item[(H3)] There is $ \rho>2 $ such that $ \|\V_\om\|_{\rho}<+\infty.$
\end{itemize}
\vskip 7pt\noindent
Clearly under these conditions  
the operator family  $ H_\ep:=L_\om+\ep V_\om$, $D(H_\ep) =H^1(\T^l)$, $\ep\in\R$, is self-adjoint in $L^2(\T^l)$ and has pure point spectrum. We can then state the main results.
\vskip 4pt\noindent
\begin{theorem}
\label{mainth}  {\rm (Uniform convergence)}
\newline
Assume the validity of conditions (H1-H3). 
Let the diophantine constants $\gamma$, $\tau$ be such that:
\vskip 7pt\noindent
\be
\label{DC1}
\gamma\tau^{\tau}(\tau+2)^{4(\tau+2)}< \frac12.
 \ee
 \vskip 5pt\noindent
Then $ H_\ep$ admits a smoothly, uniformly convergent quantum normal form
$\B_{\infty,\om}(\xi,\ep,\hbar)$, in the sense of 
Definition 1.5. The 
 radius of convergence is not smaller than:
 \vskip 6pt\noindent
\be
\label{rast}  
  \ep^\ast(\tau):=
 \frac{1}{2^{2\tau}e^{24(2+\tau)}\|\V_\om\|_{\rho}}.
   \ee
Furthermore $\B_{\infty}(t,\ep,\hbar)$ is holomorphic with respect to $t$ in $\{t\in\C\,|\,|\Im t|<\rho/2\}$.
\end{theorem}
\vskip 6pt
Our second result concerns the regularity  of $\B_{\infty,\om}(\xi;\ep,\hbar)$ with respect to $\hbar$. This property will depend on the radius of convergence as shown in the following Theorem. Although this point is not discussed here, we believe that $\B_{\infty,\om}(\xi;\ep,\hbar)$ has Gevrey regularity with respect to the Planck constant.
\begin{theorem} 
\label{regolarita} {\rm (Regularity with respect to $\hbar$).}
\newline
 For $r=0,1,\ldots$  let the diophantine constants $\gamma$, $\tau$ be such that:
 \vskip 7pt\noindent
 \be
 \label{DC2}
\gamma\tau^{\tau}(r+\tau+2)^{4(r+\tau+2)}< \frac12.
 \ee
 \vskip 4pt\noindent
and let: 
 \bea
&&
 \label{Dr}
 {\frak D}(\tau,r):=\{\ep\in\C\,|\,|\ep|<\ep^\ast(\tau,r)\},
 \\
 \label{epastr}
 &&
  \ep^\ast(\tau,r):=
 \frac{1}{e^{24(2+r+\tau)}(r+2)^{2\tau}\|\V_\om\|_{\rho}}=e^{-24 r}\left(\frac 2{2+r}\right)^{2\tau}\ep^\ast(\tau)
 \eea
\vskip 4pt\noindent
  Then, under the validity of conditions (H1-H3), there exists  $C_r=C_r(\ep^\ast)>0$ such that, for $\ep\in {\frak D}(\tau,r)$:
\vskip 4pt\noindent
\bea
\label{stimaG1}
\sum_{\gamma=0}^r\max_{\hbar\in [0,1]} \|\partial^\gamma_\hbar \B_{\infty,\om}(.;\ep,\hbar) \|_{\rho/2}\leq C_r, \;\;r=0,1,\ldots
\eea
In particular:  $\B_{\infty,\om}(\xi;\ep,.)\in C^r([0,1])$ uniformly w.r.t. $\xi\in\R^l$ and $|\ep|<\ep^\ast(\tau,r)$.
 \end{theorem}
\noindent
{\bf Remark}
\newline
 Since (see \S 2 below) functions $\F(t,\ep,\hbar)$ such that $\ds \sup_{\hbar\in [0,1]}\|\F(\cdot,\ep,\hbar)\|_\rho$ are holomorphic  w. r. t. $t$ in $\{t\in\C\,|\,|\Im t|<\rho\}$,  \eqref{stimaG1} taken for $r=0$ yields a  quantitative restatement of Theorem \ref{mainth}. 
\vskip 4pt
In view of Definition \ref{quant}, the following statement is a straightforward consequence of the above  Theorems:
\begin{corollary}[Quantization formula]\label{QFE}
$\H_\ep$ admits an exact, $\infty$-smooth quantization formula in the sense of Definition 1.1. That is, $\forall\,r\in\Na$, 
$\forall \,|\epsilon|<\ep^\ast (\tau,k)$ given by (\ref{epastr}), the eigenvalues of $H_\ep$ are  expressed by the formula:
\be
\label{EQF}
\lambda(n,\hbar,\ep)=\B_{\infty,\om}(n\hbar,\ep, \hbar)
=\L_\om(n\hbar)+\sum_{s=1}^\infty\B_s(\L_\om(n\hbar),\hbar)\ep^s
\ee
where $\B_{\infty,\om}(\xi,\ep, \hbar)$ belongs to $C^r(\R^l\times [0,\ep^\ast(\cdot,r)]\times [0,1])$, and admits an asymptotic expansion at order $r$ in $\hbar$, uniformly on compacts with respect to $(\xi,\ep)\in\R^l\times [0,\ep^\ast(\cdot,r)]$. 
\end{corollary}
{\bf Remarks}
\begin{itemize}
 \item[(i)]
(\ref{stimaG1}) and (\ref{EQF}) entail also that the Einstein-Brillouin-Keller (EBK) quantization formula:
\be
\label{EBK}
\lambda_{n,\ep}^{EBK}(\hbar):=\L_\om(n\hbar)+\sum_{s=1}^\infty \B_s(\L_\om(n\hbar))\ep^s=\B_{\infty,\om}(n\hbar,\ep),\quad n\in\Z^l
\ee
reproduces here ${\rm Spec}(H_\ep)$ up to order $\hbar$.
\item[(ii)]
Apart the  classical Cherry theorem  
 yielding convergence of the Birkhoff normal form for smooth perturbations of  the harmonic flow with {\it complex}  frequencies when $l=2$ (see e.g. \cite{SM}, \S 30; the uniform convergence of the QNF under these conditions is proved in  \cite{GV}), no simple convergence criterion seems to be known for the QNF nor for the classical NF as well.  (See e.g.\cite{PM}, \cite{Zu}, \cite{St} for  reviews on convergence of normal forms).  Assumptions (1) and (2) of Theorem \ref{mainth} entail Assertion (A2) above. Hence they represent, to our knowledge, a first  explicit convergence criterion for the NF. 
 \item[(iii)] In comparison to earlier results on QNF and quantization 
 formulas \cite{Sj}, \cite{BGP}, \cite{Po1}, \cite{Po2}, we remark that the present ones 
  are  {\it exact} and  {\it purely quantum}:
  i.e. it they are valid for $\hbar$ fixed, and not only  asymptotically as $\hbar \to 0$ 
  modulo 
  an error term of order $\hbar^\infty$ or $e^{-C/\hbar}$ . 
 \end{itemize}
 Remark that  $\L_\om(\xi)$ is also  the form taken by harmonic-oscillator Hamiltonian  in $\R^{2l}$, 
 $$
  \P_0(\eta,y;\om):= \sum_{s=1}^l\om_s(\eta^2_s+y_s^2), \quad (\eta_s,y_s)\in\R^2,\quad s=1,\ldots,l
  $$ 
  if expressed  in terms of the action variables $\xi_s>0, \,s=1,\ldots,l$, where
 $$
 \xi_s:=\eta^2_s+y_s^2=z_s\overline{z}_s, \quad z_s:=y_s+i\eta_s. 
 $$
 {Assuming} (\ref{DC}) {\it and}  the property 
 \be
 \label{Rk1}
\B_k(\xi)=(\F_k\circ\L_\om(\xi))=\F_{k}(\sum_{s=1}^l \om_s z_s\overline{z}_s), \quad k=0,1,\ldots
\ee
  R\"ussmann \cite{Ru} (see also \cite{Ga}) proved convergence of the Birkhoff NF if the perturbation $\V$, expressed as a function of $(z,\overline{z})$, is {\it in addition} holomorphic  at the origin in $\C^{2l}$.   No explicit condition   on $\V$ seems to be known ensuring {\it both} (\ref{Rk1}) and the holomorphy. In this case instead
 we {\it prove} that the assumption $\V(\xi,x)=\V(\L_\om(\xi),x)$ entails  (\ref{Rk1}), uniformly in  $\hbar\in [0,1]$; namely,  we construct $\F_s(t;\hbar):\R\times [0,1]\to\R$ such that:
\be
\label{Rk}
\B_s(\xi;\hbar)=\F_s(\L_\om(\xi);\hbar):=\F_{\om,s}(\xi;\hbar), \quad s=0,1,\ldots
\ee
The conditions of Theorem \ref{mainth} cannot however be transported to R\"ussmann's case: the map 
\vskip 6pt\noindent
$$
{\mathcal T}(\xi,x)=(\eta,y):=  \begin{cases} \eta_i=-\sqrt{\xi_i}\sin x_i, 
\\
y_i=\sqrt{\xi_i}\cos x_i, \end{cases}\quad i=1,\ldots,l,
$$
\vskip 6pt\noindent
 namely, the inverse transformation  into action-angle variable,  is defined only on $\R_+^l\times\T^l$ and does not preserve the analyticity at the origin.
On the other hand,  ${\mathcal T}$ is an analytic, canonical map between $\R_+^l\times\T^l$ and $\R^{2l}\setminus\{0,0\}$. Assuming for the sake of simplicity  $\V_0=0$ the image of $\H_\ep$ under ${\mathcal T}$ is:
\bea
\label{H0}
(\H_\ep \circ {\mathcal T})(\eta,y)= \sum_{s=1}^l\om_s(\eta^2_s+y_s^2)+\ep (\V\circ {\mathcal T})(\eta,y):=\P_0(\eta,y)+\ep \P_1(\eta,y)
\eea
where
\bea
&&
\label{H1}
\P_1(\eta,y)=(\V\circ {\mathcal T})(\eta,y)=\P_{1,R}(\eta,y)+\P_{1,I}(\eta,y), \;(\eta,y)\in\R^{2l}\setminus\{0,0\}. 
\eea
\bea
&&
\nonumber
\P_{1,R}(\eta,y)=\frac12\sum_{k\in\Z^l}(\Re{\V}_k\circ\H_0)(\eta,y)\prod_{s=1}^l \left(\frac{\eta_s-iy_s}{\sqrt{\eta^2_s+y_s^2}}\right)^{k_s}
\\
\nonumber
&&
\P_{1,I}(\eta,y)=\frac12\sum_{k\in\Z^l} (\Im{\V}_k\circ\H_0)(\eta,y)\prod_{s=1}^l \left(\frac{\eta_s-iy_s}{\sqrt{\eta^2_s+y_s^2}}\right)^{k_s}
\end{eqnarray}
\vskip 4pt\noindent
If $\V$ fulfills Assumption (H3) of Theorem \ref{mainth}, both these series converge uniformly  in any compact of $\R^{2l}$ away from the origin  and $\P_1$ is holomorphic  on $\R^{2l}\setminus\{0,0\}$. 
Therefore Theorem \ref{mainth}  immediately entails a convergence criterion for the Birkhoff normal form generated by perturbations  holomorphic away from the origin. We state it under the form of a corollary: 
\begin{corollary}
\label{mainc}
{\rm (A convergence criterion for the Birkhoff normal form)} 
Under the assumptions of Theorem \ref{mainth} on $\om$ and $\V$, consider on $\R^{2l}\setminus\{0,0\}$ the holomorphic Hamiltonian family  $P_\ep(\eta,y):=\P_0(\eta,y)+\ep\P_1(\eta,y)$, $\ep\in\R$, where $\P_0$ and $\P_1$ are defined by (\ref{H0},\ref{H1}).  Then the Birkhoff normal form of $H_\ep$ is uniformly convergent on any compact of  $\R^{2l}\setminus\{0,0\}$ if $|\ep|<\ep^\ast (\gamma,\tau)$. 
\end{corollary}
\vskip 0.5cm\noindent

\subsection{Strategy of the paper}
The proof of Theorem \ref{mainth} rests on an implementation  in the quantum context  of R\"ussmann's 
argument\cite{Ru} yielding convergence of the KAM iteration when the complex variables $(z,\overline{z})$ 
belong to an open neighbourhood of the origin in $\C^{2l}$.  Conditions (\ref{DC}, \ref{Rk}) 
prevent the occurrence of accidental degeneracies among eigenvalues at any  step of the quantum 
KAM iteration, in the same way as they prevent the formation of resonances at the same step in  the classical case.  However, the global nature of quantum mechanics prevents phase-space localization;  
therefore, and this is the main difference, at each step the coefficients of the homological equation   for the operator symbols not only have an additional dependence on $\hbar$ but also have to be controlled up to infinity.  
These difficulties are  overcome by exploiting the closeness to the identity of the whole procedure, introducing adapted spaces of symbols i(Section \ref{not}), which account also for  
 the 
properties of differentiability with respect to the Planck constant.   The link between quantum and classical settings
 is provided by a sharp (i.e. without $\hbar^\infty$ approximation) Egorov Theorem established in section 
\ref{sectionegorov}. 
Estimates for the solution of the quantum homological equation and their recursive properties are 
obtained in sections \ref{hom} (Theorem \ref{homo}) and \ref{towkam} (Theorem \ref{resto})  respectively. Recursive estimates are established in Section \ref{recesti} (Theorem \ref{final})
and the proof of our main result is completed in section \ref{iteration}. The link with the usual construction of the quantum normal form described in  Appendix.

\vskip 1cm\noindent
\section{Norms and first estimates}
\label{not} 
\setcounter{equation}{0}%
\setcounter{theorem}{0}%

  Let $m,l=1,2,\dots$. For $(\xi,x,\hbar)\mapsto\F(\xi,x;\hbar)\in C^\infty(\R^m\times\T^l\times [0,1]; \C)$  
 and $ (\xi,\hbar)\mapsto\G(\xi;\hbar)\in C^\infty(\R^m\times [0,1]; \C)$, consider, for $p\in\R^m$ and $q\in\Z^m$ the following Fourier transforms 
 \vskip 6pt\noindent
 \begin{definition}[Fourier transforms]\label{deffour}
 \be
 \widehat{\G}(p;\hbar)=\frac1{(2\pi)^{m/2}}\int_{\R^m}\G(\xi;\hbar)e^{-i\la p,\xi\ra}\,dx
\ee
 \be
 \truc{\F}(\xi,q;\hbar):=\frac1{(2\pi)^{m/2}}\int_{\T^l}\F(\xi,x;\hbar)e^{-i\la q,x\ra}\,dx .
\quad 
\ee 
\end{definition}
Note that
\be
\label{FE1}
\F(\xi,x;\hbar)=\sum_{q\in\Z^l}\truc{\F}(\xi,q;\hbar)e^{-i\la q,x\ra}
\ee
\be
\label{FE2} 
\widehat{\F}(p,q;\hbar)=\frac1{(2\pi)^{m/2}}\int_{\R^m}\truc{\F}(\xi,q;\hbar)e^{-i\la p,\xi\ra}\,dx
\ee
 \vskip 10pt\noindent
 It is convenient to rewrite the Fourier representations (\ref{FE1}, \ref{FE2}) under the form a single Lebesgue-Stieltjes integral. Consider the product measure on $\R^m\times \R^l$:
\bea
\label{pm1}
&&
d\lambda (t):=dp\,d\nu(s),  \quad t:=(p,s)\in\R^m\times \R^l;
\\
\label{pm2}
&&
dp:=\prod_{k=1}^m\,dp_k;\quad d\nu(s):=\prod_{h=1}^l \sum_{q_h\leq s_h} \delta (s_h-q_h),
 \;q_h\in\Z, h=1,\ldots,l
\eea
Then:
\be
\label{IFT}
\F(\xi,x;\hbar)=\int_{\R^m\times\R^l}\,\widehat{\F}(p,s;\hbar)e^{i\la p,\xi\ra +i\la s,x\ra}\,d\lambda(p,s)
\ee
\begin{definition}[Norms I]
{\it   For $\rho\geq 0$, $\sigma\geq 0$, we 
 introduce the weighted norms }
 \vskip 3pt\noindent
\bea
\label{norma1}
|\G|^\dagger_{\sigma}&:=&\max_{\hbar\in [0,1]}\|\widehat{\G}(.;\hbar)\|_{L^1(\R^m,e^{\sigma |p|}dp)}=\max_{\hbar\in [0,1]}\int_{\R^l}\|\widehat{\G}(.;\hbar)\|\,e^{\sigma |p|}\,dp.
\\
\label{norma1k}
|\G|^\dagger_{\sigma,k}&:=&\max_{\hbar\in [0,1]}\sum_{j=0}^k\|(1+|p|^2)^{\frac{k-j}{2}}\partial^j_\hbar\widehat{\G}(.;\hbar)\|_{L^1(\R^m,e^{\sigma |p|}dp)};\quad  |\G|^\dagger_{\sigma;0}:=|\G|^\dagger_{\sigma}.
\eea
\end{definition}
\begin{remark}
By noticing that $\vert p\vert\leq\vert p^\prime-p\vert+\vert p^\prime\vert$ and that, for 
$x\geq 0$, $\ds x^je^{-\delta x}\leq \frac 1 e\left(\frac j{\delta}\right)^j$, we immediately  get the inequalities
\be\label{plus}
\vert \F\G\vert^\dagger_{\sigma}\leq\vert \F\vert_{\sigma}^\dagger\cdot \vert \G\vert_{\sigma}^\dagger, 
\ee
\vskip 6pt\noindent
\be
\label{diff}
\vert (I-\Delta^{j/2})\F\vert_{\sigma-\delta}^\dagger\leq \frac1 e\left(\frac j\delta\right)^j\vert \F\vert_\sigma^\dagger, \quad k\geq 0.
\ee
\end{remark}
Set now for $ k\in\Na\cup\{0\} $:
\be\label{muk}
\mu_{k}(t):=(1+|t|^{2})^{\frac k 2}=(1+|p|^{2}+|s|^{2})^{\frac k 2}.
\ee
and note that
\be
 \mu_k(t-t^\prime)\leq 2^{\frac k 2} \mu_k(t)\mu_k(
 t ^\prime).
\ee
because $|x-x^\prime|^2\leq 2(|x|^2+|x^\prime|^2)$. 
\begin{definition}[Norms II] {\it  Consider $\F(\xi,x;\hbar)\in C^\infty(\R^m\times \T^l\times[0,1];\C)$, with Fourier expansion 
\be
\label{FF}
\F(\xi,x;\hbar)=\sum_{q\in\Z^l}\,\truc{\F}(\xi,q;\hbar)e^{i\la q,x\ra}
\ee
\begin{itemize}
\item
[(1)]
 Set:
\bea
\label{sigmak}
\Vert \F\Vert^\dagger_{\rho,k}:=\max_{\hbar\in [0,1]}\sum_{\gamma=0}^k
\int_{\R^m\times \R^l}\vert \mu_{k-\gamma}(p,s)\partial^\gamma_\hbar\widehat{\F}(p,s;\hbar)\vert e^{\rho(\vert s\vert+\vert p\vert)}\,d\lambda(p,s).
\eea
\vskip 4pt\noindent
\item
[(2)]
Let ${\mathcal O}_\omega$ be the set of functions ${\Phi}:\R^l\times\T^l\times[0,1]\to\C$ such that $\Phi(\xi,x;\hbar)=\F(\L_\omega(\xi),x;\hbar)$ for some $\F:\ \R\times\T^l\times [0,1]\to \C$. 
Define, for $\Phi\in {\mathcal O}_\omega$:
\bea\label{sigom}
\Vert \Phi\Vert_{\rho,k}:=\max_{\hbar\in [0,1]}\sum_{\gamma=0}^k
\int_{\R}\vert \mu_{k-\gamma}( p\omega,q)
\partial^\gamma_\hbar\widehat{\F}(p,s;\hbar)\vert e^{\rho(\vert s\vert+\vert p\vert}\,d\lambda(p,s).
\eea
\vskip 4pt\noindent
\item
[(3)] Finally we denote $Op^W(\F)$ the Weyl quantization of $\F$ recalled in Section \ref{sectionweyl} and
\bea
\label{normsymb'}
\J^\dagger_k(\rho)&=&\{\F \,|\,\Vert \F\Vert^\dagger_{\rho,k}<\infty\},
\\
\label{normop'}
J^\dagger_k(\rho)&=&\{Op^W(\F)\,|\,\F\in\J^\dagger_k(\rho)\},
\\
\label{normsymb}
\J_k(\rho)&=&\{\F\in {\mathcal O}_\omega\,|\,\Vert \F\Vert_{\rho,k}<\infty\},
\\
\label{normop}
J_k(\rho)&=&\{Op^W(\F)\,|\,\F\in\J_k(\rho)\}.
\eea 
\end{itemize}}
\end{definition}

Finally  we denote: $L^1_\sigma(\R^m):=L^1(\R^m,e^{\sigma |p|}dp)$.
\begin{remark}
Note that, if $\F(\xi,x,\hbar)$ is independent of $x$, i.e. 
$\truc{\F}(\xi,q,\hbar)=\F(\xi,\hbar)\delta_{q,0}$, then:
\be
\label{normeid}
\|\F\|^\dagger_{\rho,k}=|\F|^\dagger_{\rho,k}; \quad \|\F\|_{\rho,k}=|\F|_{\rho,k}
\ee
while in general
\bea
 &&
 \|\F\|_{\rho,k}\leq \|\F\|_{\rho^\prime,k^\prime}
 \quad {\rm whenever}\; k\geq k^\prime,\,\rho\leq \rho^\prime;
\eea
\end{remark}
\begin{remark} (Regularity properties)
\vskip 4pt\noindent
Let $\F\in \J_k^\dagger(\rho), k\geq 0$.  Then:
\begin{enumerate}
 \item There exists $K(\alpha,\rho,k)$ such that
\be
\label{maggC}
\max_{\hbar\in [0,1]}\|\F(\xi,x;\hbar)\|_{C^\alpha(\R^m\times\T^l)}\leq K \|\F\|^\dagger_{\rho,k}, \quad \alpha\in\Na
\ee
and analogous statement for the norm $\|\cdot\|_{\rho,k}$. 
\item 
Let $\rho>0$, $k\geq 0$.  Then 
$\F(\xi,x;\hbar)\in C^k([0,1];C^\om(\{|\Im \xi|<\rho\}\times \{|\Im x|<\rho\})$ and
\vskip 4pt\noindent
\be
\label{supc}
\sup_{\{|\Im \xi|<\rho\}\times \{|\Im x|<\rho\}}|\F(\xi,x;\hbar)|\leq \|\F\|^\dagger_{\rho,k}.
\ee
\vskip 6pt\noindent
Analogous statements for $\F\in \J_k(\rho)$. 
\end{enumerate}
 \end{remark}
 We will show in section \ref{sectionweyl} that:
 \bea
\|Op^W(F)\|_{\mathcal B(L^2)}\leq \|\F\|_{\rho,k}\ \ \ \forall k,\ \rho >0.
 \eea
 In what follows we will often use the notation $\F$ also to denote the function $\F(\L_\om(\xi))$, and, correspondingly, $\|\F\|_{\rho,k}$ to denote $\|\F_\om\|_{\rho,k}$, because the indication of the belonging to $J$ or $J^\dagger$, respectively, is already sufficient to mark the distinction of the two cases.  
\begin{remark}
Without loss of generality we may assume:
\be
|\om |:=|\om_1|+\ldots+|\om_l |\leq 1
\ee
Indeed, the general case $|\om|=\alpha |\om^\prime|$, $|\om^\prime|\leq 1$, $\alpha>0$ arbitrary reduces to the former one just by  the rescaling $\ep\to \alpha\ep$. 
\end{remark}

 \vskip 1.0cm\noindent
\section{Weyl quantization, matrix elements, commutator  estimates}\label{sectionweyl}
\renewcommand{\thetheorem}{\thesection.\arabic{theorem}}
\renewcommand{\theproposition}{\thesection.\arabic{proposition}}
\renewcommand{\thelemma}{\thesection.\arabic{lemma}}
\renewcommand{\thedefinition}{\thesection.\arabic{definition}}
\renewcommand{\thecorollary}{\thesection.\arabic{corollary}}
\renewcommand{\theequation}{\thesection.\arabic{equation}}
\renewcommand{\theremark}{\thesection.\arabic{remark}}
\setcounter{equation}{0}%
\setcounter{theorem}{0}%
\subsection{Weyl quantization: action and matrix elements}
We sum up here the canonical (Weyl)  quantization procedure for functions (classical observables) 
defined on the phase space $\R^l\times\T^l$. In the present case it seems more convenient to 
consider the representation (unique up to unitary equivalences) of the natural Heisenberg group 
on $\R^l\times\T^l$. Of course this procedure yields the same quantization as the standard one 
via the Br\'ezin-Weil-Zak transform (see e.g. \cite{Fo}, \S 1.10) and has already been employed in \cite{CdV}, \cite{Po1},\cite{Po2}). 
\par
Let $\He_l(\R^l\times\R^l\times\R)$  be the Heisenberg group over $\ds \R^{2l+1}$ (see e.g.\cite{Fo}, Chapt.1).  Since the dual space of $\R^l\times\T^l$ under the Fourier transformation is $\R^l\times\Z^l$, 
the relevant Heisenberg group here is the subgroup of $\He_l(\R^l\times\R^l\times\R)$,   denoted by $\He_l(\R^l\times\Z^l\times\R)$,  defined as follows: 
\begin{definition}[Heisenberg group]
\label{HSG}
{\it Let $u:=(p,q), p\in\R^l, q\in\Z^l$, and let $ t\in\R$. Then $\He_l(\R^l\times\Z^l\times\R)$ is the subgroup of $\He_l(\R^l\times\R^l\times\R)$  topologically equivalent to $\R^l\times\Z^l\times\R$ with group law
\be
\label{HGL}
(u,t)\cdot (v,s)= (u+v, t+s+\frac12\Omega(u,v))
\ee
Here $\Omega(u,v)$ is the canonical $2-$form on $\R^l\times\Z^l$:}
\be
\label{2forma}
\Omega(u,v):=\la u_1,v_2\ra-\la v_1,u_2\ra
\ee
\end{definition}
$\He_l(\R^l\times\Z^l\times\R)$ is the Lie group generated via the exponential map from the Heisenberg Lie algebra ${\mathcal H}{\mathcal L}_l(\Z^l\times\R^l\times\R)$ defined as the vector space $\R^l\times\Z^l\times\R$ with Lie bracket
\be
\label{LA}
[(u,t)\cdot (v,s)]= (0, 0,\Omega(u,v))
\ee
The unitary representations of $\He_l(\R^l\times\Z^l\times\R)$  in $L^2(\T^l)$ are defined as follows
\be
\label{UR}
(U_\hbar(p,q,t)f)(x):=e^{i\hbar t +i\la q,x\ra+\hbar\la p.q\ra/2}f(x+\hbar p)
\ee
$\forall\,\hbar\neq 0$, $\forall\,(p,q,t)\in{\mathcal H}_l$, $\forall\,f\in L^2(\T^l)$. These representations fulfill the Weyl commutation relations
\be
\label{Weyl}
U_\hbar(u)^\ast =U_\hbar(-u), \qquad U_\hbar(u)U_\hbar(v)=e^{i\hbar\Omega(u,v)}U(u+v)
\ee
For any fixed $\hbar>0$  $U_\hbar$ defines the \Sc\ representation of the Weyl commutation relations, which also in this case is unique up to unitary equivalences (see e.g. \cite{Fo}, \S 1.10). 

Consider now a family of smooth phase-space functions indexed by $\hbar$,   $\A(\xi,x,\hbar):\R^l\times\T^l\times [0,1]\to\C$,  written under its Fourier representation
\vskip 4pt\noindent
\be
\label{FFR}
\A(\xi,x,\hbar)=\int_{\R^l}\sum_{q\in\Z^l}\widehat{\A}(p,q;\hbar)e^{i(\la p.\xi\ra +\la q,x\ra)}\,dp=\int_{\R^l\times \R^l}\widehat{\A}(p,s;\hbar)e^{i(\la p.\xi\ra +\la s,x\ra)}\,d\lambda(p,s)
\ee
\vskip 6pt\noindent
\begin{definition}[Weyl quantization]
\label{Qdef}
{\it The (Weyl) quantization of  $\A(\xi,x;\hbar)$  is the operator $A(\hbar)$ definde as}
\bea
\label{Wop}
&&
(A(\hbar)f)(x):=\int_{\R^l}\sum_{q\in\Z^l}\widehat{\A}(p,q;\hbar)U_\hbar(p,q)f(x)\,dp
\\
&&
\nonumber
= \int_{\R^l\times\R^l}\widehat{\A}(p,s;\hbar)U_\hbar(p,s)f(x)\,d\lambda(p,s)
 \quad f\in L^2(\T^l)
\eea
\end{definition}
\noindent
\begin{remark}
 Formula \eqref{Wop} can be also be written as
\be
\label{Wopq}
(A(\hbar)f)(x)=\sum_{q\in\Z^l}\mbox{ where }A(q,\hbar)f, \quad (A(q,\hbar)f)(x)=\int_{\R^l}\,\widehat{\A}(p,q;\hbar)U_\hbar(p,q)f(x)\,dp
\ee
\end{remark}
\noindent
From this we compute the action of $A(\hbar)$ on 
the canonical basis in $L^2(\T^l)$:
$$
e_m(x):=(2\pi)^{-l/2}e^{i\la m, x\ra}, \quad x\in\T^l, \;m\in\Z^l .
$$
We have:
\begin{lemma}  
\label{azione}
\be
\label{azioneAop}
A(\hbar)e_m(x)=  \sum_{q\in\Z^l}e^{i\la (m+q),x\ra}\truc{\A}(\hbar (m+q/2),q,\hbar) 
\ee
\end{lemma}
\begin{proof}
By \eqref{Wopq}, it is enough to prove that the action of $A(q,\hbar)$ is
\be
\label{azioneAq}
A(q,\hbar)e_m(x)=  e^{i\la (m+q),x\ra}\truc{\A}(\hbar(m+q/2),q,\hbar)
\ee
 Applying Definition \ref{Qdef} we can indeed  write:
\begin{eqnarray*}
&&
(A(q,\hbar)e_m)(x)=(2\pi)^{-l/2}\int_{\R^l}\widehat{A}(p,q;\hbar)e^{i\la q,x\ra+i\hbar \la p,q\ra/2}e^{i\la m,(x+\hbar  p)\ra}\,dp
\\
&&
=(2\pi)^{-l/2}e^{i\la (m+q),x\ra}\,\int_{\R^l}\widehat{A}(p;q,\hbar)e^{i\hbar \la p,(m+q/2)\ra}\,dp
=e^{i\la (m+q),x\ra}\truc{\A}(\hbar(m+q/2),q,\hbar).
\end{eqnarray*}
\end{proof}
We note for further reference an obvious consequence of \eqref{azioneAq}:
\be 
\label{ortogq}
\la A(q,\hbar)e_m,A(q,\hbar)e_n\ra_{L^2(\T^l)}=0,\;m\neq n;\quad 
\la  A(r,\hbar)e_m,A(q,\hbar)e_n\ra_{L^2(\T^l)}=0,\;r\neq q.
\ee
As in the case of the usual Weyl quantization, formula (\ref{Wop}) makes sense for tempered distributions
$\A(\xi,x;\hbar)$ \cite{Fo}. Indeed we prove in this context, for the sake of completeness, a simpler, but less general, version of the standard  Calderon-Vaillancourt criterion:
\begin{proposition}
Let $A(\hbar)$ by defined by (\ref{Wop}). Then
\vskip 8pt\noindent
\be
\label{CV}
\Vert A(\hbar)\Vert_{L^2\to L^2}\leq \frac{2^{l+1}}{l+2}\cdot \frac{\pi^{(3l-1)/2}}{\Gamma(\frac{l+1}{2})}\,\sum_{|\alpha|\leq 2k}\,\Vert \partial_x^k\A(\xi,x;\hbar)\Vert_{L^\infty(\R^l\times\T^l)}.
\ee
where
$$
k=\begin{cases} \frac{l}{2}+1,\quad l\;{\rm even}
\\
{}
\\
\frac{l+1}{2}+1,\quad l\;{\rm odd}.
\end{cases}
$$
\end{proposition}
\begin{proof}
 Consider the Fourier expansion
 $$
 u(x)=\sum_{m\in\Z^l}\,\widehat{u}_me_m(x),\quad u\in L^2(\T^l).
 $$
Since:
 $$
 \|A(q,\hbar)\widehat{u}_me_m\|^2=|\truc{\A}(\hbar(m+q/2),q,\hbar)
|^2\cdot |\widehat{u}_m|^2
$$
  by Lemma \ref{azione} and \eqref{ortogq} we get:
\begin{eqnarray*}
\|A (\hbar)u\|^2&\leq & \sum_{(q,m)\in\Z^l\times\Z^l}\|A(q,\hbar)\widehat{u}_m e_m\|^2
= \sum_{(q,m)\in\Z^l\times\Z^l}|\A(\hbar (m+q/2),q,\hbar)|^2\cdot |\widehat{u}_m|^2
\\
&\leq&
 \sum_{q\in\Z^l}\,\sup_{\xi\in\R^l}|\A(\xi,q,\hbar)
|^2\sum_{m\in\Z^l}|\widehat{u}_m|^2
=
 \sum_{q\in\Z^l}\,\sup_{\xi\in\R^l}|\A(\xi,q,\hbar)
|^2\|u\|^2
\\
&\leq&
\big[ \sum_{q\in\Z^l}\,\sup_{\xi\in\R^l}|\A(\xi,q,\hbar)
|\big]^2\|u\|^2
\end{eqnarray*}
Therefore:
\[
\Vert A(\hbar)\Vert_{L^2\to L^2} \leq \sum_{q\in\Z^l}\,\sup_{\xi\in\R^l}\vert\A(\xi,q,\hbar)\vert.
\]
Integration by parts entails that, for $k\in \Na$, and $\forall \,g\in  C^\infty(\T^l)$:
\vskip 8pt\noindent
\begin{eqnarray*}
&&
\left |\int_{\T^l}e^{i\la q,x\ra}g(x)dx\right |=\frac 1 {1+|q|^{2k}}\left |\int_{\T^l} e^{i\la q,x\ra}(1+(-\triangle_x)^k) g(x)dx\right |
\\
&&
\leq \frac 1 {1+\vert q\vert^{2k}}(2\pi)^l\sup_{\T^l}\sum_{|\alpha|\leq 2k}\vert \partial_x^\alpha g(x)\vert .
\end{eqnarray*}
\vskip 8pt\noindent
Let us now take: 
\be
\label{kappa}
k=\begin{cases} \frac{l}{2}+1,\quad l\;{\rm even}
\\
{}
\\
\frac{l+1}{2}+1,\quad l\;{\rm odd}
\end{cases}
\Longrightarrow 
\begin{cases} 2k-l+1=3,\quad l\;{\rm even}
\\
2k-l+1=2,\quad l\;{\rm odd}
\end{cases}
\ee
\vskip 4pt\noindent
Then $2k-l+1\geq 2$, and hence:
$$
\sum_{q\in\Z^l}\,\frac1{1+\vert q\vert^{2k}}\leq 2\int_{\R^l}\,\frac{du_1\cdots du_l}{1+\|u\|^{2k}}\leq 2\frac{\pi^{(l-1)/2}}{\Gamma(\frac{l+1}{2})}\int_0^\infty\frac{\rho^{l-1}}{1+\rho^{2k}}\,d\rho.
$$
Now:
\begin{eqnarray*}
&&
\int_0^\infty\frac{\rho^{l-1}}{1+\rho^{2k}}\,d\rho =\frac 1 {2k}\int_0^\infty\,\frac{u^{l/2k -1}}{1+u}\,du 
\\
&&
\leq \frac 1 {2k}\left(\int_0^1\, u^{l/2k -1}\,du+\int_1^\infty\,{u^{l/2k -2}}\,du\right)=\frac{1}{(4k-l)(2k-l)}
\end{eqnarray*}
 This allows us to conclude:
\begin{eqnarray*}
\sum_{q\in\Z^l}\,\sup_\xi\vert\A(\xi,q,\hbar)\vert &\leq &(2\pi)^l
\sum_{|\alpha|\leq 2k}\Vert \partial_x^{\alpha}\A(\xi,x;\hbar)\Vert_{L^\infty(\R^l\times\T^l)}\cdot \sum_{q\in\Z^l}\,\frac1{1+\vert q\vert^{2k}}
\\
&\leq & 2^{l+1}\cdot \frac{\pi^{(3l-1)/2}}{\Gamma(\frac{l+1}{2})}\frac{1}{l+2}\sum_{|\alpha|\leq 2k}\,\Vert \partial_x^k\A(\xi,x;\hbar)\Vert_{L^\infty(\R^l\times\T^l)}.
\end{eqnarray*}
with $k$ given by (\ref{kappa}). This proves the assertion. 
\end{proof}
\begin{remark}
Thanks to Lemma \ref{azione} we immediately see that, 
when $\A(\xi, x,\hbar)=\F(\L_\omega(\xi),x;\hbar)$, 
\vskip 8pt\noindent
\bea
\label{quant2}
&&
A(\hbar)f=\int_{\R}\sum_{q\in\Z^l}\widehat{\F}(p,q;\hbar)U_h(p\om ,q)f\,dp
\\
\nonumber
&&
= \int_{\R}\sum_{q\in\Z^l}\widehat{\F}(p,q;\hbar)e^{i\la q,x\ra+i\hbar p\la\om,q\ra/2}f(x+\hbar p\om)\,dp\quad f\in L^2(\T^l)
\eea
\vskip 5pt\noindent
where, again, $p\om:=(p\om_1,\dots,p\om_l)$. Explicitly,  \eqref{azioneAq} and \eqref{azioneAop} become:
\begin{eqnarray}
\label{azioneAom}
&&
A(\hbar)e_m(x)=  \sum_{q\in\Z^l}e^{i\la (m+q),x\ra}\truc{\F}(\hbar \la\om,(m+q/2)\ra,q,\hbar) 
\\
&&
\label{azioneAqom}
A(q,\hbar)e_m(x)=  e^{i\la (m+q),x\ra}\truc{\F}(\hbar\la\om,(m+q/2)\ra,q,\hbar)
\end{eqnarray}
\end{remark}
\begin{remark} If $\A$ does not depend on $x$, then $\truc{\A}(\xi,q,\hbar)=0, q\neq 0$, and (\ref{azioneAop}) reduces to the standard (pseudo) differential action
\bea
(A(\hbar) u)(x)=\sum_{m\in\Z^l}\overline{\A}(m\hbar
,\hbar) \widehat{u}_m e^{i\la m,x\ra}=\sum_{m\in\Z^l}\overline{\A}(-i\hbar\nabla,\hbar) \widehat{u}_m e^{i\la m,x\ra}
 \eea
 because $-i\hbar\nabla e_m=m\hbar e_m$.  On the other hand, if $\A$ does not depend on $\xi$ (\ref{azioneAop}) reduces to the standard multiplicative action
 \be
  (A (\hbar)u)(x)=\sum_{q\in\Z^l}\truc{\A}(q,\hbar)e^{i\la q,x\ra}\sum_{m\in\Z^l}\widehat{u}_m e^{i\la m,x\ra}=\A(x,\hbar)u(x)
 \ee
\end{remark}
\noindent
\begin{corollary} 
\label{corA}
Let $A(\hbar): L^2(\T^l)\to L^2(\T^l)$ be defined by \eqref{Wop} (Definition \ref{Qdef}) and $A(q,\hbar)$ by \eqref{Wopq}. Then:
\begin{enumerate}
\item
$\forall\rho\geq 0, \forall\,k\geq 0$ we have:
\be\label{stimz}
\Vert A(\hbar)\Vert_{L^2\to L^2}\leq\Vert\A\Vert^\dagger_{\rho,k}
\ee
and, if $\A(\xi, x,\hbar)=\A(\L_\omega(\xi),x;\hbar)$
\be\label{stimg}
\Vert A(\hbar)\Vert_{L^2\to L^2}\leq\Vert\A\Vert_{\rho,k}.
\ee
\item
\bea
\label{elm44}
&&
\la e_{m+s}, A(q,\hbar)e_m\ra =\delta_{q,s}\truc{\A}(
(m+q/2)\ra\hbar,q,\hbar)
\\
&&
\label{elm55}
\la e_{m+s},A(\hbar)e_m\ra =\truc{\A}((m+s/2)\hbar,s,\hbar)={(2\pi)^{-\frac m 2}}\int\limits_{\T^m}\A((m+s/2)\hbar,x,\hbar)e^{-i\langle s,x\rangle}dx
\eea
and, if $\A(\xi, x,\hbar)=\F(\L_\omega(\xi),x;\hbar)$
\bea
\label{elm4}
&&
\la e_{m+s}, A(q,\hbar)e_m\ra =\delta_{q,s}\truc{\F}(\la\om, (m+q/2)\ra\hbar,q,\hbar)
=\delta_{q,s}\truc{\F}(\L_\om (m+s/2)\hbar,q,\hbar)
\\
&&
\label{elm5}
\la e_{m+s},A(\hbar)e_m\ra =\truc{\F}(\la\om,(m\hbar+s\hbar/2)\ra,s,\hbar)
=\truc{\F}(\L_\om(m\hbar+s\hbar/2),s,\hbar)
\eea
Equivalently:
\be
\la e_m,A(\hbar) e_n\ra=\truc{\F}(\la \om,(m+n)\ra\hbar/2,m-n,\hbar)
\ee
\item $A(\hbar)$ is an operator of order $-\infty$, namely there exists $C(k,s)>0$ such that 
\be
\|A(\hbar)u\|_{H^k(\T^l)}\leq C(k,s)\|u\|_{H^s(\T^l)}, \quad (k,s)\in\R,\;  k\geq s
\ee
\end{enumerate}
\end{corollary}
\begin{proof}
(1) Formulae \eqref{stimz} and \eqref{stimg} are straighforward consequences  of 
Formula (\ref{maggC}).
\vskip 5pt\noindent
(2) \eqref{azioneAop} and \eqref{azioneAq} immediately yield \eqref{elm44} and \eqref{elm55}. Moreover,   (\ref{elm4}) immediately yields (\ref{elm5}).  In turn,  (\ref{elm4}) follows at once by \eqref{azioneAq}. 
\vskip 5pt\noindent
(3)  The condition $\A\in\J(\rho)$ entails:
\begin{eqnarray}
\label{stimaexp}
\sup_{(\xi;\hbar)\in\R^l\times [0,1]}|\A(\xi;q,\hbar)|e^{\rho |q|}\leq e^{\rho |q|}\max_{\hbar\in [0,1]}\|\widehat{\A}(p;q,\hbar)\|_1\to 0, \;|q|\to \infty.
\end{eqnarray}
 Therefore:
\begin{eqnarray*}
\|A(\hbar)u\|^2_{H^k}&\leq&  
\sum_{(q,m)\in\Z^l\times\Z^l}(1+|q|^2)^k\truc{\A}((m+q/2)\hbar,q,\hbar)
|^2\cdot |\widehat{u}_m|^2 
\\
&\leq&
\sum_{q\in\Z^l}\,\sup_{q,m}(1+|q|^2)^k|\truc{\A}((m+q/2)\hbar,q,\hbar)
|^2\sum_{m\in\Z^l}\,(1+|m|^2)^{s}|\widehat{u}_m|^2 \\
&=& C(k,s)\|u\|^2_{H^s}
\\
C(k,s)&:=&\sum_{q\in\Z^l}\,\sup_{q,m}(1+|q|^2)^k|\truc{\A}((m+q/2)\hbar,q,\hbar)|^2
\end{eqnarray*}
where $0<C(k,s)<+\infty$ by (\ref{stimaexp}) above. The Corollary is proved.
 \end{proof}
\subsection{Compositions, Moyal brackets}
 We first list the main pro\-perties which are straightforward consequences of the definition, as in the case of the standard Weyl quantization in $\R^{2l}$. First  introduce the abbreviations
\bea
\label{tt}
&&
t:=(p,s); \quad t^\prime=(p^\prime,s^\prime);\quad  \om t:=(p\om,s)
\\
\label{omt}
&&
\Omega_\om(t^\prime -t,t^\prime):=\la(p^\prime-p)\om,s^\prime\ra- \la (s^\prime-s),p^\prime\omega \ra
=\la p^\prime\om ,s\ra-\la s^\prime,p\om \ra.
\eea
 Given $\F(\hbar), \G(\hbar)\in \J_k(\rho)$, define their twisted convolutions:
\vskip 3pt\noindent
\bea
&&
\label{twc} 
(\widehat{\F}(\hbar){\widetilde{\ast}} \widehat{\G}(\hbar))(p,q;\hbar):=
\int_{\R\times\R^l}\widehat{\F}(t^\prime -t;\hbar)
\widehat{\G}(t^\prime;\hbar)e^{ i[\hbar \Om_\om(t^\prime -t,t^\prime)/2]}\,d\lambda(t^\prime)
\\
\nonumber
&&
{}
\\
&&
 \label{flat3}
 (\F\sharp\G)(x,\xi,\hbar):=
\int_{\R\times\R^l}
(\widehat{\F}(\hbar){\widetilde{\ast}} \widehat{\G}(\hbar))(t,\hbar)e^{i\la s,x\ra+p\L_\om(\xi)}\,d\lambda(t)
\\
\nonumber
{}
\\
\label{MBB1}
&&
\widehat{\mathcal C}(p,q;\hbar):=
\frac{1}{\hbar}\int_{\R\times\R^l}\widehat{\F}(t^\prime -t,\hbar)
\widehat{\G}(t^\prime,\hbar)\sin[\hbar \Omega_\om(t^\prime -t,t^\prime)/2]\,d\lambda(t^\prime)
\\
\nonumber
{}
\\
&&
\label{IFMM1}
{\mathcal C}(x,\xi;\hbar):=\int_{\R\times\R^l}
\widehat{\mathcal C}(p,s;\hbar)e^{ip\L_\om(\xi) +i\la s,x\ra}\,d\lambda(t)
\eea
\vskip 5pt\noindent
Once more by the same argument valid for the Weyl quantization in $\R^{2l}$:
\begin{proposition}
\label{Quant2}
The following composition formulas hold:
\vskip 5pt\noindent
 \bea
\label{Comm6}
&&
F(\hbar)G(\hbar)= \int_{\R\times\R^l}(\widehat{\F}(\hbar){\widetilde{\ast}} \widehat{\G}(\hbar))(t;\hbar)
U_\hbar(\om t)\,d\lambda(t). 
\eea
\vskip 3pt\noindent
\bea
\label{Comm7}
&&
\frac{[F(\hbar),G(\hbar)]}{i\hbar}= \int_{\R\times\R^l}\widehat{\mathcal C}(t;\hbar)U_\hbar(\om t)\,d\lambda(t)
\eea
\vskip 5pt\noindent
\end{proposition}
\begin{remark} The symbol of  the product $F(\hbar)G(\hbar)$  is  then $(\F\sharp\G)(\L_\om(\xi),x,\hbar)$
 and 
 the symbol of the commutator $[F(\hbar),G(\hbar)]/i\hbar$ is $ {\mathcal C}(\L_\om(\xi),x;\hbar)$, which is by definition the Moyal bra\-cket  of the symbols  $\F, \G$.  From (\ref{MBB1}) we get the asymptotic expansion:
 \vskip 3pt\noindent
 \bea
 \label{Mo6}
 &&
 \widehat{\mathcal C}(p,q;\om;\hbar)=\sum_{j=0}^\infty\frac{(-1)^j \hbar^{2j}}{(2j+1)!} D^j(p,q;\om)
 \\
 &&
 D^j(p,q;\om):=\int_{\R\times\R^l}\widehat{\F}(t^\prime-t,\hbar)
\widehat{\G}(t^\prime,\hbar)[\Omega_\om(t^\prime -t,t^\prime)^j\,d\lambda(t^\prime)
 \eea
  \vskip 3pt\noindent
 whence the asymptotic expansion for the Moyal bracket
  \vskip 3pt\noindent
 \bea
 \label{Moexp}
 &&
 \{\F, \G\}_M(\L_\om(\xi),x;\hbar)=\{\F, \G\}(\L_\om(\xi),x,\hbar)+ 
 \\
 \nonumber
 &&
 \sum_{|r+j|=0}^\infty\frac{(-1)^{|r|}\hbar^{|r+j|}}{r!sj}[\partial_x^r \om\partial^j_\L \F(\L_\om(\xi),x)]\cdot [ \om\partial^j_\L \partial_x^r G(\L_\om(\xi),x,\hbar)]-
\\
\nonumber
&&
-\sum_{|r+j|=0}^\infty\frac{(-1)^{|r|}\hbar^{|r+j|}}{r!j!}[\partial_x^r \om\partial^j_\L \G(\L_\om(\xi),x)]\cdot [ \om\partial^j_\L \partial_x^r F(\L_\om(\xi),x,\hbar)]
\eea
 Remark that:
\be
\label{Mo5}
\{\F, \G\}_M(\L_\om(\xi),x;\hbar)=\{\F, \G\}(\L_\om(\xi),x)+O(\hbar)
\ee
In particular, since $L_\om(\xi)$ is linear, we have $\forall\,\F(\xi;x;\hbar)\in C^\infty(\R^l\times\T^l\times[0,1])$:
\be
\label{MP}
\{\F, \L_\om(\xi)\}_M(\L_\om(\xi),x;\hbar)=\{\F, \L_\om(\xi)\}(\L_\om(\xi),x;\hbar)
\ee
\end{remark}
The observables  $\F(\xi,x;\hbar)\in\J(\rho)$ enjoy  the crucial property of stability under compositions of their dependence on $\L_\om(\xi)$ (formulae (\ref{flat3}) and (\ref{IFMM1}) above). 
As in \cite{BGP}, we want to estimate  the relevant quantum observables uniformly with respect to $\hbar$, i.e. 
through the weighted norm (\ref{sigom}).  
\vskip 3pt\noindent
\subsection{Uniform estimates}
The following proposition is the heart of the estimates needed for the
convergence of the KAM iteration. The proof will be given in the next
(sub)section. Even though we could limit ourselves to symbols in $\J(\rho)$, we consider for the sake of generality and further reference also the general case of symbols belonging to $\J^\dagger(\rho)$. 
\begin{proposition} 
\label{stimeMo}
Let $F$, 
$G\in J^\dagger_k(\rho)$, $k=0,1,\ldots$, $d=d_1+d_2$. Let $\F, \G$ be the corresponding symbols, and $0<d+d_1<\rho$.  Then:
 \begin{enumerate}
 \item[\bf{($1^\dagger$)}] 
 $FG\in J^\dagger_k(\rho)$ and fulfills the estimate
  \vskip 3pt\noindent
 \be
\label{2conv'}
\|FG\|_{\mathcal B(L^2)}\leq \|\F\sharp\G\|^\dagger_{\rho,k}
\leq (k+1)4^k \|\F\|^\dagger_{\rho,k} \cdot \|\G\|^\dagger_{\rho,k}
\ee
 \vskip 3pt\noindent
 \item[\bf{($2^\dagger$)}] 
  $\ds \frac{[F,G]}{i\hbar}\in J^\dagger_k(\rho-d)$ and fulfills the estimate
  \vskip 3pt\noindent
 \bea
 \label{normaM2'}
\left\Vert\frac{[F,G]}{i\hbar}\right\Vert_{\mathcal B(L^2)}\leq 
\|\{\F,\G\}_M\|_{\rho-d-d_1,k}^\dagger \leq \frac{(k+1)4^k}{e^2d_1(d+d_1)}\|\F\|_{\rho,k}^\dagger \|\G\|_{\rho-d,k}^\dagger
\eea
 \vskip 3pt\noindent
 \item[\bf{($3^\dagger$)}]
 $\F\G \in \J^\dagger_k(\rho)$, and
 \be
 \label{simple'}
 \|\F\G\|^\dagger_{\rho,k}
\leq (k+1)4^k  \|\F\|^\dagger_{\rho,k} \cdot \|\G\|^\dagger_{\rho,k}
\ee
 \end{enumerate}
 Moreover if $F$, 
$G\in J_k(\rho)$, $k=0,1,\ldots$, and $\F, \G\in\J_k(\rho)$,   then:
 \begin{enumerate}
 \item[\bf{(1)}] 
  $FG\in J_k(\rho)$ and fulfills the estimate
  \vskip 3pt\noindent
 \be
\label{2conv}
\|FG\|_{\mathcal B(L^2)}\leq \|\F\sharp\G\|_{\rho,k}
\leq (k+1)4^k \|\F\|_{\rho,k} \cdot \|\G\|_{\rho,k}
\ee
 \vskip 3pt\noindent
 \item[\bf{(2)}] 
 $\ds \frac{[F,G]}{i\hbar}\in J_k(\rho-d)$ and fulfills the estimate
  \vskip 5pt\noindent
 \be
 \label{normaM2}
\left\Vert\frac{[F,G]}{i\hbar}\right\Vert_{\mathcal B(L^2)}\leq \|\{\F,\G\}_M\|_{\rho-d-d_1,k}
\leq \frac{(k+1)4^k}{e^2d_1(d+d_1)}\|\F\|_{\rho,k}\cdot \|\G \|_{\rho-d,k}
\ee
 \item[\bf{(3)}]
 $\F \G \in \J_k(\rho)$ and
\be
\label{simple}
 \|\F\G\|_{\rho,k}
\leq (k+1)4^k  \|\F\|_{\rho,k} \cdot \|\G\|_{\rho,k}.
\ee
\end{enumerate}
 \vskip 3pt\noindent
\end{proposition}
\begin{remark}
The operators $F(\hbar)$ with the uniform norm $\|F\|_{\rho,k}, k=0,1,\ldots$ form a Banach subalgebra (without unit) of the algebra of the continuous operators in $L^2(\T^l)$.
\end{remark}
Before turning to the proof we state and prove two further useful results.
\begin{corollary}
\label{multipleM}
Let $\F,\G\in\J_k(\rho)$, and 
let $0<d<\rho$,  $r\in {\Bbb N}$.  Then:
\vskip 4pt\noindent
 \bea
 \label{stimaMr}
\frac{1}{r!}\|\underbrace{\{\F,\{\F,\ldots,\{\F}_{r\ times},{\mathcal G}\}_M\}_M\ldots\}_M\|_{\rho-d,k}
\leq \left(\frac{(k+1)4^k}{ed^2}\right)^r\|\F\|_{\rho,k}^r
\|{\mathcal G}\|_{\rho,k}
\eea
\end{corollary}
\begin{proof} 
We follow the argument of  \cite{BGP},  Lemma 3.5.  If $d+d_1=d_2$, ({3.42}) entails:
\vskip 6pt\noindent
$$
\|\{\F,\G\}_M\|_{\rho-d_2,k}\leq \frac{C_k}{e^2d_2d_1}\|\F\|_{\rho,k}\cdot\|\G\|_{\rho-d,k},\quad C_k:=(k+1)4^k.
$$
\vskip 6pt\noindent
 Set now $\ds d=\frac{r-1}{r}d_2$ which yields $\ds d_1=\frac{d_2}{r}$. Then:
 \vskip 4pt\noindent
\begin{eqnarray}
\label{1}
&&
\|\{\F,\G\}_M\|_{\rho-d_2,k}\leq \frac{C_k}{e^2d_2\frac{d_2}{r}}\|\F\|_{\rho,k}\cdot\|\G\|_{\rho-\frac{r-1}{r}d_2,k}=\frac{C_kr}{(ed_2)^2} \|\F\|_{\rho,k}\cdot\|\G\|_{\rho-\frac{r-1}{r}d_2,k}.
\end{eqnarray}
\vskip 4pt\noindent
Therefore:
\vskip 4pt\noindent
\begin{eqnarray*}
\|\{\F,\{\F,\G\}_M\}_M\|_{\rho-d_2,k}\leq \frac{C_kr}{(ed_2)^2} \|\F\|_{\rho,k}\cdot\|\{\F,\G\}_M\|_{\rho-\frac{r-1}{r}d_2,k}.
\end{eqnarray*}
\vskip 4pt\noindent
To estimate $\|\{\F,\G\}_M\|_{\rho-\frac{r-1}{r}d_2,k}$ we repeat the argument yielding \eqref{1} with $\ds \frac{r-1}{r}d_2$ in place of $d_2$.  We get:
$$
 \frac{r-1}{r}d_2= \frac{r-2}{r}d_2+ \frac{1}{r}d_2
 $$
 and therefore
 \vskip 4pt\noindent
$$
\|\{\F,\G\}_M\|_{\rho-\frac{r-1}{r}d_2,k}\leq \frac{C_k}{ed_2(\frac{r-1}{r})\frac{d_2}{r}}\|\F\|_{\rho,k}\cdot\|\G\|_{\rho 
-\frac{r-2}{r}d_2,k}
$$
$$
\leq  \frac{C_kr}{(ed_2)^2}\left(\frac{r}{r-1}\right) \|\F\|_{\rho,k}\cdot \Vert \G\Vert_{\rho-\frac{r-2}{r}d_2,k}
$$
\vskip 6pt\noindent
whence
\begin{eqnarray*}
&&
\|\{\F,\{\F,\G\}_M\}_M\|_{\rho-d_2,k}\leq \frac{(C_kr)^2}{(ed_2)^4}\left(\frac{r}{r-1}\right) \|\F\|_{\rho,k}^2\cdot \Vert \G\Vert_{\rho-\frac{r-2}{r}d_2,k}.
\end{eqnarray*}
Iterating $r$ times we get:
\vskip 6pt\noindent
\be
\label{2}
\frac1{r!}\|\underbrace{\{\F,\{\F,\cdots,\{\F}_{r\ times},\G\}_M\}_M,\cdots\}_M\|_{\rho-d_2,k}\leq \frac{(C_kr)^{r}r^{2r+1}}
{(ed_2)^{2r}r!^2} \|\F\|^r_{\rho,k}\cdot\|\G\|_{\rho,k}.
\ee
\vskip 6pt\noindent
By the Stirling formula:
\vskip 6pt\noindent
\begin{eqnarray*}
\frac{r^{2r+1}}
{(ed_2)^{2r}r!^2}\leq \frac1{(ed_2^2)^{r}}
\frac 1{\sqrt{2\pi }} \frac1{(d_2^2)^{r}} \leq \frac1{(d_2^2)^{r}}.
\end{eqnarray*}
\vskip 6pt\noindent
Since $C_k=(k+1)4^k$, \eqref{2} yields \eqref{stimaMr} up to the abuse of notation $d_2=d$.
\end{proof}
\begin{corollary}
\label{stimaP}
Let $\F(\xi;x;\hbar)\in\J_k(\rho)$, $\rho>0$, $k=0,1,\ldots$. Then $\{\F,\L_\om\}_M\in\J_k(\rho-d)$ $\forall\,0<d<\rho$ and the following estimates hold:
\be
\label{stimapp}
\|[F,L_\om]/i\hbar\|_{\rho-d,k}= \|\{\F,\L_\om\}_M\|_{\rho-d,k}
\leq \frac{1}{ed}\|\F\|_{\rho,k}
\ee
\bea
\label{stimaMpp}
&&
\|\underbrace{[F,[\cdots,[F}_{r\ times},L_\om]\cdots]/(i\hbar)^r\|_{\rho-d,k}= \|\{\F,\cdots,\{\F,\L_\om\}_M\cdots\}_M\|_{\rho-d,k}
\\
&&
\nonumber
\\
&&
\nonumber
\leq \frac1{ed}\left(\frac{(k+1)4^k}{ed^2}\right)^r\|\F\|_{\rho,k}^r
\eea
\end{corollary}
\begin{proof}
By (\ref{MP}):
$$
\{\F,\L_\om\}_M=\{\F,\L_\om\}=-\la \om,\nabla_x\ra\F(\xi,x;\hbar)=\sum_{q\in\Z^l}\la\om,q\ra e^{i\la q,x\ra}\int_{\R}\widehat{\F}_q(p;\hbar)e^{ip\L_\om(\xi)}\,dp
$$
and therefore:
\begin{eqnarray*}
&&
 \|\{\F,\L_\om\}_M\|_{\rho-d,k}\leq \|\{\F,\L_\om\}\|_{\rho-d,k}\leq \sum_{q\in\Z^l}|\la\om,q\ra|e^{(\rho-d)|q|}\|\F_q\|_{\rho,k} \leq
 \\
 &&
 \sup_{q\in\Z^l}\la\om,q\ra|e^{-d|q|}\sum_{q\in\Z^l}e^{\rho|q|}\|\F_q\|_{\rho,k} \leq \frac{1}{ed}\|\F\|_{\rho,k}
\end{eqnarray*}
because $|\om|\leq 1$ by Remark 2.6. This proves (\ref{stimapp}). (\ref{stimaMpp}) is a direct consequence of Corollary \ref{multipleM}. 
\end{proof}
\subsection{Proof of Proposition \ref{stimeMo}} 
\subsubsection{Three lemmata}
\label{2l}
The proof will use the three following Lemmata.
\begin{lemma}
\label{symp}
Let $p,p'\in\R^{l},\ s,s^\prime\in\R^l$. Define $t:=(p,s), t^\prime:=(p^\prime,s^\prime)$.  Let $\Om_\om(\cdot)$ and $\mu_j(\cdot)$ be defined by (\ref{omt}) and (\ref{muk}), respectively. Then:
\be
\vert\Omega_\om(t,t^\prime)\vert^j\leq 2^j\mu_j(t)\mu_j(t^\prime).
\ee
\end{lemma}
The proof is straightforward, because  $\vert\Omega_\omega(t,t^\prime)\vert\leq 2\vert t\vert\vert t^\prime\vert$ and $|\om|\leq 1$.
\begin{lemma}\label{sin}
\be
\left\vert\frac{d^m}{d\hbar^m}\frac{\sin{\hbar x/2}}\hbar\right\vert\leq \frac{\vert x\vert^{m+1}}{2^{m+1}}.
\ee. 
\end{lemma}
\begin{proof}
Write:
\vskip 6pt\noindent
\begin{eqnarray*}
\frac{d^m}{d\hbar^m}\frac{1}{\hbar}\sin{\hbar x/2}
=\frac{d^m}{d\hbar^m}\frac12\int_0^x\cos{\hbar t/2}\,dt
=\frac{(-\hbar)^m}{2^{m+1}}\int_0^xt^m\cos^{(m)}{(\hbar t/2)}\,dt
\leq\frac{\hbar^m}{2^{m+1}}\int_0^xt^m\,dt. 
\end{eqnarray*}
\vskip 6pt\noindent
whence
$$
\left\vert\frac{d^m}{d\hbar^m}\frac{\sin{\hbar x/2}}\hbar\right\vert\leq \frac{\hbar^m}{2^{m+1}}\left\vert \int_0^xt^m\,dt\right\vert =\frac{\hbar^m\vert x\vert^{m+1}}{2^{m+1}(m+1)}\leq \frac{\vert x\vert^{m+1}}{2^{m+1}}.
$$
\end{proof}
\begin{lemma}
\label{MoyalS}
Let $(\F,G)\in\J^\dagger_\rho$, $0<d+d_1<\rho$, $t=(p,s)$, $t^\prime=(p^\prime,s^\prime)$, $|t|:=|p|+|s|$, $|t^\prime|:=|p^\prime|+|s^\prime|$. Then:
\be
\|\{\F,\G\}_M\|_{\rho-d-d_1}^\dagger \leq \frac{1}{e^2d_1(d+d_1)}\|\F\|_\rho^\dagger \|\G\|_{\rho-d}^\dagger
\ee
\end{lemma}
\begin{proof}
We have by definition
\begin{eqnarray*}
&&
|\{\F,\G\}_M\|^\dagger_{\rho-d-d_1}\leq
 \frac{1}{\hbar}\int_{\R^{2l}}e^{(\rho-d-d_1)|t|}d\lambda(t)\int_{\R^{2l}}|\F(t^\prime)\G(t^\prime-t)|\cdot |\sin{\hbar(t^\prime-t)\wedge t^\prime/\hbar}|\,d\lambda(t^\prime)
\\
&&
\leq  \int_{\R^{2l}}e^{(\rho-d-d_1)|t|}d\lambda(t)\int_{\R^{2l}}|\F(t^\prime)|\cdot |\G(t^\prime-t)|\cdot |(t^\prime-t)|\cdot |t^\prime|\,d\lambda(t^\prime)
\\
&&
=\int_{\R^{2l}}e^{(\rho-d-d_1)|t|}d\lambda(t)\int_{\R^{2l}}|\F(u+t/2)\G(u-t/2)|\cdot |u-t/2|\cdot |u+t/2|\,d\lambda(u)
\\
&&
=\int_{\R^{2l}\times\R^{2l}}e^{(\rho-d-d_1)(|x|+|y|)}|\F(x)\G(y)|\cdot |x|\cdot |y|\,d\lambda(x)d\lambda(y) \leq 
\\
&&
\frac{1}{d_1(d+d_1)}\int_{\R^{2l}}|\F(x)|e^{\rho |x|}\,d\lambda(x) \int_{\R^{2l}}|\G(y)|e^{(\rho-d) |y|}\,d\lambda(x)\leq \frac{1}{e^2d_1(d+d_1)}\|\F\|_\rho^\dagger\|\G\|_{\rho-d}^\dagger 
\end{eqnarray*}
because $\ds \sup_{\alpha\in\R}|\alpha| e^{-\delta\alpha}=\frac{1}{e\delta}, \delta>0$. 
\end{proof}
\subsubsection{ Assertion {\mbox{\bf ($1^\dagger$)}}}\label{1'}
By definition
\begin{eqnarray}
&&
\|\F(\hbar)\sharp\G(\hbar)\|_{\rho,k}^\dagger=
\nonumber
\sum_{\gamma=0}^k\int_{\R^{2l}\times\R^{2l}}|
\partial^\gamma_\hbar [\widehat{\F}(t^\prime-t,\hbar) 
\widehat{\G}(t^\prime,\hbar)e^{i\hbar\Om_\om(t^\prime,t^\prime-t)}] |\mu_{k-\gamma}(t)e^{\rho |t|}\,d\lambda(t^\prime)
d\lambda(t)
\nonumber
\end{eqnarray}
whence
\begin{eqnarray}
&&
\|\F(\hbar)\sharp\G(\hbar)\|^\dagger_{\rho,k}=
\nonumber
\\
&&
\sum_{\gamma=0}^k\sum_{j=0}^\gamma\binom {\gamma}{j}\int_{\R^{2l}\times \R^{2l} }\vert\partial_\hbar^{\gamma-j}
 [\widehat{\F}(t^\prime-t,\hbar) 
\widehat{\G}(t^\prime,\hbar)]\vert\Omega_\om(t^\prime-t,t^\prime)\vert^j \mu_{k-\gamma}(t)e^{\rho |t|}\,d\lambda(t^\prime)
d\lambda(t)=
\nonumber
\\
&&
\nonumber
\sum_{\gamma=0}^k\sum_{j=0}^\gamma\sum_{i=0}^{\gamma-j}\binom {\gamma}{j}\binom{j}{i}\int_{\R^{2l}\times\R^{2l}}
\vert\partial_\hbar^{\gamma-j-i}\widehat{\F}(t^\prime-t,\hbar)
\partial_\hbar^{i}\widehat{\G}(t^\prime,\hbar)
\vert\vert\Omega_\om(t^\prime-t,t^\prime)\vert^j\mu_{k-\gamma}(t)e^{\rho|t|}\,d\lambda(t^\prime)
d\lambda(t)
\end{eqnarray}
By Lemma \ref{symp} and the inequality $\ds \mu_k(t^\prime-t)\leq 2^{k/2}\mu_k(t^\prime)\mu_k(t)$ we get, with $t=(p,s): t^\prime=(p^\prime,s^\prime)$
\begin{eqnarray*}
&&
\vert\Omega_\om(t^\prime-t,t^\prime)\vert^j\mu_{k-\gamma}(t)\leq 2^j\mu_j(t^\prime-t)\mu_j(t^\prime)\mu_{k-\gamma}(t)
\\
&&
\leq 
2^j\mu_jt^\prime-t)\mu_j(t^\prime)\mu_{k-\gamma}(t)2^{(k-\gamma)/2}\mu_{k-\gamma}(t^\prime -t)\mu_{k-\gamma}(t)
\\
&&
\leq 2^{j+(k-\gamma)/2}\mu_{k-\gamma+j}(t^\prime -t)\mu_{k-\gamma+j}(t)
\end{eqnarray*}
Denote now $\gamma-j-i=k-\gamma^\prime$, $i=k-\gamma^{\prime\prime}$ and remark that $j\leq\gamma^\prime$,  $i\leq\gamma-j$. Then:
\begin{eqnarray*}
2^{j+(k-\gamma)/2}\mu_{k-\gamma+j}(t^\prime -t)\mu_{k-\gamma+j}(t)
\leq 2^k\mu_{\gamma^\prime}(t^\prime)\mu_{\gamma^{\prime\prime}}(t)
\end{eqnarray*}
 Since $\ds \binom {\gamma}{j}\binom{j}{i}\leq 4^k$ and the sum over $k$ has $(k+1)$ terms we get:
\begin{eqnarray*}
&&
\|\F(\hbar)\sharp\G(\hbar)\|^\dagger_{\rho,k} \leq 
\\
&&
(k+1)4^k\,\sum_{\gamma^\prime,\gamma^{\prime\prime}=0}^k\int_{\R^{2l}\times\R^{2l}}
|\partial^{k-\gamma^\prime}_\hbar\widehat{\F}(t^\prime -t,\hbar)
|\partial^{k-\gamma^{\prime\prime}}_\hbar\widehat{\G}(t^\prime,\hbar)| 
\mu_{\gamma^\prime}(t^\prime -t)\mu_{\gamma^{\prime\prime}}(t)e^{\rho |t|}\,d\lambda(t^\prime)
d\lambda(t)
\end{eqnarray*}
Now we can repeat the argument of Lemma \ref{MoyalS} to conclude: 
\begin{eqnarray*}
\|\F(\hbar)\sharp\G(\hbar)\|_{\rho,k}^\dagger \leq 
(k+1)4^k \|\F\|^\dagger_{\rho,k} \cdot \|\G\|^\dagger_{\rho,k}
\end{eqnarray*}
which is (\ref{2conv'}). Assertion {\mbox{\bf ($3^\dagger$)}}, formula (\ref{simple'})  is the  particular case of  (\ref{2conv'}) obtained for $\Om_\om=0$, and  Assertion ${\bf (3)}$, formula (\ref{simple}), is  in turn particular case of (\ref{simple'}) .

\subsubsection{ Assertion{\mbox{\bf ($2^\dagger$)}}}\label{2'}
 By definition:
\begin{eqnarray*}
\|\{\F(\hbar),\G(\hbar)\}_M\|^\dagger_{\rho,k}=
\sum_{\gamma=0}^k\int_{\R^{2l}\times\R^{2l}}|
\partial^\gamma_\hbar [\widehat{\F}(t^\prime -t,\hbar) 
\widehat{\G}(t^\prime,\hbar)\sin\hbar\Omega(t^\prime-t,t^\prime)/\hbar] |\mu_{k-\gamma}(t)e^{\rho |t|}\,d\lambda(t^\prime)
d\lambda(t).
\end{eqnarray*}
  Lemma \ref{sin} entails:
  $$
\vert\partial_\hbar^j \sin\hbar\Omega(t^\prime-t,t^\prime)/\hbar\vert\leq \vert \Omega(t^\prime-t,t^\prime)\vert^{j+1}
$$
and therefore:
\begin{eqnarray}
&&
\|\{\F(\hbar),\G(\hbar)\}_M\|_{\rho,k}\leq 
\nonumber
\\
&&
\sum_{\gamma=0}^k\sum_{j=0}^\gamma\binom {\gamma}{j}\int_{\R^{2l}\times \R^{2l} }\vert\partial_\hbar^{\gamma-j}
 [\widehat{\F}(t^\prime -t,\hbar) 
\widehat{\G}(t^\prime,\hbar)]\vert\Omega_\om(t^\prime-t,t^\prime)\vert^{j+1} \mu_{k-\gamma}(t)e^{\rho(|t|}\,d\lambda(t^\prime)
d\lambda(t)=
\nonumber
\\
&&
\nonumber
\sum_{\gamma=0}^k\sum_{j=0}^\gamma\sum_{i=0}^{\gamma-j}\binom {\gamma}{j}\binom{j}{i}\int_{\R^{2l}\times\R^{2l}}
\vert\partial_\hbar^{\gamma-j-i}\widehat{\F}(t^\prime -t,\hbar)
\partial_\hbar^{i}\widehat{\G}(t^\prime,\hbar)
\vert\vert\Omega_\om(t^\prime-t,t^\prime)\vert^{j+1}\mu_{k-\gamma}(t)e^{\rho |t|}\,d\lambda(t^\prime)
d\lambda(t)
\end{eqnarray}
Let us now absorb a factor  $\vert\Omega_\om(t^\prime-t,t^\prime)\vert^{j}$ in exactly the same way as above, and recall that $\vert\Omega_\om(t^\prime-t,t^\prime)\vert\leq \vert (t^\prime-t)t^\prime\vert$. We end up with the inequality:
\begin{eqnarray*}
&&
\|\{\F(\hbar),\G(\hbar)\}_M\|^\dagger_{\rho,k} \leq 
\\
&&
(k+1)4^k\,\sum_{\gamma^\prime,\gamma^{\prime\prime}=0}^k\int_{\R^{2l}\times\R^{2l}}
|\partial^{k-\gamma'}_\hbar\widehat{\F}(t^\prime -t,\hbar)
|\partial^{k-\gamma"}_\hbar\widehat{\G}(t^\prime,\hbar)| |t^\prime -t||t^\prime| 
 \mu_{\gamma^\prime}(t^\prime -t)\mu_{\gamma^{\prime\prime}}(t^\prime)e^{\rho( |t|}\,d\lambda(t^\prime)
d\lambda(t)
\end{eqnarray*}
Repeating once again the argument of Lemma \ref{MoyalS} we finally get:
\begin{eqnarray*}
\|\{\F(\hbar),\G(\hbar)\}_M\|^\dagger_{\rho-d-d_1,k} \leq 
\frac{(k+1)4^k}{e^2d_1(d+d_1)} \|\F\|^\dagger_{\rho,k} \cdot \|\G\|^\dagger_{\rho-d,k}
\end{eqnarray*}
which is (\ref{normaM2'}). Once more, Assertion ${\bf (2)}$ is a particular case of  (\ref{normaM2'}) and Assertion ${\bf (1)}$ a particular case of (\ref{2conv'}). This completes the proof of Proposition \ref{stimeMo}. 

 \vskip 1cm

\section{A sharper version of the semiclassical Egorov theorem}\label{sectionegorov}
Let us state and prove in this section a particular variant of the semiclassical Egorov theorem (see e.g.\cite{Ro}) which establishes the relation between the unitary transformation $\ds e^{i\ep W/i\hbar}$ and the canoni\-cal transformation $\phi^\ep_{\W_0}$ generated by the flow of the symbol $\W(\xi,x;\hbar)|_{\hbar=0}:=\W_0(\xi,x)$ (principal symbol) of $W$ at time $1$.  The present version is sharper in the sense that the usual one allows for a $O(\hbar^\infty)$ error term. 
\begin{theorem}
Let  $\rho>0, k=0,1,\ldots$ and let $A,W\in J^\dagger_k(\rho)$ with symbols $\mathcal A,\
\mathcal W$. Then: 
\be
\nonumber
S_\ep:=e^{i\frac {\ep W}\hbar}(L_\omega+A)e^{-i\frac {\ep W}\hbar}=L_\omega+B 
\ee
where:
\begin{enumerate}
\item $\forall\,0<d<\rho$, 
$B\in J^\dagger_k(\rho-d)$;
\item
\begin{eqnarray*}
\|\B\|^\dagger_{\rho-d,k}\leq\frac{|\ep|(k+1)4^k \|\W\|_{\rho,k}}{ed^2}\left[1-|\ep|(k+1)4^k\|\W\|_{\rho,k}/{ed^2}\right]^{-1}\left[\|\A\|_{\rho,k}+1/{de}\right]\end{eqnarray*}
\item  Moreover the symbol $\mathcal B$ of $B$ is such that:
$$
\L_\om+{\mathcal B}=(\L_\om+\mathcal A)\circ \Phi^\ep_{\mathcal W_0}+O(\hbar)
$$
where $\Phi^\ep_{\mathcal W_0}$ is the Hamiltonian flow of $\mathcal W_0:=\mathcal W|_{\hbar=0}$ at  time $\ep$.
\item Assertions (1), (2), (3) hold true when $(A,B,W)\in J_k(\rho)$ with $\|\A\|^\dagger_{\rho,k}$,  $\|\B\|^\dagger_{\rho,k}$, $\|\W\|^\dagger_{\rho,k}$ replaced by $\|\A\|_{\rho,k}$,  $\|\B\|_{\rho,k}$, $\|\W\|_{\rho,k}$. 
\end{enumerate}
\end{theorem}
\begin{proof}The proof is  the same in both cases,  since it it is based only on Proposition \ref{stimeMo}. Therefore we limit ourselves to the $\J_k(\rho)$ case.
 
 By Corollary \ref{corA}, Assertion (3), under the present assumptions $H^1(\T^l)$, the domain of the self-adjoint operator $\F(L_\om)+A$, is left invariant by the unitary operator $\ds e^{i\frac {\ep W}{\hbar}}$.  Therefore on  $H^1(\T^l)$  we can write the commutator expansion
$$
S_\ep=L_\om+\sum_{m=1}^\infty \frac{(i\ep)^m}{ \hbar^m m!}\underbrace{[W,[W,\ldots,[W}_{m\ times},L_\om]\ldots]+
\sum_{m=1}^\infty \frac{(i\ep)^m}{ \hbar^m m!}\underbrace{[W,[W,\ldots,[W}_{m\ times},A]\ldots]
$$
whence  the corresponding expansions for the  symbols (from now on we'll skip the $\underbrace{\ldots\ldots\ldots}_{m\ times}$ notation)
\begin{eqnarray*}
&&
{\mathcal S}(x,\xi;\hbar,\ep)=\L_\om(\xi)+\sum_{m=1}^\infty \frac{\ep^m}{m!}\{\W,\{\W,\ldots,\{\W,\L_\om\}_M\ldots\}_M
\\
&&
+\sum_{m=1}^\infty \frac{\ep^m}{m!}\{\W,\{\W,\ldots,\{\W,{\mathcal A}\}_M\ldots\}_M
\end{eqnarray*}
because  $\{\W,\L_\om\}_M=\{\W,\L_\om\}$ by the linearity of $\L_\om$.  Now apply Corollaries \ref{multipleM} and \ref{stimaP}. We get, denoting once again $C_k=(k+1)4^k$:
\begin{eqnarray*}
&&
\|\sum_{m=1}^\infty \frac{(i\ep)^m}{ \hbar^m m!}[W,[W,\ldots,[W,L_\om]\ldots]\|_{L^2\to L^2}\leq \|\sum_{m=1}^\infty \frac{\ep^m}{m!}\{\W,\{\W,\ldots,\{\W,\L_\om\}_M\ldots\}_M\|_{\rho-d,k}
\\
&&
\leq \sum_{m=1}^\infty \frac{|\ep|^m}{m!}\|\{\W,\{\W,\ldots,\{-i\la\om,\nabla_x\ra \W\}_M\ldots\}_M\|_{\rho-d,k}\leq \frac{1}{ed}\sum_{m=1}^\infty \left(\frac{|\ep|C_k\|\W\|_{\rho,k}
}{ed^2}\right)^m
\end{eqnarray*}
\begin{eqnarray*}
&&
\|\sum_{m=1}^\infty \frac{(i\ep)^m}{ \hbar^m m!}[W,[W,\ldots,[W,A]\ldots]
\|_{L^2\to L^2} \leq \|\sum_{m=1}^\infty \frac{\ep^m}{m!}\{\W,\{\W,\ldots,\{\W,{\mathcal A}\}_M\ldots\}_M\|_{\rho-d,k}
\\
&&
\leq \|\A\|_{\rho,k}\sum_{m=1}^\infty \left(\frac{|\ep|C_k\|\W\|_{\rho,k}
}{ed^2}\right)^m
\end{eqnarray*}
Now define:
\be
\label{Aprimo}
B:=\sum_{m=1}^\infty \frac{(i\ep)^m}{ \hbar^m m!}[W,[W,\ldots,[W,\L_\om]\ldots]+\sum_{m=1}^\infty \frac{(i\ep)^m}{ \hbar^m m!}[W,[W,\ldots,[W,A]\ldots].
\ee
Then  we can write:
\vskip 4pt\noindent
\begin{eqnarray*}
&&
\|\B\|_{\rho-d,k}\leq\frac{|\ep|C_k\|\W\|_{\rho,k}}{ed^2}\left[1-|\ep|C_k\|\W\|_{\rho,k}/{ed^2}\right]^{-1}\left[\|\A\|_{\rho,k}+1/{de}\right]
\\
&&
=\frac{|\ep|(k+1)4^k \|\W\|_{\rho,k}}{ed^2}\left[1-|\ep|(k+1)4^k\|\W\|_{\rho,k}/{ed^2}\right]^{-1}\left[\|\A\|_{\rho,k}+1/{de}\right]
\end{eqnarray*}
\vskip 4pt\noindent
This proves assertions (1) and (2). 
\newline
By Remark 2.9, we have:
\begin{eqnarray*}
&&
{\mathcal S}^0_\ep(x,\xi;\hbar)|_{\hbar=0}=\L_\om+\B_\ep(\xi,x;\hbar)|_{\hbar=0}=
\\
&&
\sum_{k=0}^\infty \frac{(\ep)^k}{k!}\{\W_0,\{\W,\ldots,\{\W_0,{\L+\A}\}\ldots\}=e^{\ep \L_{\W_0}}(\L_\om+\A)
\end{eqnarray*}
where $\L_{\W_0}\F=\{\W,\F\}$ denote the Lie derivative with respect to the Hamiltonian flow generated by $\W_0$. Now, by Taylor's theorem
$$
e^{\ep \L_{\W_0}}(\L_\om+\A)=(\L_\om+\A)\circ \phi^\ep_{\W_0}(x,\xi)
$$
and this concludes the proof of the Theorem.
\end{proof} 
\begin{remark}
Let $W$ be a solution of the homological equation (\ref{heq}). Then the  explicit expression of $\W_0$ clearly is:
$$
\W_0=\frac1{\F^\prime(\L_\om(\xi))}\sum_{q\in\Z^\ell}\frac{\V_q (\xi)}{\la \om,q\ra}e^{i\la q,x\ra}
$$
and 
$$
e^{\ep \L_{\W_0}}(\F(\L_\om)+\ep\A)=\F(L_\om)+\ep \N_{0,\ep}(\L_\om)+O(\ep^2).
$$
\end{remark}
Thus $\W_0$ coincides with the expression obtained by first order canonical perturbation theory.  
\vskip 1cm\noindent
\section{Homological equation:  solution and estimate}
\renewcommand{\thetheorem}{\thesection.\arabic{theorem}}
\renewcommand{\theproposition}{\thesection.\arabic{proposition}}
\renewcommand{\thelemma}{\thesection.\arabic{lemma}}
\renewcommand{\thedefinition}{\thesection.\arabic{definition}}
\renewcommand{\thecorollary}{\thesection.\arabic{corollary}}
\renewcommand{\theequation}{\thesection.\arabic{equation}}
\renewcommand{\theremark}{\thesection.\arabic{remark}}
\setcounter{equation}{0}%
\setcounter{theorem}{0}%
Let us briefly recall the well known KAM iteration in the quantum context. \par
The first step consists in looking for an $L^2(\T^l)$-unitary map $U_{0,\ep}=e^{i\ep W_0/\hbar}$, $W_0=W_0^\ast$, such that 
$$
S_{0,\ep}:=U_{0,\ep}(L_\om+\ep V_0)U_{0,\ep}^\ast=\F_{1,\ep}(L_\om)+\ep^2 V_{1,\ep}, \quad V_0:=V, \quad \F_{1,\ep}(L_\om)=L_\om+\ep N_0(L_\om).
$$
 Expanding to first order near $\ep=0$ we get that the two unknowns  $W_0$ and $N_0$ must solve the equation
$$
\frac{[L_\om,W_0]}{i\hbar}+V=N_0
$$
$V_{1,\ep}$ is the second order remainder of the expansion. Iterating the procedure:
\vskip 3pt\noindent
\begin{eqnarray*}
&&
U_{\ell,\ep}:= e^{i\ep^{2^\ell}W_\ell/\hbar}; 
\\
&&
 S_{\ell,\ep}:=U_{\ell.\ep}(\F_{\ell,\ep}(L_\om)+\ep^{2^{\ell}} V_{\ell,\ep})U_{\ell,\ep}^\ast= 
= \F_{\ell+1,\ep}(L_\om)+\ep^{2^{\ell+1}} V_{\ell+1}(\ep), 
\\
&&
\frac{[\F_{\ell,\ep}(L_\om),W_{\ell,\ep}]}{i\hbar}+V_{\ell,\ep}=N_{\ell,\ep}
\end{eqnarray*}
\vskip 3pt\noindent
With abuse of notation, we denote by $\F_{\ell,\ep}(\L_\om,\hbar)$, ${\mathcal N}_{\ell,\ep}(\L_\om,\hbar)$, $\V_{\ell,\ep}(\L_\om,\hbar)$ the corresponding symbols.
\newline
The KAM iteration procedure  requires therefore the solution in $J_k(\rho)$ of the operator homological equation in the two unknowns $W$ and $M$ (here we have dropped the dependence on $\ell$ and $\ep$, and changed the notation from $N$ to $M$ to avoid confusion with what follows):
\be
\label{heq}
\frac{[\F(L_\om),W]}{i\hbar}+V=M(L_\om)
\ee
with  the requirement $M(L_\om)\in J_k(\rho)$; the solution has to be expressed in terms of the corresponding Weyl symbols $(\L_\om, \W, \V, {\mathcal M})\in\J_k(\rho)$ in order to obtain estimates uniform with respect to $\hbar$.  Moreover, the remainder has to be estimated in terms of the estimates for $W, M$. 
\newline
 Equation (\ref{heq}), written for the symbols, becomes
\be
\label{Mo}
\{\F(\L_\om(\xi),\hbar),\W(x,\xi;\hbar)\}_M+\V(x,L_\om(\xi);\hbar)={\mathcal M}(\L_\om(\xi),\hbar)
\ee
\subsection{The homological equation}\label{hom}
We will construct and estimate the solution of (\ref{heq}), actually solving (\ref{Mo}) and estimating its solution,  under the following assumptions on $\F$: 
 \vskip 5pt\noindent
{\textbf{Condition (1)}}
{\it
 $(u,\hbar)\mapsto \F(u;\hbar)\in C^\infty(\R\times [0,1]; \R)$;}
\vskip 4pt\noindent
{\textbf{Condition (2)}}
 $$ 
 \inf_{(u,\hbar)\in\R\times [0,1]}\partial_u\F(u;\hbar)>0; 
 \quad \lim_{|u|\to \infty}\frac{|\F(u,\hbar)|}{|u|}=C>0
 $$  
 {\it uniformly with respect to $\hbar\in [0,1]$.} 
\vskip 5pt\noindent
{\textbf{Condition (3)}}
{\it 
Set: 
\be
\label{Kappa}
\K_\F(u,\eta,\hbar)=\frac{\eta}{\F(u+\eta,\hbar)-\F(u,\hbar)}
\ee 
Then there is $0<\Lambda(\F)<+\infty$ such that}
\be
\label{KB}
\sup_{u\in\R,\eta\in\R,\hbar\in [0,1]}\vert\K_\F(u,\eta,\hbar)\vert<\Lambda.
\ee

\vskip 5pt\noindent
The first result deals with the identification of the operators $W$ and $M$  through the  determination of their matrix elements and  corresponding  symbols $\W$ and $\M$. 
\begin{proposition}
\label{WN}
Let $V\in J(\rho)$, $\rho>0$, and let $W$ and $M$ be the minimal closed operators in $L^2(\T^n)$ generated by the infinite matrices 
\vskip 5pt\noindent
\be
\label{sheq1}
 \la e_m,We_{m+q}\ra =\frac{i\hbar\la e_m,Ve_{m+q}\ra}{\F(\la \om,m\ra\hbar,\hbar)-\F(\la\om,(m+q)\ra\hbar,\hbar)},\quad q\neq 0, \quad  \la e_m,We_m\ra=0
\ee
\vskip 6pt\noindent
\be
 \la e_m,Me_m\ra=\la e_m,Ve_m\ra,\qquad \la e_m,Me_{m+q}\ra=0, \quad q\neq 0
\label{sheq2}
\ee
on the eigenvector basis $e_m: m\in\Z^l$ of $L_\om$. Then:
\begin{enumerate}
 \item 
 $W$ and $M$ are continuous and solve the homological equation (\ref{heq}); 
 \item The symbols $\W(x,\xi;\hbar)$ and $\M(\xi,\hbar)$ have the expression:
 \bea
 \label{defW}
 &&
 \M(\xi;\hbar)=\overline{\truc{\V}}(\L_\om(\xi);\hbar);\quad \W(\L_\om(\xi),x;\hbar)=\sum_{q\in\Z^l,q\neq 0}\truc{\W}(\L_\om(\xi),q;\hbar)e^{i\la q,x\ra}
 \\
 &&
 \truc{\W}(\L_\om(\xi),q;\hbar):=\frac{i\hbar\truc{\V}(\L_\om(\xi);q;\hbar)}{\F(\L_\om(\xi);\hbar)-\F(\L_\om(\xi+q),\hbar)}, \;q\neq 0; \quad \overline{\truc{\W}}(\L_\om(\xi);\hbar)=0.
 \eea
 \vskip 4pt\noindent
 Here the series in (\ref{defW})  is $\|\cdot\|_\rho$ convergent; $\overline{\truc{\V}}(\L_\om(\xi);\hbar)$ is the $0$-th coefficient in the Fourier expansion of $\V(\L_\om(\xi),x,\hbar)$:
 $$
\V(\L_\om(\xi),x,\hbar)=\sum_{q\in\Z^l}\,{\truc{\V}}(\L_\om(\xi),q;\hbar)e^{i\la q,x\ra}.
$$
  \end{enumerate}
\end{proposition} 
\begin{proof}
Writing the homological equation in the eigenvector basis $e_m: m\in\Z^l$ we get
\vskip 7pt\noindent
\be
\label{mheq}
\la e_m,\frac{[\F(L_\om),W]}{i\hbar}e_n\ra+\la e_m,Ve_n\ra=\la e_m,M(L_\om)e_n\ra\delta_{m,n}
\ee
\vskip 5pt\noindent
which immediately yields (\ref{sheq1},\ref{sheq2}) setting $n=m+q$. As far the continuity is concerned, we have:
\vskip 6pt\noindent
$$
\frac{i\hbar}{\F(\la \om,m\ra\hbar,\hbar)-\F(\la\om,(m+q)\ra\hbar,\hbar)}=\la\om,q\ra^{-1}\frac{\eta}
{\F(\la \om,m\ra\hbar,\hbar)-\F(\la\om,m\ra\hbar+\eta,\hbar)},\quad \eta:=\la q,\om\ra\hbar.
$$
\vskip 7pt\noindent
and therefore, by (\ref{KB}) and the diophantine condition:
$$
|\la e_m,We_{m+q}\ra|\leq \gamma |q|^\tau\Lambda |\la e_m,Ve_{m+q}\ra|.
$$
The assertion now follows by Corollary \ref{corA}, which also entails the $\|\cdot\|_\rho$ convergence of the series (\ref{defW}) because $\V\in \J_\rho$. Finally, again by Corollary \ref{corA}, formulae \eqref{elm4}, \eqref{elm5}, we can write
$$
\la e_m,We_{m+q}\ra= \truc{\W}(\la \om,(m+q/2)\ra\hbar,q,\hbar); \quad
\la e_m,Me_m\ra=\M(\la\om,m\ra\hbar,\hbar)=\truc{\V}(\L_\om(m\hbar),0,\hbar)
$$
and this concludes the proof of the Proposition.
\end{proof}

The basic example of $\F$ is the following one. 
Let:
 \bea
\label{FNl}
&&
\bullet \qquad  \F_{\ell}(u,\ep;\hbar)=u+\Phi_{\ell}(u,\ep,\hbar),\qquad \ell=0,1,2,\ldots
\\
&&
\bullet \qquad \Phi_{\ell}(\ep,\hbar):=\ep\N_{0}(u;\ep,\hbar)+\ep^2\N_{1}(u;\ep,\hbar)+\ldots+\ep_{\ell}\N_{\ell}(u,\ep,\hbar), \quad \ep_{j}:=\ep^{2^{j}}. 
\eea
  where we assume holomorphy of $\ep\mapsto \N_s(u,\ep,\hbar)$ in the unit disk and the existence of $\rho_0>\rho_1>\ldots>\rho_{\ell}>0$ such that:
  \begin{itemize}
  \item[($N_s$)] 
   $\ds\qquad\qquad\qquad\qquad\quad \max_{|\ep|\leq 1} \vert\N\vert_{\rho_s}<\infty, \qquad 
.$ 
\end{itemize}
Denote, for $\zeta\in\R$:
\vskip 6pt\noindent
\be
\label{gl}
g_\ell(u,\zeta;\ep,\hbar):=\frac{\Phi_{\ell-1}(u+\zeta;\ep,\hbar)-\Phi_{\ell-1}(u;\ep,\hbar)}{\zeta}
\ee
\vskip 6pt\noindent

Let furthermore:
\bea
\label{ddll}
&&
 0<d_{\ell}<\ldots<d_0<\rho_0, \quad 0<\rho_0:=\rho; 
\\
&&
\nonumber
  \rho_{s+1}=\rho_s-d_{s}>0, \;s=0,\ldots,\ell-1
 \\
 &&
  \delta_\ell:=\sum_{s=0}^{\ell-1}d_\ell  <\rho
 \eea
   and set, for $j=1,2,\ldots$:
\bea
\label{theta}
&&
 \theta_{\ell,k}(\N,\ep):=\sum_{s=0}^{\ell-1}\frac{|\ep_s|\,|\N_s|_{\rho_s,k}}{ed_{s}}, \qquad \theta_{\ell}(\N,\ep):=\theta_{\ell,0}(\N,\ep). 
\eea  
By  Remark 2.4 we have 
\begin{eqnarray}
\label{Theta}
&&
 \theta_{\ell,k}(\N,\ep)=\sum_{s=0}^{\ell-1}\frac{|\ep_s|\,\|\N_s\|_{\rho_s,k}}{ed_{s}}
\end{eqnarray}
\begin{lemma}
\label{propN}
In the above assumptions:
\begin{enumerate}
\item For any $R>0$ the function $\zeta\mapsto g_\ell(u,\zeta,\ep,\hbar)$ is holomorphic in $\{\zeta\;|\,\,|\zeta|<R\,|\,|\Im\zeta|<\rho\}$, uniformly on compacts with respect to $(u,\ep,\hbar)\in\R\times\R\times [0,1]$;
\vskip 5pt\noindent
\item For any $n\in\Na\cup\{0\}$:
\be
\label{convN}
\sup_{{\zeta\in\R}}\,|[g(u,\zeta,\ep,\hbar)]^n|_{\rho_\ell}\leq [\theta_{\ell}(\N,\ep)]^{n} 
\ee
\item Let:
\be
\label{epbar}
\max_{|\ep|\leq L}{\theta_{\ell}(\N,\ep)}<1, \qquad L>0. 
\ee
Then:
\be
\label{stimaKg}
\sup_{\zeta\in\R;u\in\R}|\K_\F(u,\zeta,\ep,\hbar)|_{\rho_\ell}\leq \frac{1}{|\zeta|}\cdot \frac1{1-\theta_{\ell}(\N,\ep)}
\ee
\item
\bea
&&
\label{stimadgu}
\sup_{\zeta\in\R}\,|\partial^j_u g(u,\zeta,\ep,\hbar)|_{\rho_\ell}\leq  \theta_{\ell,j}(\N,\ep)
\\
&&
\label{stimadgeta}
\sup_{\zeta\in\R}\,|\partial^j_\zeta g(u,\zeta,\ep,\hbar)|_{\rho_\ell}\leq  \theta_{\ell,j}(\N,\ep)
\\
&&
\label{stimadgh}
\sup_{\zeta\in\R}\,|\partial^j_\hbar g(u,\zeta,\ep,\hbar)|_{\rho_\ell
}\leq \theta_{\ell,j}(\N,\ep).
\eea
\end{enumerate}
\end{lemma}
\begin{proof}
The holomorphy is obvious given the holomorphy of $\N_s(u;\ep,\hbar)$. To prove the estimate (\ref{convN}), denoting $\widehat{\N}_s(p,\ep,\hbar)$ the Fourier transform of $\N_s(\xi,\ep,\hbar)$ we write
\vskip 4pt\noindent
\begin{eqnarray}
&&
\label{gF}
g_\ell(u,\zeta,\ep,\hbar)=\frac{1}{\zeta}\sum_{s=0}^{\ell-1}\,\ep_s\,\int_\R\widehat{\N}_\ell(p,\ep,\hbar)(e^{i\zeta p}-1)e^{iu p}\,dp=
\\
\nonumber
&&
\frac{2i}{\zeta}\sum_{s=0}^{\ell-1}\,\ep_s\,\int_\R\widehat{\N}_\ell(p,\ep,\hbar)e^{ip(u+\zeta/2)}\sin{\zeta p/2}\,dp\qquad\quad 
\end{eqnarray}
which entails:
\begin{eqnarray*}
&&
\sup_{{\zeta\in\R}}|g_\ell(u,\zeta,\ep,\hbar)|_{\rho_\ell}=\sup_{{\zeta\in\R}}\int_\R\,|\widehat{g}_\ell (p,\zeta,\ep,\hbar)|e^{\rho_\ell |p|}\,dp
\\
&&
\leq  \max_{\hbar\in [0,1]}\sum_{s=0}^{\ell-1}|\ep_s|\, \int_\R|\widehat{\N}_s(p,\ep,\hbar) p|e^{(\rho_s-d_s) |p|}\,dp
 \leq \frac1{e}\sum_{s=0}^{\ell-1}\,|\ep_s|\,\frac{|\N_s|_{\rho_s}}{d_s}=
\theta_\ell(\N,\ep,1)\,\qquad 0<d_s<\rho_s. 
\end{eqnarray*}
\vskip 4pt\noindent
 Hence Assertion (3) of Proposition \ref{stimeMo}, considered for $k=0$, immediately yields (\ref{convN}). Finally, if $g_\ell$ is defined by (\ref{gl}), then:
$$
\K_\F(u,\zeta,\ep,\hbar)=\frac{1}{\zeta}\frac{1}{1+ g_\ell(u,\zeta,\ep,\hbar)}
$$
and the estimate (\ref{stimaKg}) follows from (\ref{convN}) which makes possible the expansion into the geome\-trical series
\be
\label{sgg}
\frac{1}{1+g_\ell(u,\zeta,\ep,\hbar)}=\sum_{n=0}^\infty\,(-1)^n\,g_\ell(u,\zeta,\ep,\hbar)^n
\ee
\vskip 5pt\noindent
 convergent in the $\theta_{\ell}(\N,\ep)$ norm.  To see (\ref{stimadgu}), remark that (\ref{gF}) yields:
 \vskip 5pt\noindent
 \begin{eqnarray*}
 &&
 \partial^j_u g_\ell(u,\zeta,\ep,\hbar)=\frac{2}{\zeta}\sum_{s=0}^{\ell-1}\,\ep_s\,\int_\R\widehat{\N}_\ell(p,\ep,\hbar)(ip)^j e^{ip(u+\zeta)/2}\sin{\zeta p/2}\,dp.
 \end{eqnarray*}
 Therefore:
  \begin{eqnarray*}
 &&
\sup_{{\zeta\in\R}}\,| \partial^j_u g_\ell(u,\zeta,\ep,\hbar)|_{\rho_\ell}\leq \sup_{{\zeta\in\R}}\,\max_{\hbar\in [0,1]}
2\sum_{s=0}^{\ell-1}\,|\ep_s|\int_\R|\widehat{\N}_s(p,\ep,\hbar)||p|^j|\sin{\zeta p/2}|/{\zeta}|e^{\rho_\ell|p|}\,dp
\\
&&
\leq \sup_{{\zeta\in\R}}\,\max_{\hbar\in [0,1]}
2\sum_{s=}^{\ell-1}\,|\ep_s|\int_\R|\widehat{\N}_s(p,\ep,\hbar)||p|^j|\sin{\zeta p/2}|/{\zeta}|e^{(\rho_s-d_s)|p|}\,dp
\\
&&
\leq 
\sup_{p\in\R}\,[|p|\,\sum_{s=0}^{\ell-1}\,|\ep_s|\,e^{-d_s |p|}]\max_{\hbar\in [0,1]}\int_\R\,|p|^j\widehat{\N}(p,\ep,\hbar)e^{\rho_s|p|}\,dp 
\\
&&
\leq \frac1{e}\sum_{s=0}^{\ell-1}\,|\ep_s|\frac{|\N_s|_{\rho_s,j}}{d_s}\leq  \theta_{\ell,j}(\N,\ep)
\end{eqnarray*}
(\ref{stimadgeta}) is proved by exactly the same argument. Finally, 
to show (\ref{stimadgh}) we write:
 \begin{eqnarray*}
 &&
\sup_{{\zeta\in\R}}| \partial^j_\hbar g_\ell(u,\zeta,\ep,\hbar)|_{\rho_\ell}
\leq  \sup_{{\zeta\in\R}}\max_{\hbar\in [0,1]}
2\sum_{s=0}^{\ell-1}\,|\ep_s|\int_\R|\partial^j_\hbar\widehat{\N_s}(p,\ep,\hbar)|\cdot|\sin{\zeta p/2}|/{\zeta}|e^{\rho_\ell |p|}\,dp
\\
&&
\leq 
\max_{\hbar\in [0,1]}\sum_{s=0}^{\ell-1}\,|\ep_s|\int_\R|\partial^j_\hbar\widehat{\N}(p,\ep,\hbar)|e^{(\rho_s-d_s)|p|}\,dp
\leq \theta_{\ell}(\N,\ep)
\end{eqnarray*}
\vskip 4pt\noindent
 This proves the Lemma.
\end{proof}
By  \textbf{Condition (1)} the operator family $\hbar \mapsto \F(L_\om;\ep,\hbar)$, defined by the spectral theorem, is self-adjoint in $L^2(\T^l)$; by \textbf{Condition (2)} $D(\F(L_\om))=H^1(\T^l)$. Since $L_\om$ is a first order operator with symbol  $\L_\om$, the symbol of $\F(L_\om;\ep,\hbar)$ is $\F(\L_\om(\xi),\ep,\hbar)$. 
We can now state the main result of this section. Let $\F_\ell(x,\ep,\hbar)$ be as in Lemma \ref{propN}, which entails the validity of 
\textbf{Conditions  (1), (2), (3)}.  
\begin{theorem}
\label{homo}
\label{homeq}
Let  $V_\ell\in J_k(\rho_\ell)$, $\ell=0,1\ldots$, $V_1\equiv V$  for some $\rho_\ell> \rho_{\ell+1}>0$, $k=0,1,\ldots$. Let  $\V_\ell(\L_\om(\xi),x;\ep,\hbar)\in\J_k(\rho)$ be its symbol.  Then for any  
$\ds \theta_{\ell}(\N,\ep)<1$
 the homological equation (\ref{heq}), rewritten as
 \vskip 4pt\noindent
 \be
 \label{heqell}
 \frac{[\F_\ell(L_\om),W_\ell]}{i\hbar}+V_{\ell}=N_\ell(L_\om,\ep)
\ee
 \vskip 6pt\noindent
\be
\label{Moell}
\{\F_\ell(\L_\om(\xi),\ep,\hbar),\W_\ell(x,\xi;\ep,\hbar)\}_M+\V_{\ell}(x,L_\om(\xi);\ep,\hbar)={\mathcal N}_\ell(\L_\om(\xi),\ep,\hbar)
\ee 
\vskip 4pt\noindent
   admits a unique solution $(W_\ell,N_\ell)$ of Weyl symbols  $\W_\ell(\L_\om(\xi),x;\ep,\hbar)$, $\N_\ell(\L_\om(\xi),\ep,\hbar)$ such that
\begin{enumerate}
\item $W_\ell=W^\ast_\ell\in J_k(\rho_\ell)$, with:  
\bea
&&
\label{Thm5.1}
\|W_\ell\|_{\rho_{\ell+1},k}=\|\W\|_{\rho_{\ell+1},k}\leq  A(\ell,k,\ep)\|\V_\ell\|_{\rho_{\ell},k}
\\
\nonumber
&&
{}
\\
&&
\label{Adrk}
A(\ell,k,\ep)=\gamma \frac{\tau^\tau}{(ed_\ell)^\tau}\left[1+\frac{2^{k+1}(k+1)^{2(k+1)}k^k}{(e\delta_\ell)^{k}[1-\theta_\ell(\N,\ep)]^{k+1}} \theta_{\ell,k}^{k+1}\right].
 \eea
\vskip 6pt\noindent
\item $\N_\ell=\overline{\V}_\ell$; therefore $\N_\ell\in J_k(\rho_\ell)$ and $ \|\N \|_{\rho_\ell,k}
\leq  \|\V_\ell\|_{\rho_\ell,k} .$
\end{enumerate}
\end{theorem}
\begin{proof} 

The proof of (2) is obvious and follows from the definition of the norms $\Vert\cdot\Vert_\rho$ and $\Vert\cdot\Vert_{\rho,k}$.
 The self-adjointess property $W=W^*$ is implied by the construction itself, which makes $W$ symmetric and bounded.
 
 Consider $\W_\ell$ as defined by (\ref{defW}).   Under the present assumptions, by Lemma \ref{propN} we have:
 \vskip 8pt\noindent
$$
\truc{\W}_\ell(\L_\om(\xi),q;\ep,\hbar):=\frac1{\la\om,q\ra}\frac{i\hbar\truc{\V}_\ell(\L_\om(\xi);q;\ep,\hbar)}{1+ g_\ell(\L_\om(\xi);\la\om,q\ra\hbar,\ep,\hbar)}, \quad q\neq 0; \quad \truc{\W}_\ell(\cdot,0;\hbar)=0.
$$  
\vskip 8pt\noindent
By the $\|\cdot\|_{\rho_\ell}$-convergence of the series (\ref{sgg}) we can write
\begin{eqnarray}
&&
\partial^\gamma_\hbar \truc{\W}_\ell(\L_\om(\xi),q;\ep,\hbar)=\sum_{n=0}^\infty\,(-\ep)^n\,\partial^\gamma_\hbar \truc{\W}_{\ell,n}(\L_\om(\xi),q;\ep,\hbar),
\\
&&
\truc{\W}_{\ell,n}(\L_\om(\xi),q;\ep,\hbar)=\frac1{\la\om,q\ra}\truc{\V}_\ell(\L_\om(\xi);q;\ep,\hbar)[g_\ell(\L_\om(\xi);\la\om,q\ra\hbar,\ep,\hbar)]^n
\\
\label{derivateWn}
&&
\partial^\gamma _\hbar\truc{\W}_{\ell,n}(\L_\om(\xi),q;\ep,\hbar)=
\\
\nonumber
&&\sum_{j=0}^\gamma\,\binom{\gamma}{j}\,\partial^{\gamma-j}_\hbar \truc{\V}_\ell(\L_\om(\xi);q;\ep,\hbar)D^j_\hbar [g_\ell(\L_\om(\xi);\la\om,q\ra\hbar,\ep,\hbar)]^n
\end{eqnarray}
\vskip 4pt\noindent
where $D_\hbar$ denotes the total derivative with respect to $\hbar$. We need the following preliminary result. 
\begin{lemma}
\label{derivateg} 
Let $\zeta(\hbar):=\la\om,q\ra\hbar$. 
Then: 
\begin{enumerate}
\item
\bea
\label{stimadghh}
|D^j_\hbar g_\ell(\L_\om(\xi),\zeta(\hbar),\ep,\hbar)|_{\rho_\ell}
\leq (j+1) ({2|q|})^j \theta_{\ell,j}(\N,\ep)^2
\eea
\item
\bea
\label{stimadgjn}
|D^j_\hbar [g_\ell(\L_\om(\xi);\zeta(\hbar),\ep,\hbar)]^n|_{\rho_\ell}\leq 2n^j (\theta_\ell(\N,\ep))^{n-j} [2(j+1)|q|]^j\theta_{\ell,j}(\N,\ep)^{2j}.
\eea
\end{enumerate}
\end{lemma}
\begin{proof}
The expression of total derivative $D_\hbar g$ is:
\be
\label{Dom}
D_\hbar g(\cdot;\la\om,q\ra\hbar,\ep,\hbar)=(\la\om,q\ra\ \frac{\partial}{\partial\zeta}+\frac{\partial}{\partial\hbar})\left.g_\ell(\cdot;\zeta,\ep,\hbar)\right|_{\zeta=\la\om,q\ra\hbar}
\ee
By Leibnitz's formula we then have:
\be
D^j_\hbar g_\ell(\cdot;\la\om,q\ra\hbar,\ep,\hbar)=\sum_{i=0}^j\,\binom{j}{i}\la\om,q\ra^{j-i}\frac{\partial^{j-i}g_\ell}{\partial\zeta^{j-i}}\frac{\partial^i g_\ell}{\partial\hbar^{i}}
\ee
Apply now  (\ref{simple}) with $k=0$, (\ref{stimadgu}) and (\ref{stimadgh}). We get: 
\vskip 5pt\noindent
\begin{eqnarray*}
\left\vert\frac{\partial^{j-i}g_\ell}{\partial\zeta^{j-i}}\frac{\partial^i g_\ell}{\partial\hbar^{i}}\right\vert_{\rho_\ell}\leq (j+1)2^j \theta_{\ell,j}(\N,\ep)^2
\end{eqnarray*} 
whence, since $|\om|\leq 1$:
\begin{eqnarray}
\label{stimaDjg}
\left\vert\frac{D^jg_\ell}{D\hbar^j}\right\vert_{\rho_\ell} \leq (j+1)(2)^j{|q|^j}\theta_{\ell,j}(\N,\ep)^2
\end{eqnarray}
This proves Assertion (1). To prove Assertion (2), let us first note that 
\be
D^j_\hbar [g_\ell(\L_\om(\xi);\la\om,q\ra\hbar,\ep,\hbar)]^n=P_{n,j}\left(g_\ell,\frac{Dg_\ell}{D\hbar},\ldots,\frac{D^jg_\ell}{D\hbar^j}\right).
\ee
\vskip 5pt\noindent
where $P_{n,j}(x_1,\ldots,x_j)$ is a homogeneous polynomial of degree $n$ with $n^j$ terms.  Explicitly: 
$$
P_{n,j}\left(g_\ell,\frac{Dg_\ell}{D\hbar},\ldots,\frac{D^jg_\ell}{D\hbar^j}\right)=\sum_{j=1}^n\,{g_\ell}^{n-j}\,\prod_{{k=1}\atop {j_1+\ldots+j_k=j}}^j \frac{D^{j_k}g_\ell}{D\hbar^{j_k}}.
$$
Now (\ref{stimadghh}), (\ref{stimaDjg}) and Proposition \ref{stimeMo} (3)  entail:
\begin{eqnarray*}
&&
|D^j_\hbar [g_\ell(\L_\om(\xi);\la\om,q\ra\hbar,\ep,\hbar)]^n|_{\rho_\ell}\leq  n^j|g|_{\rho_\ell}^{n-j} \prod_{{k=1}\atop {j_1+\ldots+j_k=j}}^j  2(j_k+1)\left({2|q|}\right)^{j_k}\theta_{\ell,j_k}(\N,\ep)^2
\\
&&
\leq 2n^j (\theta_\ell(\N,\ep))^{n-j} [2(j+1)|q|]^j\theta_{\ell,j}(\N,\ep)^{2j}.
\end{eqnarray*}
This concludes the proof of the Lemma.
\end{proof}
\noindent
To conclude the proof of the theorem, we must estimate the $\|\cdot\|_{\rho_{\ell+1},k}$ norm of the derivatives $\ds \partial^\gamma _\hbar\W_{\ell,n}(\L_\om(\xi),x;\ep,\hbar)$.  Obviously:
\be
\label{serieW}
\|\W_\ell(\xi,x;\ep,\hbar)\|_{\rho_\ell+1,k}\leq \sum_{n=0}^\infty\,\|\W_{\ell,n}(\xi,x;\ep,\hbar)\|_{\rho_{\ell+1,k}}.
\ee
\vskip 4pt\noindent
For $n=0$:
\begin{eqnarray*}
&&
\|\W_{\ell,0}(\xi,x;\ep,\hbar)\|_{\rho_{\ell+1,k}}\leq \gamma\sum_{\gamma=0}^k\int_{\R\times\R^l}|\partial^\gamma_\hbar\widehat{\W}_{\ell,0}(p,s;\cdot)||s|^{\tau}\mu_{k-\gamma}(p\om,s)\,e^{\rho_{\ell+1} (|p|+|s|)}\,d\lambda(p,s)
\\
&&
\leq 
\gamma\sum_{\gamma=0}^k\int_{\R\times\R^l}|\partial^\gamma_\hbar\widehat{\V}_{\ell,0}(p,s;\cdot)||s|^{\tau}\mu_{k-\gamma}(p\om,s)\,e^{\rho_{\ell+1} (|p|+|s|)}\,d\lambda(p,s)\leq \gamma\frac{\tau^\tau}{(ed_\ell)^\tau}\|\V_\ell\|{\rho_{\ell,k}}
\end{eqnarray*}
where the inequality follows again by the standard majorization 
\vskip 6pt\noindent
$$
e^{\rho_{\ell+1} (|p|+|s|)}=e^{\rho_{\ell} (|p|+|s|)}e^{-d_\ell(|p|+|s|)}, \quad \sup_{s\in\R^l}[|s|^\tau e^{-d_\ell |s|}]\leq \gamma\frac{\tau^\tau}{(ed_\ell)^\tau}
$$
\vskip 4pt\noindent
on account of the small denominator estimate (\ref{DC}).  For $n>0$ we can write, on account of (\ref{pm1},\ref{pm2}):
\begin{eqnarray*}
&&
\|\W_{\ell,n}(\xi,x;\cdot)\|_{\rho_{\ell+1},k}=\sum_{\gamma=0}^k\int_{\R\times\R^l}|\partial^\gamma_\hbar\widehat{\W}_{\ell,n}(p,s;\cdot)||s|^{\tau}\mu_{k-\gamma}(p\om,s)\,e^{\rho_{\ell+1} (|p|+|s|)}\,d\lambda(p,s)\leq
\\
&&
\leq  \gamma\frac{\tau^\tau}{(ed_\ell)^\tau}\sum_{\gamma=0}^k\sum_{j=0}^\gamma \,\binom{\gamma}{j}\,\int_{\R^l}{\mathcal Q}(s,\cdot)e^{\rho_\ell |s|}\,d\nu(s)
\end{eqnarray*}
where
\begin{eqnarray*}
{\mathcal Q}(s,\cdot):=\int_\R|[\partial^{\gamma-j}_\hbar \widehat{\V}_{\ell}(p;s;\cdot)]\ast [D^j_\hbar \widehat{g}^{\,\ast_n}_\ell(p;\la\om,s\ra\hbar,\cdot)] \mu_{k-\gamma}(p\om,s)\,e^{\rho_\ell |p|}\,dp
\end{eqnarray*}
Here $\ast$ denotes convolution with respect only to the $p$ variable, and $\widehat{g}^{\,\ast_i n}_\ell(p,\zeta,\cdot)$ denotes the $n-$th convolution of $\widehat{g}_\ell$ with itself, i.e. the $p$-Fourier transform of $g^n_\ell$.  Now, by Assertion (3) of Proposition (\ref{stimeMo}) and the above Lemma:
\begin{eqnarray*}
&&
\int_{\R^l}{\mathcal Q}(s,\cdot)e^{\rho_\ell |s|}\,d\nu(s)=
\\
&&
=\int_{\R\times\R^l}|[\partial^{\gamma-j}_\hbar \widehat{\V}_{\ell}(p;s;\cdot)]\ast_\xi [D^j_\hbar g^{\ast_\xi n}_\ell(p;\la\om,s\ra\hbar,\cdot)] \mu_{k-\gamma}(p\om,s)\,e^{\rho_\ell(|p|+|s|)}\,d\lambda(p,s)
\\
&&
\leq \int_{\R^l}\left[\int_{\R}|[\partial^{\gamma-j}_\hbar \widehat{\V}_{\ell}(p;s;\hbar)]\ast [D^j_\hbar \widehat{g}^{\,\ast_ n}(p;\la\om,s\ra\hbar,\cdot)]|\mu_{k-\gamma}(p\om,s)\,e^{\rho_\ell |p|}\,dp\right]e^{\rho_\ell |s|} \,d\nu(s)
\\
&&
\leq 2A(j)^j\theta_\ell(\N,\ep)^{n-j}\int_{\R^l}\int_{\R}|\partial^{\gamma-j}_\hbar \widehat{\V}_{\ell}(p;s;\cdot)|\mu_{k-\gamma}(p\om,s)\,e^{\rho_\ell |p|}|s|^{j}e^{\rho_\ell |s|} \,\,dp d\nu(s),
 \end{eqnarray*}
 with
\[
 A(j):= 2n  (j+1) \theta_{\ell,j}(\N,\ep)^{2}.
\]
This yields,   with $\delta_\ell$ defined by (\ref{ddll}):
\begin{eqnarray*}
&&
\|\W_{\ell,n}(\xi,x;\cdot)\|_{\rho_\ell+1,k}\leq \gamma\frac{\tau^\tau}{(ed_\ell)^\tau}\sum_{\gamma=0}^k\int_{\R\times\R^l}|\partial^\gamma_\hbar\widehat{\W}_{\ell,n}(p,s;\cdot)\mu_{k-\gamma}(p\om,s)\,e^{\rho_\ell(|p|+|s|)}\,d\lambda(p,s)\leq
\\
&&
\leq \frac{\gamma \tau^\tau (k+1)(2A(k))^k}{(ed_\ell)^\tau}\theta_\ell(\N,\ep)^{n-j}\sum_{\gamma=0}^k\int_{\R\times \R^l} |\partial^{\gamma}_\hbar \widehat{\V}_{\ell}(p;s;\cdot)|\cdot \mu_{k-\gamma}(p\om,s)\,e^{\rho_\ell |p|}|s|^{j}e^{\rho_\ell |s|} \,\,d\lambda(p,s)
\\
&&
\leq \frac{\gamma \tau^\tau (k+1)(2A(k))^k}{(ed_\ell)^\tau}\frac{k^{k}}{(e\delta_\ell)^{k}}\theta_\ell(\N,\ep)^{n-j}\sum_{\gamma=0}^k\int_{\R^l}\int_{\R}|\partial^{\gamma}_\hbar \widehat{\V}_{\ell}(p;s;\cdot)| \mu_{k-\gamma}(p\om,s)e^{\rho |p|}e^{\rho |s|}\,d\lambda(p,s)
\\
&&
\leq  \gamma\frac{\tau^\tau}{(ed_\ell)^\tau}\frac{(k+1)k^{k}}{(e\delta_\ell)^{k}} 2(2n)^k(\theta_\ell(\N,\ep))^{n-j}(k+1)^k\theta_{\ell,k}^{2k} \|\V_\ell\|_{\rho,k}.
\end{eqnarray*}
\vskip 4pt\noindent 
Therefore, by (\ref{serieW}):
\begin{eqnarray*}
&&
\|{\W}_\ell(\xi;x;\ep,\hbar)\|_{\rho_{\ell+1},k} \leq \sum_{n=0}^\infty\,{\W}_{\ell,n}
(\xi;x;\ep,\hbar)\|_{\rho_{\ell+1},k} \leq
\\
&&
\leq
\gamma \frac{\tau^\tau}{(ed_\ell)^\tau}\|\V_\ell\|_{\rho_\ell,k}\left[1+\frac{2^{k+1}(k+1)^{k+1}k^k}{(e\delta_\ell)^{k}} \theta_{\ell,k}^{2k}\sum_{n=1}^\infty\, n^k (\theta_\ell(\N,\ep))^{n-j}\right]
\\
&&
\leq
\gamma \frac{\tau^\tau}{(ed_\ell)^\tau}\|\V_\ell\|_{\rho_\ell,k}\left[1+\frac{2^{k+1}(k+1)^{k+1}k^k}{(e\delta_\ell)^{k}} \theta_{\ell,k}^{2k-j}\sum_{n=1}^\infty\, n^k (\theta_\ell(\N,\ep))^{n}\right]
\\
&&
\leq\gamma \frac{\tau^\tau}{(ed_\ell)^\tau}\|\V_\ell\|_{\rho_\ell,k}\left[1+\frac{2^{k+1}(k+1)^{2(k+1)}k^k}{(e\delta_\ell)^{k}[(1- \theta_\ell(\N,\ep)^{k+1}]} \theta_{\ell,k}^{k+1}\right].
\end{eqnarray*}
\vskip 4pt\noindent
because $j\leq k$, and
\begin{eqnarray*}
&&
\sum_{n=1}^\infty \,n^kx^n\leq \sum_{n=1}^\infty\,(n+1)\cdots (n+k) x^n=\frac{d^k}{dx^k}\,\sum_{n=1}^\infty x^{n+k}
\\
&&
=\frac{d^k}{dx^k}\frac{x^{k+1}}{1-x}=(k+1)!\sum_{j=0}^{k+1}\left(k+1-j\atop j\right)
 \frac{x^{k+1-j}}{(1-x)^j}\leq 
\frac{2^{k+1}(k+1)!}{(1-x)^{k+1}}.
\end{eqnarray*}
By the Stirling formula this estimate  concludes the proof of the Theorem.  
\end{proof}
\vskip 2pt\noindent

 \subsection{Towards KAM iteration}\label{towkam}
 
 Let us now prove the estimate which represents the starting point of the KAM iteration:
 \begin{theorem}
 \label{resto}
 Let $\F_\ell$ and $V_\ell$ be as in Theorem \ref{homeq}, and let $W_\ell$ be the solution of the homological equation (\ref{heq}) as constructed and estimated in Theorem \ref{homo}.  Let (\ref{epbar}) hold and let furthermore 
 \be
 \label{condepell}
  |\ep|<\overline{\ep}_\ell, \quad \overline{\ep}_\ell:=\left(\frac{d_\ell}{\|\W_\ell\|_{\rho_{\ell+1},k}}\right)^{2^{-\ell}}.
\ee
 Then we have:
 \be
 \label{resto1}
 e^{i\ep_\ell W_\ell/\hbar}(\F_\ell(L_\om)+\ep_\ell V_\ell)e^{-i\ep_\ell W_\ell/\hbar}=(\F_\ell+\ep_\ell  N_\ell)(L_\om)+\ep_\ell^2V_{\ell+1,\ep}
 \ee
 where, $\forall\,0<2d_\ell<\rho_\ell$ and $k=0,1,\ldots$:
 \bea
 &&
 \label{resto2}
 \|V_{\ell+1,\ep}\|_{\rho_\ell-2d_\ell,k}\leq 
 C(\ell,k,\ep) \frac{\|\V_\ell\|^2_{\rho_\ell,k}} {1-{|\ep_\ell |}(k+1)4^k A(\ell,k,\ep)
\|\V\|_{\rho_\ell,k}/{(ed_\ell)^2}}
 \\
 &&
 \nonumber
 {}
 \\
 \label{Cdrk}
 &&
 C(\ell,k,\ep):=\frac{(k+1) 4^{k}}{(ed_\ell)^2}{A(\ell,k,\ep)}\left[2+|\ep_{\ell} |\frac{(k+1) 4^{k}}{(ed_\ell)^2 }{A(\ell,k.\ep)\|\V_\ell\|_{\rho_\ell,k}}{}  \right]
 \eea
 \vskip 6pt\noindent
Here $A(\ell,k,\ep)$ is defined by (\ref{Adrk}).
 \end{theorem}
 \begin{remark} We will verify in the next section (Remark  \ref{verifica} below) that (\ref{condepell}) is actually fulfilled for $|\ep|<1/|\V|_\rho$.
 \end{remark}
 \begin{proof} To prove the theorem we need an auxiliary  result, namely:
 \begin{lemma}
 \label{RResto4}
 For $\ell=0,1,\ldots$ let $\rho_\ell>0, \rho_0:=\rho$, $A\in J_k(\rho)$, $W_\ell\in J_k(\rho_\ell)$, $k=0,1,\ldots$.  Let $W_\ell^\ast=W_\ell$, and define:
 \be
 \label{resto5}
 A_{\ep}(\hbar):=e^{i\ep_\ell W_\ell/\hbar}Ae^{-i\ep_\ell W/\hbar}.
 \ee
 Then, for $\ds |\ep_\ell|< [e d^2_\ell/((k+1)4^k\|\W\|_{\rho_{\ell+1},k})]^{2^{-\ell}}$, and  $\forall\,0<d_\ell<\rho_\ell$, $k=0,1,\ldots$:
 \vskip 8pt\noindent
 \be
 \label{resto6}
 \|A_{\ep}(\hbar)\|_{\rho_\ell-d_\ell,k}\leq
\frac{\|{\mathcal A}\|_{\rho_\ell,k}}{1-|\ep_\ell | (k+1)4^k \|\W\|_{\rho_{\ell+1},k}/(ed^2_\ell)}
 \ee
  \vskip 4pt\noindent
 \end{lemma}
 \begin{proof}
Since the operators $W_\ell$ and $A$ are bounded, there is ${\ep}_0>0$ such that the commutator expansion for $A_{\ep}(\hbar)$:
\vskip 4pt\noindent
$$
A_{\ep}(\hbar)=\sum_{m=0}^\infty \frac{(i\ep_\ell)^m}{ \hbar^m m!}[W_\ell,[W_\ell,\ldots,[W_\ell,A]\ldots]
$$
\vskip 4pt\noindent
is norm convergent  for $|\ep|<\ep_0$   if $\hbar\in]0,1[$ is fixed. The corresponding expansion for the symbols is
\vskip 4pt\noindent
$$
\A_{\ep}(\hbar)=\sum_{m=0}^\infty \frac{(\ep_\ell)^m}{m!}\{\W_\ell,\{\W,\ldots,\{\W_\ell,{\mathcal A}\}_M\ldots\}_M
$$
\vskip 4pt\noindent
Now we can apply  once again Corollary \ref{multipleM}. We get:
\vskip 5pt\noindent
\be
\frac{1}{m!}\|\{\W_\ell,\{\W_\ell,\ldots,\{\W_\ell,{\mathcal A}\}_M\ldots\}_M\|_{\rho_\ell-d_\ell,k}
\leq  
\left(\frac{(k+1)4^{k}\|\W_\ell\|_{\rho_{\ell+1},k}}{ed^2_\ell}\right)^m
\|{\mathcal A}\|_{\rho_\ell,k}
\ee
\vskip 5pt
\noindent
Therefore:
\vskip 4pt\noindent
\begin{eqnarray*}
\|A_{\ep}(\hbar)\|_{\rho_\ell-d_\ell,k} &\leq & \|{\mathcal A}\|_{\rho_\ell,k}\sum_{m=0}^\infty |\ep_\ell|^m [(k+1)4^{k}\|\W\|_{\rho_{\ell+1},k}/(ed^2_\ell)]^m
\\
&=&
\frac{\|{\mathcal A}\|_{\rho_\ell,k}}{1-|\ep_\ell | (k+1)4^k \|\W\|_{\rho_{\ell+1},k}/(ed^2_\ell)}
\end{eqnarray*}
\vskip 4pt\noindent
and this concludes the proof. 
 \end{proof}
   $W_\ell$ solves  the homological equation (\ref{heq}).  Then by Theorem \ref{homo}  $W_\ell=W_\ell^\ast\in J_k(\rho_\ell-d_\ell)$, $k=0,1,\ldots$; in turn, by  Assertion (3) of Corollary \ref{corA} the unitary operator $\ds e^{i\ep_\ell W_\ell/\hbar}$ leaves $H^1(\T^l)$ invariant.  Therefore the unitary image of $H_\ep$ under $\ds e^{i\ep_\ell W/\hbar}$ is the real-holomorphic operator family in $L^2(\T^l)$ 
\be
\label{S}
\ep_\ell\mapsto S_{\ep_\ell}:=e^{i\ep_\ell W_\ell/\hbar}(\F_\ell(L_\om)+\ep_\ell V_\ell)e^{-i\ep_\ell W/\hbar}, \quad D(S(\ep_\ell))=H^1(\T^l)
\ee
Computing its Taylor expansion  at $\ep_\ell=0$  with second order  remainder we obtain:
\begin{eqnarray}\label{lemmm}
&&
S_{\ep_\ell}u=\F_\ell(L_\om)u+\ep_\ell N_\ell(L_\om)u+ \ep_\ell^2 V_{\ell+1,\ep}u, \quad u\in H^1(\T^l)
\\
\nonumber
&&
{}
\\
&&  
V_{\ell+1,\ep_\ell}=\frac12\int_0^{\ep_\ell} (\ep_\ell -t)e^{i t W_\ell/\hbar}\left(\frac{[N_\ell,W_\ell]}{i\hbar}+\frac{[W_\ell,V_\ell]}{i\hbar}+t \frac{[W_\ell,[W_\ell,V_\ell]]}{(i\hbar)^2}\right)e^{-itW_\ell/\hbar}\,dt
\end{eqnarray}
\vskip 5pt\noindent
To see this, first remark that $S_0=\F(L_\om)$. Next, we compute, as equalities between continuous operators in $L^2(\T^l)$:
\begin{eqnarray*}
&& 
S^\prime_{\ep_\ell}=e^{i\ep_\ell W/\hbar}([\F_\ell(L_\om),W_\ell]/i\hbar +V_\ell+\ep_\ell [V,W]/i\hbar)e^{-i\ep_\ell W/\hbar}=
\\
&&
e^{i\ep_\ell W/\hbar}(N_\ell+\ep_\ell [V_\ell,W_\ell]/i\hbar)e^{i\ep_\ell W_\ell/\hbar}; \qquad
S^\prime_0= N_\ell
\\
&&
S^{\prime\prime}_{\ep_\ell}=e^{i\ep_\ell W_\ell/\hbar}([N_\ell,W_\ell]/i\hbar +  [V_\ell,W_\ell]/i\hbar +\ep_\ell [W_\ell,[W_\ell,V_\ell]]/(i\hbar)^2)e^{-i\ep_\ell W_\ell/\hbar},
\end{eqnarray*}
and this proves (\ref{lemmm}) by the second order Taylor's  formula with remainder:
$$
S_{\ep_\ell}=S(0)+\ep_\ell S^\prime_0+\frac12\int_0^{\ep_\ell} (\ep_\ell-t)S^{\prime\prime}(t),dt
$$
The above formulae obviously yield
\be
\label{stimar2}
\| {V}_{l+1,\ep_\ell}\|\leq |\ep_\ell |^2 \max_{0\leq |t|\leq |\ep_\ell |}\|S^{\prime\prime}(t)\|
\ee
Set now:
\be
\label{R1}
R_{\ell+1,\ep_\ell}:=[N_\ell,W_\ell]/i\hbar +  [V_\ell,W_\ell]/i\hbar +\ep_\ell [W_\ell,[W_\ell,V_\ell]]/(i\hbar)^2
\ee
$R_{\ell+1,\ep_\ell}$ is a continuous operator in $L^2$, corresponding to the symbol
\be
\label{simbR1}
{\mathcal R}_{\ell+1,\ep_\ell}(\L_\om(\xi),x;\hbar)=\{\N_\ell,\W_\ell\}_M+\{\V_\ell,\W_\ell\}_M+\ep_\ell\{\W_\ell,\{\W_\ell,\V_\ell\}_M\}_M
\ee
 Let us estimate  the three terms individually.  
  By Theorems \ref{homo} and \ref{stimeMo}  we can write, with $A(\ell,k,\ep)$  given by (\ref{Adrk}):
  \begin{eqnarray*}
 &&
 \|[N_\ell,W_\ell]/i\hbar\|_{\rho_\ell-d_\ell,k}\leq \|\{\N_\ell,\W_\ell\}_M\|_{\rho_\ell-d_\ell,k}\leq 
 \frac{(k+1)4^k}{(ed_\ell)^2}\|\W_\ell\|_{\rho_{\ell+1},k}\|\N_\ell\|_{\rho_\ell,k}
 \\
 &&
 \leq  
 \frac{(k+1)4^k}{(ed_\ell)^2} A(\ell,k,\ep)\|\V_\ell\|^2_{\rho_\ell,k}
 \\
 &&
 \|[V_\ell,W_\ell]/i\hbar\|_{\rho_\ell-d_\ell,k}\leq\|\{\V_\ell,\W_\ell\}_M\|_{\rho_\ell-d_\ell,k}\leq 
\frac{(k+1)4^k}{(ed_\ell)^2}\|\V_\ell\|_{\rho_\ell,k}\|\W_\ell\|_{\rho_{\ell+1},k}\leq \\
 &&
 \leq  \frac{(k+1)4^k}{(ed_\ell)^2}A(\ell,k,\ep)\|\V_\ell\|^2_{\rho_\ell,k}
 \\
 &&
 \|[W_\ell,[W_\ell,V_\ell]]/(i\hbar)^2\|_{\rho_\ell-d_\ell,k}\leq \|\{\W_\ell,\{\W_\ell,\V_\ell\}_M\}_M\|_{\rho_\ell-d_\ell,k}\leq \frac{(k+1)^2 4^{2k}}{(ed_\ell)^4} \|\W_\ell\|_{\rho_{\ell+1},k}^2 \|\V_\ell\|_{\rho_\ell,k}
 \\
 &&
 \leq  \frac{(k+1)^2 4^{2k}}{(ed_\ell)^4}A(\ell,k,\ep)^2\|\V_\ell\|_{\rho_\ell,k}^3
  \end{eqnarray*}
  \vskip 6pt\noindent
 We can now apply Lemma \ref{RResto4}, which yields:
 \vskip 2pt\noindent
 \begin{eqnarray*}
 &&
 \|e^{i\ep_\ell W_\ell/\hbar}[N_\ell,W_\ell] e^{-i\ep_\ell W_\ell/\hbar}/i\hbar\|_{\rho_\ell-d_\ell-d^\prime_\ell,k}\leq  \frac{(k+1) 4^{k}}{(ed_\ell)^2}\Xi(\ell,k)
\\
 &&
 \|e^{i\ep_\ell W_\ell/\hbar}[V_\ell,W_\ell] e^{-i\ep_\ell W_\ell/\hbar}/i\hbar\|_{\rho_\ell-d_\ell-d^\prime_\ell,k}\leq  \frac{(k+1) 4^{k}}{(ed_\ell)^2 }\Xi(\ell,k)
 \\
 &&
 \|e^{i\ep_\ell W_\ell/\hbar}[W_\ell,[W_\ell,V_\ell]] e^{-i\ep_\ell W_\ell/\hbar}/(i\hbar)^2\|_{\rho_\ell-d_\ell-d^\prime_\ell,k}\leq  \frac{(k+1)^2 4^{2k}}{(ed_\ell)^4 }\Xi_1(\ell,k)
 \end{eqnarray*}
 where
 \bea
 &&
 \label{Xi}
 \Xi(\ell,k):= A(\ell,k,\ep)\cdot\frac{\|\V_\ell\|^2_{\rho_\ell,k}} {1-|\ep_\ell (k+1)4^k |\|\W\|_{\rho_{\ell+1},k}/(ed^2_\ell)}
 \\
 &&
 \label{Xi1}
 \Xi_1(\ell,k)=A(\ell,k,\ep)^2\cdot \frac{\|\V\|^3_{\rho_\ell,k}} {1-|\ep_\ell (k+1)4^k |\|\W\|_{\rho_{\ell+1},k}/(ed^2_\ell)}
\eea
 \vskip 6pt\noindent
 Therefore,  summing the three inequalities we get
 \vskip 3pt\noindent
 \begin{eqnarray*}
 &&
 \|V_{\ell+1,\ep}\|_{\rho_\ell-d_\ell-d^\prime_\ell,k} \leq  
 \\
 &&
 \frac{(k+1) 4^{k}}{(ed_\ell)^2 }A(\ell,k,\ep)
 \frac{\|\V_\ell\|^2_{\rho_\ell,k}} {1-|\ep_\ell | (k+1)4^k \|\W_\ell\|_{\rho_{\ell+1},k}/(ed^2_\ell)}\left[2+|\ep_\ell|\frac{(k+1) 4^{k}}{(ed_\ell)^2 }A(\ell,k,\ep){\|\V_\ell\|_{\rho_\ell,k}} \right]
 \end{eqnarray*}
 \vskip 8pt\noindent
 If we choose $d^\prime_\ell=d_\ell$ this is (\ref{resto2}) on account of Theorem \ref{homo}. 
 This concludes the proof of Theorem \ref{resto}. 
 \end{proof}

\vskip 1cm\noindent
\section{Recursive estimates}\label{recesti}
\renewcommand{\thetheorem}{\thesection.\arabic{theorem}}
\renewcommand{\theproposition}{\thesection.\arabic{proposition}}
\renewcommand{\thelemma}{\thesection.\arabic{lemma}}
\renewcommand{\thedefinition}{\thesection.\arabic{definition}}
\renewcommand{\thecorollary}{\thesection.\arabic{corollary}}
\renewcommand{\theequation}{\thesection.\arabic{equation}}
\renewcommand{\theremark}{\thesection.\arabic{remark}}
\def\P{{\mathcal P}}
\setcounter{equation}{0}%
\setcounter{theorem}{0}%
Consider the $\ell$-th step of the KAM iteration.  Summing up the results of the preceding Section we can write: 
\begin{eqnarray*}
&& 
\bullet\  S_{\ell,\ep}:=e^{i\ep_\ell W_\ell/\hbar}\cdots e^{i\ep_2 W_1/\hbar}e^{i\ep W_0/\hbar}(\F(L_\om)+\ep V)e^{-i\ep W_0/\hbar}e^{-i\ep_2 W_1/\hbar}\cdots e^{-i\ep_\ell W_\ell/\hbar}
\\
&& 
= e^{i\ep_\ell W_\ell/\hbar}(\F_{\ell,\ep}(L_\om)+\ep^{2^\ell} V_{\ell,\ep})e^{-i\ep_\ell W_\ell/\hbar}
=\F_{\ell+1,\ep}(L_\om)+\ep_{\ell +1} V_{\ell +1,\ep},  
\\
&&
\bullet\  \F_{\ell,\ep}(L_\om)=\F(L_\om)+\sum_{k=1}^{\ell-1} \ep_kN_k(L_\om), \quad [\F_{\ell}(L_\om),W_\ell]/i\hbar +V_{\ell,\ep} =N_\ell(L_\om,\ep)
\\
&&
\bullet 
V_{\ell+1,\ep}=\frac12\int_0^{\ep_\ell} (\ep_\ell -t)e^{i t W_\ell/\hbar}R_{\ell+1,t}e^{-itW_\ell/\hbar}\,dt
\\
&&
\bullet\  R_{\ell+1,\ep}:=[N_{\ell},W_{\ell}]/\hbar+[W_{\ell},V_{\ell,\ep}]/\hbar+\ep_{\ell} [W_{\ell},[W_{\ell},V_{\ell,\ep}]]/\hbar^2
\end{eqnarray*}
 \vskip 6pt\noindent
We now proceed to obtain recursive estimates for the above quantities in the $\|\cdot\|_{\rho_\ell,k}$ norm. Consider (\ref{resto2}) and denote: 
 \vskip 6pt\noindent
\bea
&&
\label{stimaPsi}
\Psi(\ell,k)=\frac{(k+1)4^k}{(ed_\ell)^2}; \quad \Pi(\ell,k):= \frac{[2(k+1)^2]^{k+1}k^k}{e^{k}d_\ell^{k}}
\\
\label{Pll}
&&
P(\ell,k,\ep):=\frac{\theta_{\ell,k}(\N,\ep)^{k+1}}{[1-\theta_\ell(\N,\ep)]^{k+1}}
\eea
\vskip 6pt\noindent
where $\theta_{\ell,k}(\N,\ep)$ is defined by (\ref{Theta}). (\ref{stimaPsi}) and (\ref{Pll})  yield
\be
\label{alk}
A(\ell,k,\ep)= \gamma \frac{\tau^\tau}{(ed_\ell)^\tau}[1+\Pi(\ell,k)P(\ell,k,\ep)].
 \ee
 Set furthermore:
\bea
&&
\label{resto22}
 E({\ell}, k,\ep)
:=
\frac{\Psi(\ell,k)A(\ell,k,\ep)[2+
|\ep_{\ell}| \Psi(\ell,k)A(\ell,k,\ep)\|\V_{\ell,\ep}\|_{\rho_\ell,k}]}{1-|\ep_\ell | \Psi(\ell,k) A(\ell,k,\ep)\|\V_{\ell,\ep}\|_{\rho_\ell,k}}
 \eea
Then we have:
\begin{lemma}
\label{stimaVl+1}
Let:
 \be
 \label{stimadenE}
|\ep_\ell | \Psi(\ell,k)A(\ell,k,\ep)\|\V_{\ell,\ep}\|_{\rho_\ell,k}<1.
  \ee 
Then:
\be
 \label{restoll}
 \|V_{\ell+1,\ep}\|_{\rho_{\ell+1},k}\leq  E({\ell}, k,\ep)\|V_{\ell,\ep}\|^{2}_{\rho_{\ell},k}
  \ee
 \vskip 3pt\noindent
 \end{lemma}
 \begin{remark} The validity of the assumption (\ref{stimadenE}) is to be verified in Proposition \ref{estl} below. 
 \end{remark}
\begin{proof}
 By (\ref{Cdrk}), (\ref{stimaPsi}) and (\ref{alk}) we can write:
\be
C(\ell,k,\ep)\leq \Psi(\ell,k)A(\ell,k,\ep))\left[2+
|\ep_{\ell}| \Psi(\ell,k)A(\ell,k,\ep)\|\V_{\ell,\ep}\|_{\rho_\ell,k}\right]
\ee
and therefore, by (\ref{resto2}):
\begin{eqnarray*}
&&
\|V_{\ell+1,\ep}\|_{\rho_\ell-2d_\ell,k}\leq 
 C(\ell,k,\ep) \frac{\|\V_\ell\|^2_{\rho_\ell,k}} {1-|\ep_\ell| \Psi(\ell,k) A(\ell,k,\ep)\|\V\|_{\rho_\ell,k}}
\\
&&
{}
\\
&&
\leq \frac{\Psi(\ell,k)A(\ell,k,\ep)\left[2+
|\ep_{\ell}| \Psi(\ell,k)A(\ell,k,\ep)\|\V_{\ell,\ep}\|_{\rho_\ell,k}\right]}{1-|\ep_\ell | \Psi(\ell,k)A(\ell,k,\ep)\|\V_{\ell,\ep}\|_{\rho_\ell,k}}\|\V_\ell\|^2_{\rho_\ell,k}
\\
&&
= E(\ell,k,\ep)\|\V_\ell\|^2_{\rho_\ell,k}.
\end{eqnarray*}
\vskip 6pt\noindent
This yields (\ref{restoll}) and proves the Lemma.
\end{proof}
 Now recall that the sequence $\{\rho_j\}$ is decreasing. Therefore:
\be
\|\N_{j,\ep}\|_{\rho_\ell,k}\leq \|\N_{j,\ep}\|_{\rho_j,k}=
\|\overline{\V}_{j,\ep}\|_{\rho_j,k} 
\leq \|{\V}_{j,\ep}\|_{\rho_j,k}, \quad \;j=0,\ldots,\ell-1. 
\ee
 \vskip 4pt\noindent
At this point we can specify the sequence $d_\ell, \ell=1,2,\ldots$,   setting:
\vskip 4pt\noindent
\be
\label{ddelta}
d_\ell:=\frac{\rho}{(\ell+1)^2}, \qquad \ell=0,1,2,\ldots
\ee
\vskip 4pt\noindent
Remark that (\ref{ddelta}) yields 
 $$
d- \sum_{\ell=0}^\infty d_\ell=\rho-\frac{\pi^2}{6}>\frac{\rho}{2}.
$$
as well as  the following  estimate
\be
\label{stimapigreco}
\Pi(\ell,k)\leq  \frac{[2(k+1)^2]^{k+1}k^k(\ell+1)^{2k}}{e^{k}\rho^{k}}
\ee
 \vskip 2pt\noindent
We are now in position to discuss the convergence of the recurrence (\ref{restoll}). 
 \begin{proposition}
 \label{estl}
 Let: 
  \be
  \label{condrho}
 \rho> 2
\ee
 \be
 \label{condep}
 |\ep|< \ep^\ast(\tau,k):=
 \frac{1}{e^{24(2+k+\tau)}(k+2)^{2\tau}\|\V\|_{\rho,k}}
   \ee
 \vskip 4pt\noindent
\be
\label{3condep}
\gamma\tau^\tau (k+\tau+2)^{4(k+\tau+2)} < \frac{1}{2}
\ee
\vskip 8pt\noindent
Then the following estimate holds:
\vskip 8pt\noindent
 \begin{equation}
\label{rec2} 
\|\V_{\ell,\ep}\|_{\rho_\ell,k} \leq \left(e^{8(
2+k+\tau)}  \|\V\|_{\rho,k}\right)^{2^{\ell}}, \quad \ell=1,2,\ldots.
\end{equation}
\end{proposition}
\vskip 4pt\noindent
\begin{proof}
We proceed by induction. The assertion is true for $\ell=0$.  Now assume inductively:
\vskip 6pt\noindent
\be
\label{Hell}
 |\ep_j|\|\V_{j,\ep}\|_{\rho_j,k}\leq (k+2)^{-2\tau( j+1)},
\ee
\vskip 6pt\noindent
for $0\leq j\leq \ell$. Out of this we prove the validity of (\ref{rec2}) and of (\ref{stimadenE}); to complete the induction it will be enough to show that (\ref{rec2}) implies the validity of (\ref{Hell}) for $j=\ell+1$. 
\vskip 6pt
A preliminary result is the estimate of $ |\ep_\ell | \Psi_{\ell,k}A(\ell,k,\ep)\|\V\|_{\rho_\ell,k}$:
\begin{lemma}
\label{stimadenE1}
Let  \eqref{Hell} hold. Then:
\be
\label{stimadenE2}
|\ep_\ell | \Psi_{\ell,k} | A(\ell,k,\ep)\|\V\|_{\rho_\ell,k}\leq \frac1{2}.
\ee
\end{lemma}
\begin{proof}
Let  us first estimate $\theta_\ell(\N,\ep)$ as defined by \eqref{theta} assuming the validity of \eqref{Hell} . We obtain:
\begin{eqnarray*}
&&
\theta_\ell(\N,\ep)\leq \theta_{\ell,k}(\N,\ep) \leq \sum_{s=0}^{\ell-1}|\ep_s|\|\V\|_{\rho_s,k}/d_s
= \frac{1}{\rho}\sum_{s=0}^{\ell-1}\,(s+1)^2(k+2)^{-2\tau  (s+1)}=
\\
&&
\frac{1}{4\rho}\frac{d^2}{d\tau^2}\sum_{s=0}^{\ell-1}\,(k+2)^{-2\tau  (s+1)}
=\frac{1}{4\rho}\frac{d^2}{d\tau^2}[(k+2)^{-2\tau}\frac{1-(k+2)^{-2\tau \ell}}{1-(k+2)^{-2\tau}} \leq \frac{1}{\rho}(k+2)^{-2}\leq \frac{1}{\rho}
\end{eqnarray*}
because $\tau>l-1\geq 1$. 
   Now $\rho>1$  entails that 
  \be
  \label{dentheta}
  \frac1{1-\theta_\ell}<\frac{\rho}{\rho-1}.
  \ee
  \vskip 4pt\noindent
   Hence we get,  by (\ref{Pll}) and (\ref{Theta}), the further $(\ell,\ep)-$in\-de\-pen\-dent estimate:
\be
\label{Hells}
P(\ell,k,\ep)\leq \frac{\rho^{k+1}}{(\rho-1)^{k+1}}\left((k+2)^2{\rho}\right)^{-k-1}\leq \left(\frac{1}{(k+2)^2}\right)^{k+1}.
\ee
whence, by (\ref{alk}):
\begin{eqnarray}
\nonumber
&&
A(\ell,k,\ep)\leq \gamma\frac{\tau^\tau (\ell+1)^{2\tau}}{(e\rho)^\tau}[1+[2(k+1)^2]^{k+1}\left[(k+2)^2\right]^{-(k+1)}k^{k}(\ell+1)^{2k}]
\\
\label{stimaAell}
&&
\leq 4\gamma\frac{\tau^\tau (\ell+1)^{2(\tau+k)}k^k}{(e\rho)^\tau}.
\end{eqnarray}
\vskip 5pt\noindent
Upon application of the inductive assumption and \eqref{stimaAell} we get:
\vskip 5pt\noindent
\begin{eqnarray*}
&&
|\ep_\ell | \Psi_{\ell,k}A(\ell,k,\ep)\|\V\|_{\rho_\ell,k}\leq \frac{ 4^{k} (k+1)}{e^{2}\rho^{2}}(\ell+1)^{4}|\ep_\ell | A(\ell,k,\ep)\|\V\|_{\rho_\ell,k}
\\
&&
\leq \gamma\tau^{\tau} \frac{ 4^{k+1} (k+1)}{(e\rho)^{\tau+2}} (\ell+1)^{2(k+\tau+2)}k^k|\ep_\ell |\|\V\|_{\rho_\ell,k}
\\
&&
\leq \gamma\tau^{\tau} \frac{ 4^{k+1} (k+1)}{(e\rho)^{\tau+2}} (\ell+1)^{2(k+\tau+2)}k^k(k+2)^{-2(\ell+1)\tau} 
\end{eqnarray*}
\vskip 5pt\noindent
whence
\vskip 4pt\noindent
\bea
\label{stimaAell2}
&&
|\ep_\ell | \Psi_{\ell,k}A(\ell,k,\ep)\|\V\|_{\rho_\ell,k}\leq
\gamma\tau^{\tau} \frac{ 4^{k+1} (k+1)}{(e\rho)^{\tau+2}}k^k\kappa(k,\tau)^{2(k+\tau+2)}(k+2)^{-2\kappa(k,\tau)}\\
&&
\kappa(k,\tau):=\frac{k+\tau+2}{\tau\ln{(k+\tau)}}
\eea
because
$$
\sup_{\ell\geq 0}
(\ell+1)^{2(\tau+k+2)}(k+2)^{-2(\ell+1)\tau} =\kappa(k,\tau)^{2(k+\tau+2)}(k+2)^{-2\kappa(k,\tau)}.
$$
\vskip 4pt\noindent
Hence:
\be
\label{stimaBpsi}
|\ep_\ell | \Psi_{\ell,k}A(\ell,k,\ep)\|\V\|_{\rho_\ell,k}\leq \frac1{2}
\ee
provided \eqref{condrho} and \eqref{3condep} hold. As a matter of fact:
\vskip 5pt\noindent
\begin{eqnarray*}
&&
|\ep_\ell | \Psi_{\ell,k}A(\ell,k,\ep)\|\V\|_{\rho_\ell,k} \leq \gamma\tau^{\tau} \frac{ 4^{k+1} (k+1)}{(e\rho)^{\tau+2}}k^k\kappa(k,\tau)^{2(k+\tau+2)}(k+2)^{-2\kappa(k,\tau)}
\\
&&
 \leq  \gamma\tau^{\tau} [4(k+1)]^{k+1}(k+\tau+2)^{2(k+\tau+2)} \leq \gamma\tau^{\tau} (k+\tau+2)^{4(k+\tau+2)}
\end{eqnarray*}
\vskip 5pt\noindent
because  $e\rho >1$, $4(k+1)< (k+\tau+2)^2$ since $\tau\geq 2$, and $\kappa(k,\tau)\leq (k+\tau+2)$.   Hence \eqref{stimaBpsi} is implied by the inequality
\be
\label{4condep}
\gamma\tau^{\tau} (k+\tau+2)^{4(k+\tau+2)}
<\frac12
\ee
\vskip 6pt\noindent
which is \eqref{3condep}. 
    The Lemma is proved.
\end{proof}
\vskip 4pt\noindent
{\it Proof of Proposition \ref{estl}}.  By (\ref{resto22}):
  $$
 E(\ell,k,\ep) \leq 5 \Psi_{\ell,k}A(\ell,k,\ep) \leq 20 \gamma\tau^\tau (\ell+1)^{2(\tau+k)}k^k \Psi_{\ell,k}
   $$
 once more because $e\rho>1$. 
   (\ref{restoll}) in turn entails:
$$
\|\V_{\ell+1,\ep}\|_{\rho_\ell+1,k}\leq \Phi_{\ell,k}
\|\V_{\ell,\ep}\|_{\rho_\ell,k}^2, \quad \Phi_{\ell,k}:=20 \gamma\tau^\tau (\ell+1)^{2(\tau+k)} \Psi_{\ell,k}. 
$$
This last inequality immediately yields
\be
\label{rec3}
\|\V_{\ell+1,\ep}\|_{\rho_{\ell+1},k} \leq [\|\V\|_{\rho,k}]^{2^{\ell+1}}\prod_{m=0}^{\ell}\Phi_{\ell -m,k}^{2m}.
\ee
\vskip 3pt\noindent
Now:
\begin{eqnarray*}
&&
\Phi_{\ell,k}= 20 \gamma\tau^\tau (\ell+1)^{2(\tau+k)}k^k \frac{(k+1)4^{k}}{(ed_{\ell})^2}\leq \nu(k,\tau)(\ell+1)^{2(k+\tau+2)}
\\
&&
\nu(k,\tau):=20\gamma \tau^{\tau} 4^{k}(k+1)k^k
\end{eqnarray*} 
\vskip 5pt\noindent
Now  the following inequality is easily checked:
\vskip 4pt\noindent
\bea
\label{gammanu}
\nu(k,\tau)= 20\gamma\tau^\tau (k+1)4^k k^k \leq 2 \gamma\tau^\tau
(k+\tau+2)^{4(k+\tau+2)}
\eea
\vskip 5pt\noindent
because $\tau\geq 2$, and therefore, by \eqref{3condep}, \eqref{condrho} we get: $\nu(k,\tau)\leq 1$. 
 As a consequence we have
 \be
 \Phi_{\ell,k}(\ell+1)^{2(k+\tau+2)}
\ee
  Moreover, since $\Phi_{j,k}\leq \Phi_{\ell,k}, j\leq \ell$, we get, by \eqref{gammanu}: 
 \vskip 5pt\noindent
 \begin{eqnarray*}
\prod_{m=1}^{\ell}\Phi^{2m}_{\ell+1-m,k} \leq  [\Phi_{\ell,k}]^{\ell(\ell+1)}\leq  (\ell+1)^{2(k+\tau+2)\ell(\ell+1)}
\end{eqnarray*} 
\vskip 3pt\noindent
 Now using $\ell(\ell+1)\log{\ell+1}<4\times2^{\ell+1}$, $\forall\,\ell\in\Na$, we get:
 $$
 (\ell+1)^{2(k+\tau+2)\ell(\ell+1)}
< [e^{8(k+\tau+2)}]^{2^{\ell+1}}.
 $$
The following estimate is thus established
\bea
\label{stimapsi}
&&
\prod_{m=0}^{\ell}\Phi^{2m}_{\ell -m,k} \leq 
[e^{8(k+\tau+2)}] ^{2^{\ell+1}}.
\eea
If we now define:
\bea
\label{mu}
&&
\mu :=e^{8(k+\tau+2)}, 
 \eea
then  (\ref{rec3}) and (\ref{stimapsi}) yield: 
\vskip 6pt\noindent
\bea
&&
\label{GVS}
\|\V_{\ell+1,\ep}\|_{\rho_{\ell+1,k}} \leq \left[\mu^{2^l}\|\V_{\ell,\ep}\|_{\rho_\ell,k}\right]^{2}\leq \left[\|\V\|_{\rho,k}\,\mu\right]^{2^{\ell+1}}
\\
\label{GVSS}
&&
\mbox{ and therefore }\nonumber\\
&&
\ep_{\ell+1}\|\V_{\ell+1,\ep}\|_{\rho_{\ell+1,k}} \leq \left[ \|\V\|_{\rho_\ell,k}\,\mu^{2^l}\ep_\ell\right]^{2}
 \leq \left[ \|\V\|_{\rho,k}\,\mu\ep\right]^{2^{\ell+1}}
\eea
\vskip 5pt\noindent

\eqref{GVS} is exactly \eqref{rec2}. Let us now prove out of (\ref{GVS},\ref{GVSS}) that the condition (\ref{Hell}) preserves its validity also for $j=\ell+1$. We have indeed, by the inductive assumption (\ref{Hell}) and (\ref{GVS}):
\begin{eqnarray*}
\label{verifica}
&&
 |\ep_{\ell+1}| \|\V_{\ell+1,\ep}\|_{\rho_{\ell+1,k}} \leq   \left[ \|\V\|_{\rho_\ell,k}\,\mu^{2^l}\ep_\ell\right]^{2}\leq (k+2)^{-2\tau(\ell+1)}\ep_\ell(\mu^{2^l})^2\|\V\|_{\rho_\ell,k}
 \\
 &&
 \leq (k+2)^{-2\tau(\ell+1)}\left[\ep\mu^3\|\V\|_{\rho,k}\right]^{2^\ell}\leq (k+2)^{-2\tau(\ell+2)}
\end{eqnarray*}
provided
\vskip 4pt\noindent
\be
\label{epsast}
|\ep|< \frac{1}{\mu^3\|\V\|_{\rho,k}(k+2)^{2\tau}}= \frac{1}{e^{24(k+\tau+2)}\|\V\|_{\rho,k}(k+2)^{2\tau}}:=\ep^\ast(\tau,k)
\ee
\vskip 8pt\noindent
where the last expression follows from (\ref{mu}).  This proves (\ref{Hell}) for $j=\ell+1$, and concludes the proof of the Proposition.
\end{proof}
\noindent
  \begin{theorem}\label{final}[Final estimates of $W_\ell$, $N_\ell$, $V_\ell$]
 \newline
 Let $\V$ fulfill Assumption (H2-H4), and let \eqref{3condep} be verified. Then the following estimates hold,  $\forall \ell\in\Na$:
 \vskip 4pt\noindent
 \begin{eqnarray}
   \label{stimafw}
 \ep_\ell \|W_{\ell,\ep}\|_{\rho_{\ell+1},k}\leq   
 (\ell+1)^{2(\tau+k)}\cdot(\mu \ep \|\V\|_{\rho})^{2^{\ell}}.
  \end{eqnarray}
 \begin{eqnarray}
  \label{stimafn}
 \ep_\ell \|N_{\ell,\ep}\|_{\rho_\ell,k}\leq \ep_\ell \|\V_{\ell,\ep}\|_{\rho_\ell,k}\leq \left[ \|\V\|_{\rho}\,\ep \mu\right]^{2^{\ell}}.
 \end{eqnarray}
 \begin{eqnarray}
 \label{stimafv}
\ep_{\ell+1} \|V_{\ell+1,\ep}\|_{\rho_{\ell+1},k}\leq  
 \left[ \|\V\|_{\rho}\,\ep \mu\right]^{2^{\ell+1}}.
\end{eqnarray}
\end{theorem}
\begin{proof}
Since $\V$ does not depend on $\hbar$, obviously $ \|\V\|_{\rho,k}\equiv  \|\V\|_{\rho}$. Then formula (\ref{Thm5.1}) yields, on account of (\ref{stimaAell}), (\ref{dentheta}),  (\ref{3condep}),  
(\ref{GVS}), (\ref{GVSS}):
\vskip 3pt\noindent
 \begin{eqnarray*}
 \nonumber
&&
\label{stimaWfl}
\ep_\ell \|W_{\ell,\ep}\|_{\rho_{\ell+1},k} \leq \gamma\tau^\tau(\ell+1)^{2(k+\tau)}k^k \|\V\ell,\ep\|_{\rho_\ell,k}\leq 
\\ 
&&
\leq  2 \frac12 (\ell+1)^{2(k+\tau)}\cdot (\mu \ep \|\V\|_{\rho})^{2^{\ell}}.
 \end{eqnarray*}
 \vskip 5pt\noindent
This proves (\ref{stimafw}). 
Moreover, since $\N_{\ell,\ep}=\overline{\V}_{\ell,\ep}$, again by (\ref{GVS}), (\ref{GVSS}):
\be\nonumber
\label{stiman}
\ep_\ell \|\N_{\ell,\ep}\|_{\rho_\ell,k}= \ep_\ell \|\overline{\V}_{\ell,\ep}\|_{\rho_\ell,k}\leq \left[ \|\V\|_{\rho}\,\ep \mu\right]^{2^{\ell}}.
\ee
The remaining assertion follows  once more from 
(\ref{GVSS}).
 This concludes the proof of the Theorem.
 \end{proof}
 \begin{remark}
 \label{verifica1}
 (\ref{stimafw}) yields:
 \vskip 4pt\noindent
 $$
 \ep_\ell \frac{\|W_{\ell,\ep}\|_{\rho_{\ell+1},k}}{d_\ell}\leq 4\gamma\tau^\tau\ep^{2^\ell}(\ell+1)^{2(k+\tau+1)}\|\V\|_\rho^{2^\ell}
 $$
 This inequality in  turn entails:
 $$
 |\ep|\left(\frac{\|W_{\ell,\ep}\|_{\rho_{\ell+1},k}}{d_\ell}\right)^{2^{-\ell}}\leq [4\gamma\tau^\tau(\ell+1)^{2(k+\tau+1)}]^{2^{-\ell}}\|\V\|_{\rho}\to \|\V\|_\rho, \quad \ell\to\infty
 $$
 so that (\ref{condepell}) is actually fulfilled for $\ds |\ep|<  \frac1{\|\V\|_\rho}.$
 \end{remark}
 \begin{corollary}
 \label{maincc}
 In the above assumptions set:
 \be
 \label{Un}
 U_{n,\ep}(\hbar):= \prod_{s=0}^ne^{i\ep_{n-s}W_{n-s,\ep}}, \quad n=0,1,\ldots.
 \ee
  Then:
  \begin{enumerate}
  \item $U_{n,\ep}(\hbar)$ is a unitary operator in $L^2(\T^l)$, with
  $$
  U_{n,\ep}(\hbar)^\ast=U_{n,\ep}(\hbar)^{-1}=\prod_{s=0}^ne^{-i\ep_{s}W_{s,\ep}}
  $$
  \item Let:
 \be
 S_{n,\ep}(\hbar):=U_{n,\ep}(\hbar)(\L_\om+\ep V)U_{n,\ep}(\hbar)^{-1} 
 \ee 
 Then:
 \bea
S_n&=&D_{n,\ep}(\hbar)+\ep_{n+1}V_{n+1,\ep}
\\
D_{n,\ep}(\hbar)&=&L_\om+\sum_{s=1}^n\ep_sN_{s,\ep}
\eea
The corresponding symbols are:
\bea
&&
{\mathcal S}_n(\xi,x;\hbar)=\D_{n,\ep}(\L_\om(\xi),\hbar)+\ep_{n+1}V_{n+1,\ep}(\L_\om(\xi),x;\hbar)
\\
\label{sumD}
&&
\D_{n,\ep}(\L_\om(\xi),\hbar)=\L_\om(\xi)+\sum_{s=1}^n \ep_s\N_{s,\ep}(\L_\om (\xi),\hbar).
\eea
 Here the operators $W_{s,\ep}$, $N_{s,\ep}$, $V_{\ell+1,\ep}$ and their symbols $\W_{s,\ep}$, $\N_{s,\ep}$, $\V_{\ell+1,\ep}$ fulfill the above estimates.
 \item 
 Let $\ep^\ast$ be defined as in (\ref{condep}). 
Remark that $ \ep^\ast(\cdot,k)> \ep^\ast(\cdot,k+1),  \,k=0,1,\ldots$.
Then, if $|\ep|<\ep(k,\cdot)$:
\be
\lim_{n\to\infty}\D_{n,\ep}(\L_\om(\xi),\hbar)=\D_{\infty,\ep}(\L_\om(\xi),\hbar)
\ee
where in the convergence takes place in the $C^k([0,1];C^\om (\rho/2))$ topology, namely
\be
\label{limD}
\lim_{n\to\infty}\|\D_{n,\ep}(\L_\om(\xi),\hbar)-\D_{\infty,\ep}(\L_\om(\xi),\hbar)\|_{\rho/2,k}=0. 
\ee
 \end{enumerate}
 \end{corollary}
 \begin{proof}
 Since Assertions (1) and (2) are straightforward, we limit ourselves to the simple verification of Assertion (3).  If $|\ep|<\ep^\ast(\tau,k)$ then  $\ds \|V\|_{\rho,k}\mu
 \ep < \Lambda<1$. Recalling that $\|\cdot\|_{\rho,,k} \leq \|\cdot\|_{\rho^\prime,k}$ whenever $\rho\leq \rho^\prime$,  and that $\rho_\ell <\rho/2$,  $\forall\,\ell \in {\Bbb N}$, (\ref{stimafv}) yields:
\begin{eqnarray*}
&&
\ep_{n+1}\|\V_{n+1,\ep}\|_{\rho/{2},k}\leq \ep_{n+1}\|\V_{n+1,\ep}\|_{\rho_{n+1},k}\leq
\\
&&
\left[\|V\|_{\rho,k}\mu
 \ep\right]^{2^{n+1}}\to 0, \quad n\to\infty, \;k\;{\rm fixed}.
 \end{eqnarray*}
 In the same way, by (\ref{stimafn}):
 \begin{eqnarray*}
&&
\|\N_{n,\ep}\|_{\rho/{2},k}\leq \|\N_{n,\ep}\|_{\rho_{n},k}=
\|\overline{\V}_{n,\ep}\|_{\rho_{n},k}\leq
\|\V_{n,\ep}\|_{\rho_{n},,k}\leq
\\
&&
 \left[\|V\|_{\rho,k}\mu
 \ep\right]^{2^{n}}=\left[\|V\|_{\rho}\mu
 \ep\right]^{2^{n}}\to 0, \quad n\to\infty, \;k\;{\rm fixed}. \ \end{eqnarray*}
This concludes the proof of the Corollary. 
 \end{proof}
 \vskip 1cm\noindent
\section{Convergence of the iteration and of the normal form.}
\label{iteration}
\renewcommand{\thetheorem}{\thesection.\arabic{theorem}}
\renewcommand{\theproposition}{\thesection.\arabic{proposition}}
\renewcommand{\thelemma}{\thesection.\arabic{lemma}}
\renewcommand{\thedefinition}{\thesection.\arabic{definition}}
\renewcommand{\thecorollary}{\thesection.\arabic{corollary}}
\renewcommand{\theequation}{\thesection.\arabic{equation}}
\renewcommand{\theremark}{\thesection.\arabic{remark}}
\setcounter{equation}{0}%
\setcounter{theorem}{0}%
Let us first prove the uniform convergence of the unitary transformation sequence as $n\to\infty$.  Recall that $\ep^\ast(\tau,k)> \ep^\ast(\tau,k+1),  \; k=0,1,\ldots$, and recall the abbreviation $\|\cdot\|_{\rho,0}:=\|\cdot\|_{\rho}$.  Let $\ep^\ast(\tau)$ be defined by \eqref{rast}.  Then:
\begin{lemma}
\label{Wsequence}
Let  $\hbar$ be fixed, and $\ds |\ep|<\ep^\ast(\tau)$. Consider the sequence  $\ds \{U_{n,\ep}(\hbar)\}$ of unitary operators  in $L^2(\T^l)$ defined by (\ref{Un}).
Then there is a unitary operator $U_{\infty,\ep}(\hbar)$ in $L^2(\T^l)$ such that
$$
\lim_{n\to\infty}\|U_{n,\ep}(\hbar)-U_{\infty,\ep}(\hbar)\|_{L^2\to L^2}=0
$$
\end{lemma}
\begin{proof}
We have, for $p=1,2,\ldots$:
\begin{eqnarray*}
&&
U_{n+p,\ep}-U_{n,\ep}=\Delta_{n+p,\ep}e^{i\ep_n \frac{W_n}\hbar}\cdots e^{i\ep \frac{W_1}\hbar}, \quad \Delta_{n+p,\ep}:=(e^{i\ep_{n+p}\frac{W_{n+p}}\hbar}\cdots e^{i\ep_{n+1}\frac{W_{n+1}}\hbar}-I)
\\
&&
\|U_{n+p,\ep}-U_{n,\ep}\|_{L^2\to L^2}\leq 2\|\Delta_{n+p,\ep}\|_{L^2\to L^2}
\end{eqnarray*}
Now we apply the mean value theorem and obtain
$$
e^{i\ep_\ell \frac{W_{\ell,\ep}}\hbar}=1+\beta_{\ell,\ep} \quad \beta_{\ell,\ep}:=i\ep_\ell \frac{W_{\ell,\ep}}\hbar
\int_0^{\ep_\ell}e^{i\ep^\prime_\ell\frac{ W_{\ell,\ep}}\hbar}\,d\ep^\prime_\ell ,
$$
whence, by  (\ref{stimafw}) in which we make $k=0$:
\be
\label{stimaesp}
\hbar\|\beta_{\ell,\ep}\| \leq \ep_\ell \|W_{\ell,\ep}\|_{\rho_{\ell}} = \ep_\ell \|\W_{\ell,\ep}\|_{\rho_{\ell},0}
\leq  4\gamma\tau^\tau (\ell+1)^{2\tau}
 \leq A^\ell
\ee
for some $A<1$.  Now:
\begin{eqnarray*}
&&
\Delta_{n+p,\ep}=[(1+\beta_{n+p,\ep}\ep_{n+p})(1+\beta_{n+p-1,\ep}\ep_{n+p-1})\cdots (1+\beta_{n+1,\ep}\ep_{n+1})-1]=\sum_{1\leq j\leq p}\beta_{n+j,\ep}\ep_{n+j}
\\
&&
+\sum_{1\leq j_1<j_2\leq p}\beta_{n+j_1,\ep}\ep_{n+j_1}\beta_{n+j_2,\ep}\ep_{n+j_2}+
\sum_{1\leq j_1<j_2<j_3\leq p}\beta_{n+j_1,\ep}\ep_{n+j_1}\beta_{n+j_2,\ep}\ep_{n+j_2}\beta_{n+j_3,\ep}\ep_{n+j_3}
\\
&&
+\ldots +\beta_{n+1,\ep}\cdots\beta_{n+p,\ep}\ep_{n+1}\cdots\ep_{n+p}
\end{eqnarray*}
Therefore, by (\ref{stimaesp}):
\begin{eqnarray*}
&&
\|\Delta_{n+p,\ep}\|_{L^2\to L^2}\leq \sum_{1\leq j\leq p}\frac{A^{n+j}}\hbar+\sum_{1\leq j_1<j_2\leq p}\frac{A^{n+j_1}A^{n+j_2}}{\hbar^2}+\sum_{1\leq j_1<j_2<j_3\leq p}\frac{A^{n+j_1}A^{n+j_2}A^{n+j_3}}{\hbar^3}+\ldots 
\\
&&
\leq \frac{A^{n}}\hbar\frac{A}{1-A}+\frac{A^{2n}}{\hbar^2}\left(\frac{A}{1-A}\right)^2+\ldots +\frac{A^{pn}}{\hbar^p}\left(\frac{A}{1-A}\right)^p
\\
&&
\leq \frac{A^n}{\hbar(1-A)}\frac{1}{1-\frac{A^n}{\hbar(1-A)}}\ \ \ \mbox{ for } n>\frac{\log{(\hbar(1-A))}}{\log{A}}.
\end{eqnarray*}
Therfeore 
\[
\Delta_{n+p,\ep}\to 0,\quad n\to\infty,\quad \forall\,p,\hbar>0.
\]
\vskip 5pt\noindent
Hence $\{U_{n,\ep}(\hbar)\}_{n\in{\Bbb N}}$ is a Cauchy sequence in the operator norm, uniformly with respect to $|\ep|<\ep^\ast_0$,  and the Lemma is proved.
\end{proof}
We are now in position to prove existence and analyticity of the limit of the KAM iteration, whence the uniform convergence of the QNF. 
\vskip 0.3cm\noindent
{\bf Proof of Theorems \ref{mainth} and \ref{regolarita}}
\newline
The operator family $H_\ep$ is self-adjoint in $L^2(T^l)$ with pure point spectrum $\forall\,\ep\in\R$ because $V$ is a continuous operator. 
By  Corollary \ref{maincc}, the operator sequence $\{D_{n,\ep}(\hbar)\}_{n\in {\Bbb N}}$ admits for $|\ep|<\ep^\ast_0$ the  uniform norm limit 
$$
D_{\infty,\hbar}(L_\om,\hbar)=L_\om+\sum_{m=0}^\infty\ep^{2^m}N_{m,\ep}(L_\om,\hbar)
$$ 
of symbol $\D_{\infty,\hbar}(\L_\om(\xi))$.  The series is norm-convergent by (\ref{stimafn}). 
By Lemma (\ref{Wsequence}), $D_{\infty,\hbar}(L_\om,\hbar)$ is unitarily equivalent to $H_\ep$. The 
  operator family $\ep\mapsto D_{\infty,\ep}(\hbar)$ is holomorphic for $|\ep|<\ep^\ast_0$, uniformly with respect to $\hbar\in[0,1]$.   As a consequence,  $D_{\infty,\ep}(\hbar)$ admits the norm-convergent expansion:
 $$
 D_{\infty,\ep}(L_\om,\hbar)=L_\om+\sum_{s=1}^\infty B_s(L_\om,\hbar)\ep^s, \quad |\ep|<\ep^\ast(\tau)
 $$
which is  the convergent quantum normal form. 

On the other hand, (\ref{limD}) entails that the symbol $\D_{\infty,\ep}(\L_\om(\xi),\hbar)$ is a $\J(\rho/2)$-valued holomorphic function of $\ep$, $|\ep|<\ep^\ast (\tau)$, continuous with respect to $\hbar\in [0,1]$.  Therefore it admits the  expansion 
\be
\label{fnormale}
\D_{\infty,\ep}(\L_\om(\xi),\hbar)=\L_\om(\xi)+\sum_{s=1}^\infty{\mathcal B}_s(\L_\om(\xi),\hbar)\ep^s, \quad |\ep|<\ep^\ast(\tau)
 \ee
  convergent in the $\|\cdot\|_{\rho/2}$-norm, with radius of convergence $\ep^\ast(\tau)$.  Hence, in the notation of Theorem \ref{mainth},   $\D_{\infty,\ep}(\L_\om(\xi),\hbar)\equiv \B_{\infty,\ep}(\L_\om(\xi),\hbar)$.  By construction, ${\mathcal B}_s(\L_\om(\xi),\hbar)$ is the symbol of $B_s(L_\om,\hbar)$.  $\B_{\infty,\ep}(\L_\om(\xi),\hbar)$  is the symbol yielding the quantum normal form  via Weyl's quan\-ti\-za\-tion.  Likewise, the symbol $\W_{\infty,\ep}(\xi,x,\hbar)$ is a $\J(\rho/2)$-valued holomorphic function of $\ep$, $|\ep|<\ep^\ast(\tau)$,  continuous with respect to $\hbar\in [0,1]$, and admits the expansion:
\be
\label{fgen}
\W_{\infty,\ep}(\xi,x,\hbar)=\la\xi,x\ra+\sum_{s=1}^\infty{\mathcal W}_s(\xi,x,\hbar)\ep^s, \quad |\ep|<\ep^\ast
\ee
convergent in the $\|\cdot\|_{\rho/2}$-norm, once more with radius of convergence $\ep^\ast(\tau)$.  Since Since $\|\B_s\|_1 \leq \|\B_s\|_{\rho/2}$,  $\|{\mathcal W}_s \|_1\leq \|{\mathcal W}_s\|_{\rho/2} $ $\forall\,\rho>0$.   By construction,  $\B_{\infty,\ep}(\xi,x,\hbar)=\B_{\infty,\ep}(t,x,\hbar)|_{t=\L_\om(\xi)}$.  Theorem \ref{mainth} is proved.

  Remark furthermore that the principal symbol of  $\B_{\infty,\ep}(\L_\om(\xi),\hbar)$ is just the convergent Birkhoff normal form:
 $$
  \B_{\infty,\ep}=\L_\om(\xi)+\sum_{s=1}^\infty{\mathcal B}_s(\L_\om(\xi))\ep^s, \quad |\ep|<\ep^\ast(\tau)
  $$
\vskip 4pt
Theorem (\ref{regolarita}) is a direct consequence of (\ref{limD}) on account of the fact that
$$
\sum_{\gamma=0}^r\max_{\hbar\in [0,1]} \|\partial^\gamma_\hbar \B_\infty(t;\ep,\hbar) \|_{\rho/2}\leq \|\B_\infty\|_{\rho/2,k}
$$ 
Remark indeed that  by (\ref{limD}) the series (\ref{fnormale}) converges in the $\|\cdot\|_{\rho/2,r}$ norm if $|\ep|<\ep^\ast(\tau,r)$. Therefore $\B_s(t,\hbar)\in C^r([0,1];C^\om(\{t\in\C\,|\,|\Im t|<\rho/2\})$ and the formula (\ref{EQF}) follows from (\ref{fnormale}) upon Weyl quantization. This concludes the proof of the Theorem. 
\vskip 1.0cm\noindent
\begin{appendix}
\section{The quantum normal form}
\renewcommand{\thetheorem}{\thesection.\arabic{theorem}}
\renewcommand{\theproposition}{\thesection.\arabic{proposition}}
\renewcommand{\thelemma}{\thesection.\arabic{lemma}}
\renewcommand{\thedefinition}{\thesection.\arabic{definition}}
\renewcommand{\thecorollary}{\thesection.\arabic{corollary}}
\renewcommand{\theequation}{\thesection.\arabic{equation}}
\renewcommand{\theremark}{\thesection.\arabic{remark}}
\setcounter{equation}{0}%
\setcounter{theorem}{0}%
\noindent
The quantum normal form in the framework of semiclassical analysis has been introduced by Sj\"ostrand \cite{Sj}. We follow here the presentation of \cite{BGP}.  
\vskip 6pt\noindent
{\bf 1. The formal construction}
Given the operator family $\ep\mapsto H_\ep=L_\om+\ep V$, look for a unitary 
transformation $\ds U(\om,\ep,\hbar)=e^{i W(\ep)/\hbar}:
L^2(\T^l)\leftrightarrow L^2(\T^l)$,
$W(\ep)=W^\ast(\ep)$,  such that: 
\be
\label{A1}
S(\ep):=UH_\ep U^{-1}=L(\om)+\ep B_1+\ep^2 B_2+\ldots+
\ep^k R_k(\ep) 
\ee
where $[B_p,L_0]=0$, $p=1,\ldots,k-1$. Recall the formal commutator expansion: 
\be
S(\ep)=e^{it W(\ep)/\hbar}He^{-it W(\ep)/\hbar}=\sum_{l=0}^\infty t^lH_l,\quad H_0:=H,\quad
H_l:=\frac{[W,H_{l-1}]}{i\hbar l}, \;l\geq 1
\label{A2}
\ee
and look for $W(\ep)$ under the form of a power series: 
$W(\ep)=\ep W_1+\ep^2W_2+\ldots$. Then   (\ref{A2}) becomes:
\be
\label{A3}
S(\ep)=\sum_{s=0}^{k-1}\ep^s  P_s +\ep^{k}{R}^{(k)}
\ee
where
\be
\label{A4}
P_0=L_\om;\quad {P}_s:=\frac{[W_s,H_0]}{i\hbar}+V_s,\quad s\geq 1, \;V_1\equiv V 
\ee
\begin{eqnarray*}
V_s =\sum_{r=2}^s\frac{1}{r!}\sum_{{j_1+\ldots+j_r=s}\atop {j_l\geq
1}}\frac{[W_{j_1},[W_{j_2},\ldots,[W_{j_r},H_0]\ldots]}{(i\hbar)^r} 
+\sum_{r=2}^{s-1}\frac{1}{r!}\sum_{{j_1+\ldots+j_r=s-1}\atop {j_l\geq
1}}\frac{[W_{j_1},[W_{j_2},\ldots,[W_{j_r},V]\ldots]}{(i\hbar)^r} 
\end{eqnarray*}
\begin{eqnarray*}
{R}^{(k)}=\sum_{r=k}^\infty\frac{1}{r!}\sum_{{j_1+\ldots+j_r=k}\atop {j_l\geq
1}}\frac{[W_{j_1},[W_{j_2},\ldots,[W_{j_r},L_\om]\ldots]}{(i\hbar)^r} 
+\sum_{r=k-1}^{\infty}\frac{1}{r!}\sum_{{j_1+\ldots+j_r=k-1}\atop {j_l\geq
1}}\frac{[W_{j_1},[W_{j_2},\ldots,[W_{j_r},V]\ldots]}{(i\hbar)^r}
\end{eqnarray*}
Since  $V_s$ depends on $W_1,\ldots,W_{s-1}$,  (A1) and (A3) yield the
recursive homological equations:
\be
\label{A5}
\frac{[W_s,P_0]}{i\hbar} +V_s=B_s, \qquad [L_0,B_s]=0 
\ee
To solve for $S$, $W_s$, $B_s$, we can equivalently look for their symbols.  The 
 equations (\ref{A2}), (\ref{A3}), (\ref{A4}) 
become, once written for the symbols:
\be
 \label{A6}
\Sigma(\ep)=\sum_{l=0}^\infty {\H}_l,\quad {\H}_0:=\L_\om+\ep
\V,\quad  {\H}_l:=\frac{\{w,{\H}_{l-1}\}_M}{ l}, \;l\geq 1\ee
\be
\label{A7}
\Sigma(\ep)=\sum_{s=0}^{k}\ep^s {\mathcal P}_s +\ep^{k+1}{\mathcal R}^{(k+1)} 
\ee
where
\be
\label{A8}
{\mathcal P}_0=\L_\om;\qquad {\mathcal P}_s :=\{\W_s,{\mathcal P}_0 \}_M+\V_s,\quad s=1, \ldots,\qquad  \V_1\equiv \V_0=\V \ee
\begin{eqnarray*}
&&
\V_s :=\sum_{r=2}^s\frac{1}{r!}\sum_{{j_1+\ldots+j_r=s}\atop {j_l\geq
1}}\{\W_{j_1},\{\W_{j_2},\ldots,\{\W_{j_r},\L_\om\}_M\ldots\}_M + 
\\
&&
+\sum_{r=1}^{s-1}\frac{1}{r!}\sum_{{j_1+\ldots+j_r=s-1}\atop {j_l\geq
1}}\{\W_{j_1},\{\W_{j_2},\ldots,\{\W_{j_r},\V\}_M\ldots\}_M, \quad s>1 \end{eqnarray*}
\begin{eqnarray*}
&&
{\mathcal R}^{(k)}=\sum_{r=k}^\infty\frac{1}{r!}\sum_{{j_1+\ldots+j_r=k}\atop {j_l\geq
1}}\{\W_{j_1},\{\W_{j_2},\ldots,\{\W_{j_r},\L_\om\}_M\ldots\}_M+
\\
&&
\sum_{r=k-1}^{\infty}\frac{1}{r!}\sum_{{j_1+\ldots+j_r=k-1}\atop {j_l\geq
1}}\{\W_{j_1},\{\W_{j_2},\ldots,\{\W_{j_r},\V\}_M\ldots\}_M\end{eqnarray*}
In turn, the recursive homological equations become:
\be
\label{A9}
\{\W_s,\L_{\om}\}_M +\V_s=\B_s, \qquad \{\L_{\om},\B_s\}_M =0 
\ee
\vskip 6pt\noindent
{\bf 2. Solution of the homological equation and estimates of the solution} 
\vskip 3pt\noindent
The key
remark is that
$\{\A,\L_\om\}_M=\{\A,\L_\om\}$ for any smooth symbol $\A(\xi;x;\hbar)$ because $\L_\om$ is linear in $\xi$. The
homological equation  (A.9) becomes therefore
\be
\label{A10}
\{\W_s,\L_\om\} +\V_s=\B_s, \qquad \{\L_\om,\B_s\} =0 \ee
We then have:
\begin{proposition}
  Let $\V_s(\xi,x;\hbar)\in\J(\rho_s)$. Then  the
equation
\be
\label{A11}
\{\W_s,\L_\om\} +\V_s=\B_s, \qquad \{\L_\om,\B_s\} =0 \ee
admits $\forall\,0<d_s<\rho_s$ the solutions $\B_s(\L_\om(\xi;)\hbar)\in \J(\rho_s)$, $\W\in\J(\rho-d_s)$ given by:
\be
\label{A12}
 \B_s(\xi;\hbar)=\overline{\V_s}; \quad \W_s(\xi,x;\hbar)=\L_\om^{-1}\V_s, \quad \L_\om^{-1}\V_s:=\sum_{0\neq\in\Z^l
}\frac{\V_{s.q}(\L_\om(\xi))}{i\la\om,q\ra}e^{i\la q,x\ra}.  
\ee 
 Moreover:
\be
\label{stimaWs}
\|\B_s\|_{\rho_s}\leq  \|\V_s\|_{\rho_s}; \qquad 
\|\W_s\|_{\rho_s-d_s} \leq \gamma \left(\frac{\tau}{d_s}\right)^\tau \|\V_s\|_{\rho_s}.
\ee
\end{proposition}
\begin{proof}  $\B_s$ and $\W_s$ defined by (A.12) clearly solve the homological equation (A.11). The estimate for $\B_s$ is obvious, and the estimate for $\W_s$ follows once more by the small denominator inequality (\ref{DC}). 
\end{proof}
By definition of $\|\cdot\|_{\rho}$ norm:
\be
\label{A13}
\|B_s\|_{L^2\to L^2}\leq \|B_s\|_{\rho} \leq \|\V_s\|_{\rho_s}; \quad \|B_s\|_{L^2\to\L^2}\leq \|B_s\|_{\rho} \leq \|\V_s\|_{\rho_s}
\ee
  Hence all terms of the quantum normal form and the remainder can be recursively estimated  in terms of $\|\V\|_{\rho}$ by Corollary 3.11.  Setting now, for $s\geq 1$:
\begin{eqnarray*}
&&
  \rho_s:=\rho- s d_s, \quad d_s< \frac{\rho}{s+1}; \qquad \rho_0:=\rho
\\
&&
\mu_s:=8\gamma \tau^\tau \frac{E}{d_s^\tau\delta_s^2}, \quad E:=\|\V\|_{\rho}. 
\end{eqnarray*}
  we actually have, applying without modification the  argument of \cite{BGP}, Proposition 3.2:
 \begin{proposition}
  Let $\mu_s<1/2, s=1,\ldots,k$.  Set:
 $$
 K:=\frac{8\cdot 2^{\tau+5}\gamma\tau^\tau}{\rho^{2+\tau}}.
 $$
 Then the following estimates hold for the quantum normal form: 
 \begin{eqnarray*}
 &&
\sum_{s=1}^k \|B_s\|_{\rho/2}\ep^s \leq \sum_{s=1}^k \|\B_s\|_{\rho/2}\ep^s\leq \sum_{s=1}^k E^sK^s s^{(\tau+2)s}\ep^s
\\
&&
{}
\\
&&
\|R_{k+1}\|_{\rho/2}\leq \|{\mathcal R}_{k+1}\|_{\rho/2}\leq (EK)^{k+1}(k+1)^{(\tau+2)(k+1)}\ep^{k+1}
\end{eqnarray*}
\end{proposition}
\vskip 1cm\noindent

\end{appendix}

\newpage


\vskip 1.0cm\noindent

\begin{thebibliography}{GGGG}
\bibitem[Ar1]{Ar1} V.I.Arnold {\it Mathematical Methods of Classical Mechanics}, Springer-Verlag, 1983
\bibitem[Ar2]{Ar2} V.I.Arnold {\it Geometric  Methods in the Theory of Ordinary Differential Equations}, Chapter V, Springer-Verlag, 1985
\bibitem[BGP]{BGP} D:Bambusi, S.Graffi, T.Paul {\it Normal Forms and Quantization Formulae}, Comm.Math.Phys. {\bf 207}, 173-195 (1999)
\bibitem[Ca]{Ca}  J.R.Cary, {\it Lie transform perturbation theory for Hamiltonian systems},  Phys.Reports {\bf 79}, 15-52 (1977)
\bibitem[CdV]{CdV}  Y.Colin de Verdi\`ere, {\it Quasi-modes sur les varietes Riemanniennes},  Inventiones Mathematicae {\bf 43}, 129-159 (1981)
\bibitem[De]{De}  A.Deprit, {\it Canonical transformations depending on a small parameter}, Cel.Mech.Dyn.Astr. {\bf 1}, 12-30 (1969)
\bibitem[FM]{FM} F.M.Fedoryuk, V.P.Maslov, {\it Semi-classical approximation in quantum mechanics}  D.Reidel,  1981
\bibitem[Fo]{Fo}  G.Folland, {\it Harmonic Analysis in Phase Space},  Annals of Mathematics Studies {\bf 122}, Princeton University Press 1989
\bibitem[Ga]{Ga} G.Gallavotti,  {\it Criterion of Integrability for Perturbed Nonresonant Harmonic Oscillators, "Wick Ordering" of the Perturbations in Classical Mechanics and Invariance of the Frequency Spectrum}, Comm.Math.Phys.{\bf 87}, 365-383 (1982)
\bibitem[GV]{GV} S.Graffi, C.Villegas-Blas, {\it A Uniform Quantum Version of the Cherry Theorem}, Comm.Math.Phys. (2008)
\bibitem[Ho]{Ho}  G.I.Hori, {\it Theory of general perturbations with unspecified canonical variables},  Publ.Astr.Soc.Japan {\bf 18}, 287-296 (1966)
\bibitem[Ma]{Ma} V.P.Maslov, {\it Theorie des perturbations et m\'ethodes asymptotiques }  Dunod, Paris 1972
\bibitem[PM]{PM} J.Perez-Marco, {\it Convergence or generic divergence of the Birkhoff normal form}, Ann. of Math. {\bf 157}, 557-574 (2003)
\bibitem[Po1]{Po1} A.Popov, {\it Invariant Tori, Effective Stability, and Quasimodes with Exponentially Small Error Terms I - Birkhoff Normal Forms}, Ann.Henri Poincar\'e {\bf 1}, 223-248 (200);
\bibitem[Po2]{Po2} A.Popov, {\it Invariant Tori, Effective Stability, and Quasimodes with Exponentially Small Error Terms II - Quantum Birkhoff Normal Forms}, Ann.Henri Poincar\'e {\bf 1}, 223-248 (2000);
\bibitem[Ro]{Ro}D.Robert, {\it Autour de l'approximation s\'emiclassique}, Birkh\"auser 1987
\bibitem[Ru]{Ru}H.R\"ussmann, {\it \"Uber die Normalform analytischer Hamiltonscher Differentialgleichungen in der N\"ahe einer gleichgewichtl\"osung}, Math.Annalen {\bf 169}, 55-72 (1967)
\bibitem[Sj]{Sj} J.Sj\"ostrand, {\it Semi-excited levels in non-degenerate potential wells}, Asymptotic Analysis {\bf 6}, 29-43 (1992)
\bibitem[SM]{SM} C.L.Siegel and J.Moser, {\it Lectures in Celestial Mechanics}, Classics in Mathematics Vol. Springer-Verlag 1971
\bibitem[St]{St} L.Stolovich, {\it Progress in normal form theory},  Nonlinearity {\bf 22}, 7423-7450 (2009)
\bibitem[Vo]{Vo} A.Voros, {\it Exercises in Exact Quantization},  J.Phys.A  {\bf 33}, 7423-7450 (2000)
\bibitem[Zu]{Zu} N.T.Zung,  {\it Convergence Versus Integrability in Normal Form Theory}, Ann. of Math.{\bf 161}, 141-156  (2005)
\end{thebibliography}
\end{document}